\documentclass[preprint,12pt,number]{elsarticle}
\usepackage{multirow}
\usepackage{subcaption}
\usepackage[normalem]{ulem} 
\usepackage{xcolor}
\usepackage{amssymb}
\usepackage{mydef}
\usepackage[colorlinks=true, linkcolor=blue, citecolor=blue, urlcolor=blue]{hyperref}
\newtheorem{proposition}{Proposition}
\newproof{proof}{Proof}
\newdefinition{remark}{Remark}

\usepackage{amssymb}
\usepackage{amsmath}

\usepackage{comment}

\journal{Journal of Computational Physics}

\begin{document}

\begin{frontmatter}



\title{An anisotropic nonlinear stabilization for finite element approximation of Vlasov-Poisson equations} 


 \author[1]{Junjie Wen\corref{cor1}}
 \ead{junjie.wen@it.uu.se}

\author[1]{Murtazo Nazarov}
\ead{murtazo.nazarov@uu.se}

 \cortext[cor1]{Corresponding author}
\address[1]{Department of Information Technology, Division of Scientific Computing, Uppsala University, Sweden}

\begin{abstract}
We introduce a high-order finite element method for approximating the Vlasov-Poisson equations. This approach employs continuous Lagrange polynomials in space and explicit Runge-Kutta schemes for time discretization. To stabilize the numerical oscillations inherent in the scheme, a new anisotropic nonlinear artificial viscosity method is introduced. Numerical results demonstrate that this method achieves optimal convergence order with respect to both the polynomial space and time integration. The method is validated using classic benchmark problems for the Vlasov-Poisson equations, including Landau damping, two-stream instability, and bump-on-tail instability in a two-dimensional phase space.
\end{abstract}




\end{frontmatter}



\section{Introduction}
\label{sec1}
Plasma, one of the fundamental states of matter, is a crucial element in many academic and industrial fields. However, it is usually difficult to predict the motion of the plasma in practical research; this difficulty is mainly caused by the fluid nature of plasma and its charged state. The simulation of fluid plasma can be considered as a computational fluid dynamics (CFD) problem, which is known to be complicated. In addition, modeling plasma becomes even more challenging when the effects of electromagnetic fields are taken into consideration. To simulate plasma, the kinetic model is widely applied, where the distribution of charged particles in the phase space is quantified by a distribution function $f(\pmb{x}, \pmb{v}, t)$. This distribution function is determined by the position $\pmb{x}$, the velocity $\pmb{v}$, and the time $t$. The time evolution of $f(\pmb{x}, \pmb{v}, t)$ under self-consistent electromagnetic fields is described by the Vlasov equation. To estimate the electromagnetic fields in the system, one approach is disregarding the magnetic field $\pmb{B}$ and computing the electric field $\pmb{E}$ by solving the Poisson equation, which gives the name Vlasov-Poisson to the system.

The history of numerical analysis for Vlasov-Poisson equations dates back to the 1960s, see \cite{TAYLOR1967155}. In general, the numerical schemes for PDEs that describe particle propagation such as Vlasov-Poisson equations can be classified into particle-based and grid-based methods. The Particle-in-Cell (PIC) method is one of the most popular particle-based methods. Many attempts have been made to apply the PIC method to Vlasov-Poisson equations, and some of them have achieved satisfactory results, see \cite{MR1119268,DEGOND20105630,MR2818633,MR3485969,MR3421063,MR4456467} for instance. The PIC method makes use of so-called super nodes, and these super nodes are used to approximate the distribution of particles. The trajectories of these super nodes in the phase space are governed by the Vlasov equation, and the electric field is computed on a grid of physical space. The advantage that drives the application of the PIC method is its relatively low computational cost compared to other numerical schemes, especially when it is applied to high-dimensional problems including Vlasov-Poisson. Nevertheless, the drawback of the PIC method is that applying super nodes can introduce a numerical error, and this error is decreased by $1/\sqrt{N}$ when the number of super nodes increases by $N$, according to 
the study \citep{MR4456467} and references therein.

Unlike particle-based methods, which track the trajectories of the super nodes, grid-based methods operate on a grid of the phase space and discretize the PDE by evaluating the function values at grid nodes. Popular grid-based methods for Vlasov-Poisson equations include finite volume (FV) and semi-Lagrangian schemes. The FV scheme was proven feasible for Vlasov-Poisson equations in \cite{MR1870837}, where a first-order Vlasov-Poisson approximation is introduced. Subsequently, researchers have attempted to achieve higher-order and stabilized FV approximations with proper reconstruction algorithms and stabilization techniques; relevant studies can be found in \cite{FILBET2001166,5518448,MR3854174}. The semi-Lagrangian method moves function data on the grid of phase space based on the characteristics of the Vlasov equation. This approach is unconditionally stable and allows large time steps, as introduced in \cite{MR1672731}. The semi-Lagrangian method is often used in combination with other numerical frameworks, such as discontinuous Galerkin (DG), for instance \cite{MR2843721,MR2806222}. In addition to the approaches mentioned above, some Vlasov-Poisson solvers employing other numerical schemes such as finite difference (FD) and spectral methods have also been proposed, see \cite{KLIMAS1987202,MR1259903,MR1977366} and references therein for instance.\par

The finite element (FE) scheme is a robust numerical method that divides the computation domain into a finite number of subdomains and approximates unknown functions within each subdomain using a set of basis functions. This method is known to be able to handle complex geometries and boundary conditions, and can achieve high accuracy through mesh refinement. The application of the continuous FE method to Vlasov-Poisson equations has not been extensively studied yet. A FE approximation of the Vlasov-Poisson equation was published in 1988; however, there has been a notable lack of subsequent research building on this foundation in recent decades. The authors of the 1988 paper verified some properties of the FE approximations without exploring the high accuracy potential of the FE method, see more details in \cite{ZAKI1988184, ZAKI1988200}. Additionally, it is well-known that the FE method introduces numerical oscillations when it is applied to conservation laws, including Vlasov-Poisson equations. Notably, this issue was not addressed in that paper. 

The residual-based artificial viscosity method serves as an effective stabilization technique for FE approximations. This method enhances stability by incorporating a viscosity term based on the PDE residual. It was proved in \cite{MR3011445} that the FE approximation of scalar conservation laws converges to the unique entropy solution when using the residual-based artificial viscosity. The FE method with the residual-based artificial viscosity has been applied to many models of CFD and computational plasma, such as Navier–Stokes equations and Magnetohydrodynamics, e.g., \cite{MR4456197, MR4496880}.\par

In this article, we introduce a continuous FE approximation of the Vlasov-Poisson equation; this approximation utilizes high-order Lagrange polynomials in space and is explicit in time. The oscillations in the approximation are addressed by the residual-based artificial viscosity scheme and the numerical results prove that this approach can stabilize the FE approximation without degrading its accuracy. The Poisson equation is solved on a grid of the physical space using the continuous FE method. We implement our methods on the DOLFINx platform \cite{barrata2023dolfinx}, an open-source FE library that includes parallelization packages to speed up the calculations. We would like to clarify that the implementation of the methods to high-dimensional Vlasov-Poisson is not included in the article.

This article is organized as follows: In Section \ref{sec:prelim}, all the equations in the system are introduced together with some conservation properties of Vlasov-Poisson equations. In Section \ref{sec:FEA}, important settings for the FE schemes such as the meshes, the elements, and the function spaces are defined. In Section \ref{sec:RVM}, we introduce a nonlinear finite element framework for Vlasov-Poisson equations using a residual-based artificial viscosity method. 
The results of numerical experiments are shown in Section \ref{sec:experiment}. In addition, we apply our methods to the guiding-center model, which serves as a variant of Vlasov-Poisson, and the results are presented in \ref{Ap:GCM}.

\section{Preliminaries}\label{sec:prelim}
In this section, we introduce the equations that model the distribution and propagation of plasma. We also introduce some conservation properties of the Vlasov-Poisson equation.
\subsection{Governing equations}
Let $T>0$ be the final time, then the temporal domain be $\left[0,T\right]$. Consider two open polyhedral domains $\Omega_{\pmb{x}}\subset\mathbb{R}^d$, $\Omega_{\pmb{v}}\subset\mathbb{R}^d$, where $d$ is the dimension, $d = 1,2,3$. For all $(\pmb{x},\pmb{v},t)\in\Omega\times[0,T]$, where $\Omega:=\Omega_{\pmb{x}}\times\Omega_{\pmb{v}}$, we denote the distribution function as $f(\pmb{x},\pmb{v},t)$. The normalized Vlasov-Poisson equation with appropriate boundary conditions and initial data is given by
\begin{equation}
  \partial_t f+\pmb{\bbetaa}\cdot\nabla f = 0, \label{eq:vp}
\end{equation}
where the nonlinear vector field $\pmb{\bbetaa}$ and the operator $\nabla$ are defined as follows
\begin{equation}
  \pmb{\bbetaa}:=\begin{pmatrix}\pmb{v} \\ \pmb{E} \end{pmatrix},\qquad \nabla:=\begin{pmatrix}\nabla_{\pmb{x}} \\ \nabla_{\pmb{v}} \end{pmatrix}.\nonumber
\end{equation}
Let $\Phi(\pmb{x},t)$ be the electric potential of the system, which is associated with the electric field $\pmb{E}(\pmb{x},t)$, according to the Gauss's law:
\begin{equation}
  \pmb{E} = -\nabla_{\pmb{x}}\Phi.\label{eq:gauss}
\end{equation}
The function $\Phi$ can be computed by solving the following Poisson equation
\begin{equation}
    -\nabla^2_{\pmb{x}}\Phi=\rho-\rho_0,\label{eq:poiss}
\end{equation}
where $\rho(\pmb{x},t)$ is the charge density, which can be derived by the distribution function as follows
\begin{equation}
    \rho = \int_{\Omega_{\pmb{v}}}f\ {\rm d}\pmb{v},\label{eq:charge}
\end{equation}
and the constant $\rho_0$ is the average density such that
\begin{equation}
  \int_{\Omega_{\pmb{x}}}(\rho-\rho_0)\ {\rm d}\pmb{x}=0.\notag
\end{equation} 

\begin{proposition}
    [Conservation]
\label{prop:cons}
The following conservation properties hold for the solution of Vlasov-Poisson equations:
\begin{enumerate}
\item Conservation of mass:
$
\frac{\rm d}{{\rm d}t}\int_{\Omega}f\ {\rm d}\pmb{x}{\rm d}\pmb{v}=0.
$

\item Conservation of momentum:
$
\frac{\rm d}{{\rm d}t}\int_{\Omega}f\pmb{v}\ {\rm d}\pmb{x}{\rm d}\pmb{v}=0.
$

\item Conservation of energy:
$
\frac{\rm d}{{\rm d}t}\left(\frac{1}{2}\int_{\Omega}f \pmb{v}^2\ {\rm d}\pmb{x}{\rm d}\pmb{v}+\frac{1}{2}\int_{\Omega_{\pmb{x}}}\pmb{E}^2\ {\rm d}\pmb{x}\right)=0.
$

\item Conservation of square mass:
$
\frac{\rm d}{{\rm d}t}\int_{\Omega}f^2\ {\rm d}\pmb{x}{\rm d}\pmb{v}=0.
$
\end{enumerate}
\end{proposition}

\begin{proof}
    The mass conservation is trivial, which is obtained by integrating the equation \eqref{eq:vp}  in $\Omega$. The momentum, energy and square mass conservation are obtained by multiplying \eqref{eq:vp} by $\pmb{v}$, $\pmb{v}^2$ and $f$ respectively and integrating in $\Omega$.
\end{proof}

\section{Finite element approximations}\label{sec:FEA}
In this section, we present the quadrangulation  of the domain, the construction of finite element spaces, and some notations that appear in this article.

We construct the grid $\calT_{\pmb{x}}$ on $\Omega_{\pmb{x}}$ with the spacing $\Delta x_i$, $i=1,\ldots,d$. Let $\pmb{x}=(x_1,\ldots,x_d)^{\mathsf{T}}$ be a set of coordinates and  $\pmb{\alpha}=(\alpha_1,\ldots,\alpha_d)^{\mathsf{T}}$ be the corresponding multi-index, the polynomial space $\polQ_k$ is defined as follows
\begin{equation}
	\polQ_k={\rm span}\left\{\prod^d_{i=1}x_i^{\alpha_i}=\pmb{x}^{\pmb{\alpha}}:0\le\alpha_i\le k\ {\rm for}\ i=1,\ldots,d\right\},\notag
\end{equation}
and for the mesh $\calT_{\pmb{x}}$ we associate the following continuous approximation space
\begin{equation}
	\calV_{\pmb{x}} := \{w(\pmb{x}): w\in \calC^0(\overline\Omega_{\pmb{x}}),w|_K\in \polQ_k,\forall K\in \calT_{\pmb{x}}\},\notag
\end{equation}
where $\overline\Omega_{\pmb{x}}$ denotes the closure of $\Omega_{\pmb{x}}$. We require that the distribution of the nodes in $\calV_{\pmb{x}}$ is uniform, and denote by $ N_{\pmb{x}}$ the number of the nodes. Let $\phi_i$ be the Lagrange basis function in $\calV_{\pmb{x}}$ whose value is 1 at the $i$-th node and 0 at any other nodes, for $i = 1,\ldots, N_{\pmb{x}} $.

We build the mesh $\calT_{\pmb{v}}$ and the function space $\calV_{\pmb{v}}$ on $\Omega_{\pmb{v}}$ in the same way as described above, and adopt the above definitions to $\Omega_{\pmb{v}}$. In order to facilitate the distinction, we replace the subscript with $\pmb{v}$ and use the notation $\varphi$ for the basis functions in $\calV_{\pmb{v}}$.  Since $\Omega:=\Omega_{\pmb{x}}\times\Omega_{\pmb{v}}$, there exists a grid mesh $\calT_h:=\calT_{\pmb{x}}\times\calT_{\pmb{v}}$ and a function space $\calV_h:=\calV_{\pmb{x}}\times\calV_{\pmb{v}}$ on $\calT_h$; we also know that $\calV_h$ should have $N:= N_{\pmb{x}} N_{\pmb{v}}$ nodes. We denote the basis function of $\calV_h$ by $\psi$ and define $\calI(i)$ as the set of all nodes within the support of $\psi_i$, $i=1,\ldots, N$.
\begin{remark}
  Since $\calT_h=\calT_{\pmb{x}}\times\calT_{\pmb{v}}$ and $\calV_h=\calV_{\pmb{x}}\times\calV_{\pmb{v}}$, for any basis function $\psi\in\calV_h$, there exists a pair $(\phi,\varphi)$, $\phi\in\calV_{\pmb{x}}$ and $\varphi\in\calV_{\pmb{v}}$, such that $\psi = \phi\varphi$. In order to make a understandable notation, we require that the basis functions in $\calV_h$ are indexed such that $\psi_l=\phi_i\varphi_j$, where $l = (i-1) N_{\pmb{v}}+j$, $i=1,\ldots, N_{\pmb{x}}$, $j=1,\ldots, N_{\pmb{v}}$.
\end{remark}

We use the inner product in the $L^2$-spaces and denote the norm as follows
\[
\big( \psi_i, \psi_j \big):=\sum_{K\in \calT}\int_K\psi_i\psi_j\ {\rm d}\pmb{x}{\rm d}\pmb{v},\qquad\Vert\psi_i\Vert^2=(\psi_i,\psi_i).
\]

\subsection{Finite Element Method for Vlasov-Poisson}\label{sec:RVM}

We now introduce Galerkin finite element methods for the Vlasov-Poisson equation \eqref{eq:vp}. This approximation is given by: find $f_h(t)\in \mathcal{C}^1(\mathbb{R}^+;\calV_h)$ such that
\begin{equation}\label{eq:GFEM}
  (\partial_t f_h+\pmb{\bbetaa}_{h}\cdot\nabla f_h, w)=0,\qquad\forall w\in\calV_h,
\end{equation}
where $\pmb{\bbetaa}_h = (\pmb{v},-\nabla_{\pmb{x}}\Phi_h)^{\mathsf{T}}$. We compute $\Phi_h$ by solving the Poisson equation.

However, it is commonly agreed that such Galerkin finite element approximations produce oscillations. The reason for this issue is that the term $\nabla f_h$ when approximated by the finite element method is equivalent to the central difference scheme, which is not stable. We introduce an artificial viscosity scheme to stabilize the finite element approximation, ensuring sufficient stabilization without compromising the method's high accuracy. 

In the following section, for simplicity, we focus on the Forward Euler discretization in time. However, extending the method to higher-order time integration is straightforward using high-order Runge-Kutta schemes.

\subsection{Anisotropic viscosity}
The physical and velocity spaces are often discretized using different mesh resolutions, resulting in an anisotropic mesh structure. The isotropic artificial viscosity methods discussed in the literature, such as \citep{MR2427085, GUERMOND20114248, GuermondNaPo11, MR3011445, Nazarov_Hoffman_2013, MR3612753, MR4754161}, do not provide sufficient stabilization for the numerical approximation. This highlights the need for anisotropic stabilization, which introduces the appropriate amount of viscosity in both the velocity and physical spaces.

Let us define the anisotropic viscosity matrix as 
$
\pmb{\polA}^\textrm{L} := \textrm{diag}(\pmb{\varepsilon}^\textrm{L}_{\pmb{x}}, \pmb{\varepsilon}^\textrm{L}_{\pmb{v}}),
$
where $\pmb{\varepsilon}^\textrm{L}_{\pmb{x}} := \textrm{diag}(\varepsilon^\textrm{L}_{x_1},\ldots, \varepsilon^\textrm{L}_{x_d})$ and $\pmb{\varepsilon}^\textrm{L}_{\pmb{v}} := \textrm{diag}(\varepsilon^\textrm{L}_{v_1},\ldots, \varepsilon^\textrm{L}_{v_d})$, $\varepsilon^\textrm{L}_{x_i}\in\calV_{h}$, $\varepsilon^\textrm{L}_{v_i}\in\calV_{h}$, $i=1,\ldots,d$, are componentwise artificial viscosity coefficients and are constructed as the follows:
\begin{equation}\label{eq:visc:FO}
  \begin{aligned}
    \varepsilon_{x_j}^{\textrm{L},i} := \frac12 \frac{\Delta x_j}{k}\Vert v_j\Vert_{L^{\infty}\calI(i)},\quad 
    \varepsilon_{v_j}^{\textrm{L},i} := \frac12 \frac{\Delta v_j}{k}\Vert E_j\Vert_{L^{\infty}\calI(i)}, 
    \end{aligned}
  \end{equation}
for $i=1,\ldots,N$, $j=1,\ldots,d$, where $\Delta x_j$ and $\Delta v_j$ are mesh-sizes corresponding to $j$-th coordinate of $\bx$ and $\bv$ respectively. 

Let $f_h^n$ and $\pmb{\bbetaa}_h^n$ be the finite element approximations of $f$ and $\pmb{\bbetaa}$ at time $t^n$ respectively.
Denoting $f_h(\bx, \bv, t_n) = \sum_{j=1}^{N} f^n_{j} \psi_j(\bx, \bv)$, and discretizing the time derivative of \eqref{eq:GFEM} using Forward Euler method, we obtain the following fully discrete system:
\begin{equation}\label{eq:FO}
  \sum_{j=1}^{N} m_{ij}\frac{f^{n+1}_{j} - f^n_{j}}{\Delta t} 
  + 
  \sum_{j=1}^{N} c_{ij}(\pmb{\bbetaa}^n_h) f^n_{j} 
  +
  \sum_{j=1}^{N} d_{ij}^{\textrm{L},n} f^n_{j} = 0,
\end{equation} 
where 
$
m_{ij} := (\psi_j, \psi_i)
$
is the mass matrix, 
$
c_{ij}(\pmb{\bbetaa}^n_h):= (\pmb{\bbetaa}^n_h \cdot \nabla \psi_j, \psi_i)
$ is the advection matrix,
and 
$
d_{ij}^{\textrm{L},n} := (\pmb{\polA}^\textrm{L} \GRAD \psi_j, \GRAD \psi_i)
$, $\forall \psi_i,\psi_j\in\calV_h$. The subscript "L" is used for the viscosity matrix, since the artificial viscosity coefficients of \eqref{eq:visc:FO} are only first order that can also be seen later in this section.

The time-step size in \eqref{eq:FO} is controlled using the following CFL condition:
\begin{equation}\label{eq:CFL}
  \Delta t = {\rm CFL}\frac{\big(\sum_{i=1}^d ((\Delta x_i)^2+(\Delta v_i)^2)\big)^\frac12}{k\Vert\pmb{\bbetaa}_{h}^n\Vert_{L^\infty(\Omega)}}.
\end{equation}

The following stability result holds for the first-order forward Euler scheme \eqref{eq:FO}.

\begin{proposition}
The forward Euler scheme \eqref{eq:FO} is stable if the viscosity coefficients are computed from \eqref{eq:visc:FO} and the CFL condition \eqref{eq:CFL} is satisfied.
\end{proposition}

\begin{proof}
To simplify the discussion, let us focus on the proof for the one-dimensional case only. Consider \eqref{eq:vp} in one space dimension for physical and velocity spaces. We discretize the computational domain uniformly with quadrangular elements of size $h_1:=\Delta x$ in the $x$-direction and $h_2:=\Delta v$ in the $v$-direction, as shown in Figure~\ref{fig:2D:example}, where the patch of the mesh around the node $\bx_{i,j}:=(x_i, v_j)$ is depicted. Let $\calV_h$ be a continuous piecewise linear polynomial space in $\polQ_1$ defined on this mesh. 
\begin{figure}[htbp]
\centering{
  \includegraphics[width=0.65\textwidth]{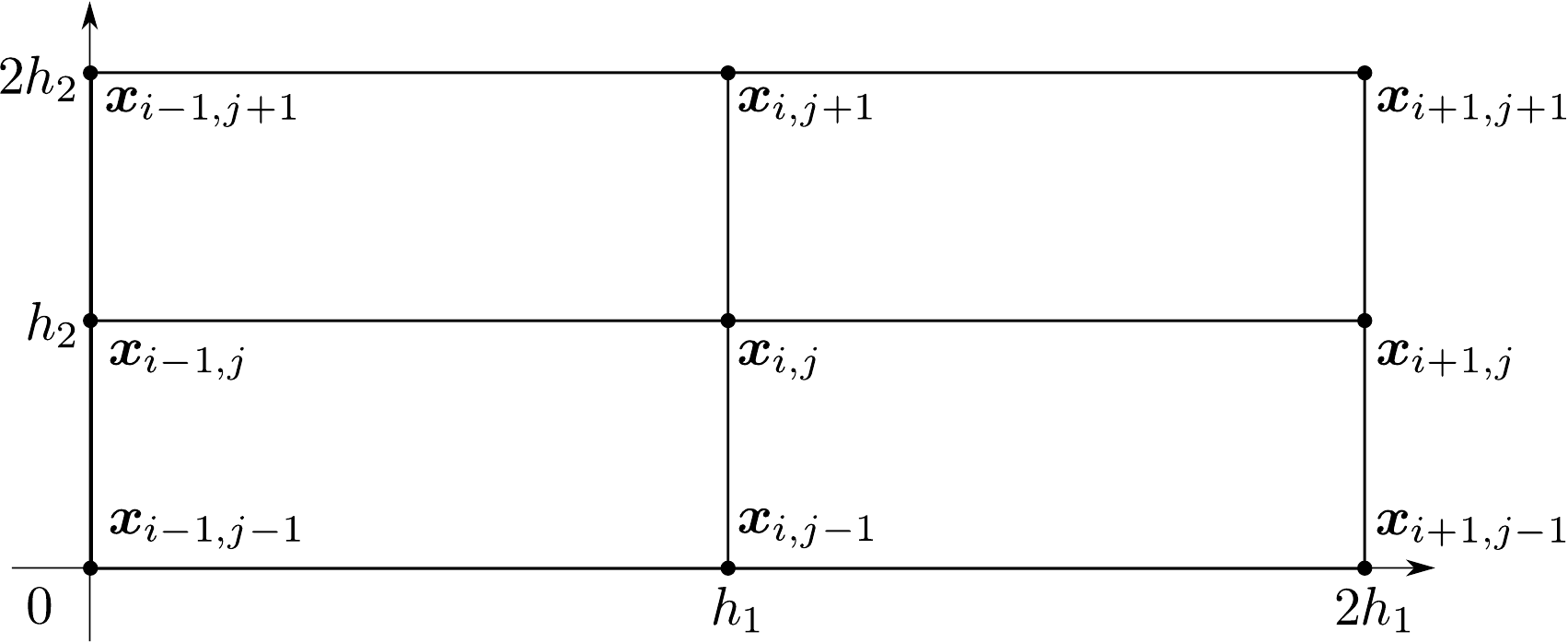}
}
  \caption{Quadrangulation of the domain around the point $\bx_{i,j}:=(x_i, v_j)$.
}\label{fig:2D:example}
\end{figure}

To ease the discussion, let us denote $\bbetaa = (\beta_1, \beta_2)^\top :=(v,E)^\top$, and let both $v\ge 0$ and $E\ge 0$. 
We calculate the lumped mass matrix corresponding to the node $\bx_{i,j}$ which is equal to $h_1 h_2$. In addition let $f^n_{i,j} \approx f(t_n, \bx_{i,j})$ be the approximation of the solution at time level $t_n$ and node $\bx_{i,j}$.

We obtain a system of linear equations, where the equation of this system that is corresponding to the node $\pmb{x}_{i,j}$ has the form
\begin{equation}\label{eq:ith}
  \begin{aligned}
  h_1h_2& \frac{f^{n+1}_{i,j} - f^n_{i,j}}{\Delta t} \\
  &+ 
    \frac1{12}
      \beta_1 h_2 \,
      \Big(
      \big(
      f^n_{i+1,j+1} - f^n_{i-1,j+1}
      \big)
      +
      4\big(
      f^n_{i+1,j} - f^n_{i-1,j}
      \big)
      +
      \big(
      f^n_{i+1,j-1} - f^n_{i-1,j-1}
      \big)
      \Big)\\
    &+ 
    \frac1{12}
      \beta_2 h_1 \,
       \Big(
       \big(
       f^n_{i+1,j+1} - f^n_{i+1,j-1}
       \big)
       +
       4\big(
       f^n_{i,j+1} - f^n_{i,j-1}
       \big)
       +
       \big(
       f^n_{i-1,j+1} - f^n_{i-1,j-1}
       \big)
       \Big) 
    \\
    &+
    \frac16
    \frac{h_2}{h_1}
    \e_x^{\textrm{L}} \,
      \big(
      -f^n_{i+1,j+1}+2 f^n_{i,j+1}-f^n_{i-1,j+1}-4 f^n_{i+1,j}\\
      & \hspace{1.7in}
      +8f^n_{i,j}-4f^n_{i-1,j}-f^n_{i+1,j-1}+2f^n_{i,j-1}-f^n_{i-1,j-1}
      \big)\\
      & +
      \frac16
      \frac{h_1}{h_2}
      \e_v^{\textrm{L}} \,
      \big(
      -f^n_{i+1,j+1}+2f^n_{i+1,j}-f^n_{i+1,j-1}-4f^n_{i,j+1}\\
      & \hspace{1.7in}
      +8f^n_{i,j}-4f^n_{i,j-1}-f^n_{i-1,j+1}+2f^n_{i-1,j}-f^n_{i-1,j-1}
      \big)\\
      &= 0.
  \end{aligned}
\end{equation}

From here one can see that when $\e_x^{\textrm{L}}=0$ and $\e_v^{\textrm{L}}=0$, we get central difference stencils in space, therefore the forward Euler scheme can be unstable. However setting
$\e_x^{\textrm{L}} = \frac12 h_1 \beta_1$
and
$\e_v^{\textrm{L}} = \frac12 h_2 \beta_2$,
and in the definition \eqref{eq:visc:FO}, 
gives us the following scheme
\begin{equation}
  \begin{aligned}
  h_1h_2& \frac{f^{n+1}_{i,j} - f^n_{i,j}}{\Delta t} \\
  &+ 
  \frac16 \beta_1 h_2 
    \Bigg(
    \big(
    f^n_{i,j-1} - f^n_{i-1,j-1}
    \big)
    +
    \big(
    f^n_{i,j+1} - f^n_{i-1,j+1}
    \big)
    + 
    4\big(
    f^n_{i,j}
    -
    f^n_{i-1,j}
    \big)
    \Bigg) \\
  &+
     \frac16 \beta_2 h_1 
     \Bigg(
     \big(
     f^n_{i-1,j} - f^n_{i-1,j-1}
     \big)
     +
     \big(
     f^n_{i+1,j} - f^n_{i+1,j-1}
     \big)
     + 
     4\big(
     f^n_{i,j}
     -
     f^n_{i,j-1}
     \big)
     \Bigg) \\
      &= 0,
  \end{aligned}
\end{equation}
which is a first-order upwind scheme on the given mesh-patch and is stable for the given time-step restriction in the proposition's statement. The proof is completed. 
\end{proof}

\begin{remark}[Isotropic viscosity]\label{remark:isotropic}
Let us assume that the viscosity matrix $\polA$ is constructed in an isotropic manner. Then, the mesh size is typically defined as the smallest edge, which in this example is $h := \min(h_1, h_2) = h_2$. Let us assume that $h_1 < 1$. The viscosity coefficients are then defined as $\e_x^{\textrm{L}} = \e_v^{\textrm{L}} = \frac{1}{2} h |\bbetaa|$, where $|\bbetaa| := (\beta_1^2 + \beta_2^2)^{\frac{1}{2}}$, and the coefficients of the corresponding viscous terms of \eqref{eq:ith} become:
  \[
  \frac16 \frac{h_2}{h_1}\e_x^{\textrm{L}}
  = \frac{1}{12 h_1} h_2^2 |\bbetaa|,
  \text{ and }
  \frac16 \frac{h_1}{h_2}\e_v^{\textrm{L}}
  = \frac{1}{12} h_1|\bbetaa|.
  \]
Now, if we assume that the velocity field is defined as $\bbetaa = (1, 0)$, and that the anisotropy ratio is sufficiently large so that $h_2 \ll h_1$, which implies $h_2^2/h_1 \ll 1$, we obtain the following relation from \eqref{eq:ith}
\begin{equation*}
  \begin{aligned}
  & h_1h_2\frac{f^{n+1}_{i,j} - f^n_{i,j}}{\Delta t} \\
  &+ 
    \frac1{12}
      h_2 \,
      \Big(
      \big(
      f^n_{i+1,j+1} - f^n_{i-1,j+1}
      \big)
      +
      4\big(
      f^n_{i+1,j} - f^n_{i-1,j}
      \big)
      +
      \big(
      f^n_{i+1,j-1} - f^n_{i-1,j-1}
      \big)
      \Big)\\
      & +
      \frac{1}{12}
      h_1\,
      \big(
      (-f^n_{i+1,j+1}+2f^n_{i+1,j}-f^n_{i+1,j-1}) \\
      & \hspace{0.8in}
      +4 (-f^n_{i,j+1}+2f^n_{i,j}-f^n_{i,j-1}) \\
      & \hspace{1.2in}
      +(-f^n_{i-1,j+1}+2f^n_{i-1,j}-f^n_{i-1,j-1})
      \big)
      = 0.
  \end{aligned}
\end{equation*}

The row of the linear system that is corresponding to the node $\pmb{x}_{i,j}$ shows that, since the advection field is acting in the $x$-direction, there are three centered difference schemes on the edges parallel to the $x$-axis: $(\pmb{x}_{i-1, j+1},\pmb{x}_{i+1, j+1})$, $(\pmb{x}_{i-1, j}, \pmb{x}_{i+1, j})$, and $(\pmb{x}_{i-1, j-1}, \pmb{x}_{i+1, j-1})$. However, the amount of viscosity is active on the edges parallel to the $v$-axis: $(\pmb{x}_{i+1, j-1}, \pmb{x}_{i+1, j+1})$, $(\pmb{x}_{i, j-1}, \pmb{x}_{i, j+1})$, and $(\pmb{x}_{i-1, j-1}, \pmb{x}_{i-1, j+1})$, which cannot stabilize the resulting centered difference schemes. 
\end{remark}

\begin{proposition}\label{prop:conservation}
The forward Euler scheme \eqref{eq:FO} is conservative. 
\end{proposition}

\begin{proof}
First of all, we notice that the matrix $c_{ij}(\pmb{\bbetaa}^n_h)$ is skew-symmetric, which gives 
$
c_{ij}(\pmb{\bbetaa}^n_h) = - c_{ji}(\pmb{\bbetaa}^n_h),
$
and 
$
\sum_{j=1}^{N} c_{ij}(\pmb{\bbetaa}^n_h) = 
\sum_{i=1}^{N} c_{ji}(\pmb{\bbetaa}^n_h) = 0
$. This skew-symmetry of $c_{ij}$ comes from the fact that $\nabla \cdot \pmb{\bbetaa}^n_h \equiv 0$, for any $n=0,1,\ldots$.
Similarly, the partition of unity property $\sum_{j=1}^{N} \psi_j = 1$ gives that 
$
\sum_{j=1}^{N} d_{ij}^{\textrm{L},n} = \sum_{i=1}^{N} d_{ji}^{\textrm{L},n} = 0.
$

Summing the equation \eqref{eq:FO} for all $i=1,\ldots,N$, we obtain the following:
\begin{equation*}
  \sum_{i=1}^{N} \sum_{j=1}^{N} m_{ij}\frac{f^{n+1}_{j} - f^n_{j}}{\Delta t} 
  + 
  \sum_{i=1}^{N} \sum_{j=1}^{N} c_{ij}(\pmb{\bbetaa}^n_h) f^n_{j} 
  +
  \sum_{i=1}^{N} \sum_{j=1}^{N} d_{ij}^{\textrm{L},n} f^n_{j} = 0,
\end{equation*} 
which is the same as 
\begin{equation*}
  \sum_{i=1}^{N} \frac{f^{n+1}_{i} - f^n_{i}}{\Delta t} \sum_{j=1}^{N} m_{ij} 
  + 
  \sum_{j=1}^{N} f^n_{j} \sum_{i=1}^{N} c_{ij}(\pmb{\bbetaa}^n_h)  
  +
  \sum_{j=1}^{N}  f^n_{j} \sum_{i=1}^{N} d_{ij}^{\textrm{L},n} = 0.
\end{equation*} 
Again using partition of unity and properties of the advection and diffusion matrices we obtain
\[
\sum_{i=1}^{N} m_i f^{n+1}_{i} = \sum_{i=1}^{N} m_i f^{n}_{i},
\]
which is the mass conservation. Here $m_i := \sum_{j=1}^{N} m_{ij}$ is the lumped mass matrix. 
\end{proof}

\subsection{Nonlinear viscosity}
The upwind scheme \eqref{eq:FO}, which is basically the same as  a well-known Lax-Friedrichs scheme, known for its robustness, stabilizes finite element approximations and mitigates numerical oscillations; however, it is limited to first-order in time and space. One can use higher-order time integration to achieve a higher-order method in time. However, the viscosity coefficients in \eqref{eq:visc:FO} restrict the method to first-order accuracy in space. The convergence order of this scheme can be improved by constructing the viscosity based on the residual of equation \eqref{eq:vp}, we refer to such methods as {\em residual viscosity} or RV methods in short; and this method is expected to be high-order in space, especially when we use high-order polynomials. The residual is effective for detecting discontinuities and strong shocks, and we use it as an indicator to activate the first-order viscosity when the solution is not smooth. Moreover, the artificial viscosity should decrease along with the residual in smooth regions.\par
We now introduce the process to construction the residual viscosity. First, we show how we compute the finite element residual of equation \eqref{eq:vp}. We denote by BDFs($f_h)^n$ $s$-th order backward difference approximation of the time derivative $\partial_t f_h$ at time $t^n$, see for example \cite[Sec.3.2, Appendix A.]{stiernstrom2021residual}.
Let us seek the residual through the following $L^2$-projection: find $ R_h(f^n_h) \in \calV_h$ such that 
\begin{equation}\label{eq:residual}
\begin{aligned}
  \big( R_h,  \psi \big)
  & + 
  \sum_{K\in\calT_h} \sum_{j=1,\ldots,d}\Big( \frac{(\Delta x_j)^2}{k}  
  \partial_{x_j} R_h,\partial_{x_j}\psi \Big)_K \\
  & \hspace{0.4in}+ 
  \sum_{K\in\calT_h} \sum_{j=1,\ldots,d} \Big(\frac{(\Delta v_j)^2}{k}
  \partial_{v_j} R_h,\partial_{v_j}\psi \Big)_K \\
  & \hspace{0.8in} =  \big(\vert{\rm BDFs}(f_h)^n+\pmb{\bbetaa}_{h}^n\cdot\nabla f_h^n\vert, w \big), \quad \forall w \in\calV_h.
  \end{aligned}
\end{equation}
Note that the second and third terms in the above projection problem is elliptic smoothing of the residual, which was first used in the computation of the residual in \cite[Eq. (3.21)]{Lundgren_2024}.

Once the finite element residual is approximated, we construct the following high-order anisotropic viscosity tensor 
$
\pmb{\polA}^\textrm{H} := \textrm{diag}(\pmb{\varepsilon}^\textrm{H}_{\pmb{x}}, \pmb{\varepsilon}^\textrm{H}_{\pmb{v}})
$,
where the viscosity coefficients are defined as 
\begin{equation}\label{eq:rv:visc}
  \begin{aligned}
    \varepsilon_{x_j}^{\textrm{H},i} = \min\left(\varepsilon_{x_j}^{\textrm{L},i}, \left(\frac{\Delta x_j}{k}\right)^2 \frac{R_i}{n_i}\right), \quad 
    \varepsilon_{v_j}^{\textrm{H},i} = 
    \min\left(\varepsilon_{v_j}^{\textrm{L},i}, \left(\frac{\Delta v_j}{k}\right)^2 \frac{R_i}{n_i} \right),
     \end{aligned}
\end{equation}
for every node $i=1,\ldots,N$ and $j = 1, \dots, d$, where $n_i$ represents the nodal values of the normalization function $n_h(f^n_h)$, defined below, we proceed as follows. The normalization function is used to remove the scale of $f_h$ from the residual. Various normalization functions have been presented in the literature where entropy and residual-based stabilization were used. We adapt the normalization technique introduced in \cite{MR4456197} and define
\[
  n_i := (1 - 0.5 \alpha_i) \Big \| f^n_h - \overline{f^n_h} \Big \|, \qquad i=1,\ldots,N,
\]
where $\overline{f^n_h} = \frac{1}{|\Omega|} \int_{\Omega} f^n_h \ud \bx \ud \bv$ is the average of the solution in the domain and $\alpha_i$ is the discontinuity indicator on the patch $\calI(i)$ which is defined as
\[
\alpha_i = \frac{\max_{j \in \calI(i)} f^n_j - \min_{j \in \calI(i)} f^n_j}
{\max_{j\in\{1,\ldots,N\}} f^n_j - \min_{j\in\{1,\ldots,N\}} f^n_j}.
\]
Note, that $\alpha_i$ is zero when solution is constant on the patch. In practice, to avoid division by zero, one can normalize the residual 
$
\frac{R_i}{n_i} \approx R_i \frac{n_i}{n_i^2  + \epsilon},
$ 
with small safety factor $\epsilon = 10^{-14}$.

Finally, upon defining 
$
d_{ij}^{\textrm{H},n} := (\pmb{\polA}^\textrm{H} \GRAD \psi_j, \GRAD \psi_i)
$, $\forall \psi_i,\psi_j\in\calV_h$, we obtain the following high-order in space method:
\begin{equation}\label{eq:HO}
  \sum_{j=1}^{N} m_{ij}\frac{f^{n+1}_{j} - f^n_{j}}{\Delta t} 
  + 
  \sum_{j=1}^{N} c_{ij}(\pmb{\bbetaa}^n_h) f^n_{j} 
  +
  \sum_{j=1}^{N} d_{ij}^{\textrm{H},n} f^n_{j} = 0.
\end{equation} 

As mentioned earlier, the method can be made high-order in time using Runge-Kutta schemes. In this paper, we use the five stages fourth-order explicit strong stability preserving (SSP) Runge-Kutta method from \cite{Kraaijevanger_1991} in all our numerical simulations, with the time step controlled by the CFL condition \eqref{eq:CFL}.

For the smooth solution, the accuracy of the residual computation of \eqref{eq:residual} is $R_h \approx \calO(\Delta t^s, \Delta x^{k+1}, \Delta v^{k+1})$, which gives the accuracy of the viscosity coefficients to be 
$\varepsilon_{\pmb{x}}^{\textrm{H}} \approx \Delta x^2 \, \calO(\Delta t^s, \Delta x^{k+1}, \Delta v^{k+1})$
and
$\varepsilon_{\pmb{v}}^{\textrm{H}} \approx \Delta v^2 \, \calO(\Delta t^s, \Delta x^{k+1}, \Delta v^{k+1})$. 
Since the highest polynomial space used in this paper is $\polQ_3$, we aim to achieve a fourth-order convergence rate in both time and space for smooth initial data. 
Therefore, the second-order Backward Differentiation Formula
\[
  {\rm BDF}2(f_h)^n = \frac{3f_h^n-4f_h^{n-1}+f_h^{n-2}}{2\Delta t}.
\]
is sufficient to ensure that the viscosity coefficients in \eqref{eq:rv:visc} remain 4-th order accurate, thereby preserving the desired accuracy. The time derivative of $f_h$ in the residual is approximated using BDF2 in all numerical simulations presented later in this paper.

For discontinuous initial data, the residual has a jump of order $\calO(\Delta x^{-1})$ and/or $\calO(\Delta v^{-1})$. Consequently, the first-order viscosity term becomes active in this situation, resulting in a first-order accuracy of the scheme.

\subsection{Computation of the charge density} 
Let the function $\rho_h\in\calV_{\pmb{x}}$ approximate the charge density $\rho$, then this function can be expressed as a linear combination of the basis functions in $\calV_{\pmb{x}}$ as follows
\begin{equation}
  \rho_h(t, \pmb{x}) =\sum_{i=1}^{N_{\pmb{x}}} \phi_i(\pmb{x}) \rho_i(t), \label{eq:charge1}
\end{equation}
where  $\{\rho_i\}_{i=1}^{N_{\pmb{x}}} $ are the nodal values of $\rho_h$. Considering that $f_h\in\calV_h$, we also apply the following representation to $f_h$
\begin{equation*}
  f_h(t, \pmb{x}, \pmb{v})=\sum_{l=1}^N \psi_l(\pmb{x}, \pmb{v}) f_l(t).
\end{equation*}
According to equation \eqref{eq:charge}, we calculate $\rho_h$ as follows
\begin{equation}\label{eq:charge2}
  \begin{aligned}
    \rho_h(t, \pmb{x})
    &=
    \int_{\Omega_{{\pmb{v}}}}f_h(t, \pmb{x}, \pmb{v})\ {\rm d}{\pmb{v}}
    =\int_{\Omega_{{\pmb{v}}}} 
    \sum_{l=1}^{N} \psi_i(\pmb{x}, \pmb{v}) f_l(t)\ {\rm d}{\pmb{v}}\\
    &=
    \int_{\Omega_{\pmb{v}}}
    \sum_{i=1}^{N_{\pmb{x}}}
    \sum_{j=1}^{N_{\pmb{v}}}\phi_i(\pmb{x})\varphi_j(\pmb{v}) f_{(i-1) N_{\pmb{v}}+j}(t)\ {\rm d}{\pmb{v}}\\
    &=
    \sum_{i=1}^{N_{\pmb{x}}} 
    \Big(
    \phi_i(\pmb{x})\int_{\Omega_{\pmb{v}}}
    \sum_{j=1}^{N_{\pmb{v}}}\varphi_j(\pmb{v}) f_{(i-1) N_{\pmb{v}}+j}(t)\ {\rm d}{\pmb{v}}
    \Big).
  \end{aligned}
\end{equation}
By comparing \eqref{eq:charge1} and \eqref{eq:charge2}, we conclude that the nodal values of $\rho_h$ can be computed as follows
\begin{equation}\label{eq:charge3}
\begin{aligned}
  \rho_i(t) 
  &= 
  \int_{\Omega_{\pmb{v}}}
  \sum_{j=1}^{N_{\pmb{v}}}\varphi_j(\pmb{v}) f_{(i-1) N_{\pmb{v}}+j}(t)\ {\rm d}{\pmb{v}} \\
  &= 
  \sum_{j=1}^{N_{\pmb{v}}}
  \Big(f_{(i-1) N_{\pmb{v}}+j}(t)\int_{\Omega_{\pmb{v}}}\varphi_j(\pmb{v})\ {\rm d}{\pmb{v}}
  \Big),\qquad\forall i\in\{1: N_{\pmb{x}}\}.
\end{aligned}
\end{equation}

The above computation is performed through cell iteration. Let us consider the case when $\Omega_x$ and $\Omega_v$ are one-dimensional. For any element $K$ in $\calT_h$, there exists one element in $\calT_x$ in the same column as $K$, and one element in $\calT_v$ in the same row as $K$, as shown in Figure \ref{fig:dofs}, we assign local index to the nodes in these elements. 
\begin{figure}[htbp]
  \centering
      \includegraphics[width=0.8\linewidth]{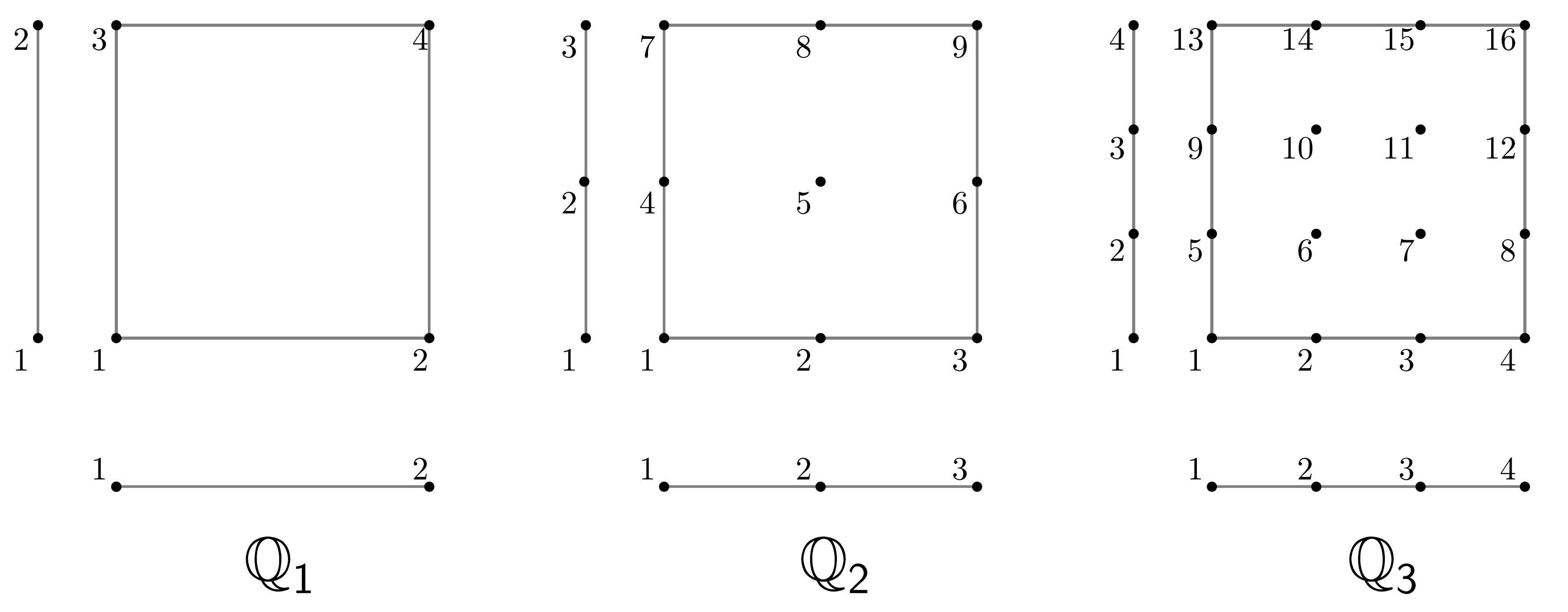}
      \caption{Local index for the nodes in $\calV_x$, $\calV_v$, and $\calV_h$, when the polynomial degree is 1, 2, and 3.}
  \label{fig:dofs}
\end{figure}

For the $\mathbb{Q}_1$ element in $\calV_h$, assume that the coordinates of the left lower corner and right upper corner are $(x_{\rm min},v_{\rm min})$ and $(x_{\rm max},v_{\rm max})$, respectively. For the node whose local index is $i \in \{1,2\}$, we denote by $\widehat{\varphi}_i$ and $\widehat{\rho}_i$ the basis function and the function value of $\rho_h$ at this node; let $\widehat{f_l}$, $l\in\{1,2,3,4\}$ be the nodal value of $f_h$ at the node with the local index $l$, see the left panel of Figure~\ref{fig:dofs}. The first basis function is $\widehat{\varphi}_1(v)=\frac{v_{\rm max}-v}{\Delta v}$, and we compute the integral of $\widehat{\varphi}_1$ in the $v$-direction within the element:
\begin{equation}
  \begin{aligned}
    \int_{v_{\rm min}}^{v_{\rm max}}\widehat{\varphi}_1(v)\ {\rm d}v&=\int_{v_{\rm min}}^{v_{\rm max}}\frac{v_{\rm max}-v}{\Delta v}\ {\rm d}v
    =\frac{\Delta v}{2},\label{eq:integ}
  \end{aligned}\notag
\end{equation}
by making use of the symmetry of the elements, we also know that $ \int_{v_{\rm min}}^{v_{\rm max}}\widehat{\varphi_2}(v)\ {\rm d}v=\frac{\Delta v}{2}$. Therefore, the charge density given by the particles within each single element is computed as follows 
\begin{equation}\label{eq:nodecharge}
  \left\{
  \begin{aligned}
    \widehat{\rho_1}=\Delta v\left(\frac{\widehat{f_1}}{2}+\frac{\widehat{f_3}}{2}\right),\\
    \widehat{\rho_2}=\Delta v\left(\frac{\widehat{f_2}}{2}+\frac{\widehat{f_4}}{2}\right).
  \end{aligned}
  \right.
 \end{equation}

Now, let us consider the $\polQ_2$ space. With the similar discussion above, we follow the indices depicted in the middle panel of Figure~\ref{fig:dofs} and get the following conclusions
for $\polQ_2$:
 \begin{equation}\label{eq:nodecharge:Q2}
  \left\{
  \begin{aligned}
    \widehat{\rho_1}=\Delta v\left(\frac{\widehat{f_1}}{6}+\frac{\widehat{f_4}}{3}+\frac{\widehat{f_7}}{6}\right),\\
    \widehat{\rho_2}=\Delta v\left(\frac{\widehat{f_2}}{6}+\frac{\widehat{f_5}}{3}+\frac{\widehat{f_8}}{6}\right),\\
    \widehat{\rho_3}=\Delta v\left(\frac{\widehat{f_3}}{6}+\frac{\widehat{f_6}}{3}+\frac{\widehat{f_9}}{6}\right).\\
  \end{aligned}
  \right.
\end{equation}
and with the similar discussion we follow the indices depicted in the right panel of Figure~\ref{fig:dofs} for the $\polQ_3$ case and get:
\begin{equation}\label{eq:nodecharge:Q3}
  \left\{
  \begin{aligned}
    \widehat{\rho_1}=\Delta v\left(\frac{\widehat{f_1}}{8}+\frac{3 \widehat{f_5}}{8}+\frac{3 \widehat{f_9}}{8}+\frac{\widehat{f_{13}}}{8}\right),\\
    \widehat{\rho_2}=\Delta v\left(\frac{\widehat{f_2}}{8}+\frac{3 \widehat{f_6}}{8}+\frac{3 \widehat{f_{10}}}{8}+\frac{\widehat{f_{14}}}{8}\right),\\
    \widehat{\rho_3}=\Delta v\left(\frac{\widehat{f_3}}{8}+\frac{3 \widehat{f_7}}{8}+\frac{3 \widehat{f_{11}}}{8}+\frac{\widehat{f_{15}}}{8}\right),\\
    \widehat{\rho_4}=\Delta v\left(\frac{\widehat{f_4}}{8}+\frac{3 \widehat{f_8}}{8}+\frac{3 \widehat{f_{12}}}{8}+\frac{\widehat{f_{16}}}{8}\right).\\
  \end{aligned}
  \right. 
\end{equation}

Note that, the quadrature rules shown in \eqref{eq:nodecharge}, \eqref{eq:nodecharge:Q2} and \eqref{eq:nodecharge:Q3} are exact for the corresponding polynomial spaces. 

\begin{remark}
  The local index does not have to be identical as described above. Once the meshes and function spaces are created, we can compute and store the local and global indices of all the degrees of freedom; this procedure is straightforward to implement for example in the framework of the DOLFINx project \cite{barrata2023dolfinx}, that is used to do the numerical simulation of this work.
\end{remark}

\subsection{Solution to the Poisson equation} \label{sec:FEPoiss}
The solution to the Poisson equation \eqref{eq:poiss} is not unique when boundary conditions such as pure Neumann or periodic boundary conditions are applied. In these cases, if a solution $\Phi_h$ to the Poisson equation is found, then clearly $\Phi_h + C$ is also a solution, where $C$ is any constant. However, both $\Phi_h$ and $\Phi_h + C$ yield the same electric field, since $\nabla_{\pmb{x}} C \equiv 0$, meaning the constant does not affect the electric field $\pmb{E}$. Therefore, it is sufficient to determine a single solution to the Poisson equation by introducing additional constraints, such as
\begin{equation}
    \int_{\Omega_{\pmb{x}}}\Phi_h\ {\rm d}\pmb{x}=0.\label{eq:const}
\end{equation}
The finite element method for the Poisson equation can be obtained by multiplying it with a test function $\phi\in\calV_{\pmb{x}}$ and integrating by parts. We employ the Lagrange multiplier to accomplish constraint \eqref{eq:const}. By introducing the additional unknown constant $c$, which belongs to the real number space $\mathbb{R}$, the following variational formulation for the Poisson equation \eqref{eq:poiss} is obtained: find $(\Phi_h,c)\in\calV_{\pmb{x}}\times\mathbb{R}$ such that
\begin{equation}
  \begin{aligned}
\int_{\Omega_{\pmb{x}}}\nabla_{\pmb{x}}\Phi_h\cdot\nabla_{\pmb{x}}\phi\ {\rm d}\pmb{x}+\int_{\Omega_{\pmb{x}}}c \phi\ {\rm d}\pmb{x}+\int_{\Omega_{\pmb{x}}}\Phi_h w\ {\rm d}\pmb{x}\\
    =\int_{\Omega_{\pmb{x}}}(\rho_h-\rho_0)\phi\ {\rm d}\pmb{x},\qquad\forall(\phi,w)\in\calV_{\pmb{x}}\times\mathbb{R}.\notag
  \end{aligned}
\end{equation} 
The Lagrange multiplier approach ensures that the solution to the Poisson equation \eqref{eq:poiss} is well-posed and unique.
Once the above Poisson equation is solved, we can obain the electric field by $\pmb{E}_h=-\nabla_{\pmb{x}}\Phi_h$.

\subsection{Summary of the Algorithm}  
In conclusion, the finite element solution to the Vlasov-Poisson equation \eqref{eq:vp} given by the RV method can be obtained by the following steps. At each time step, assume that we have obtained the solution $f_h^n$, we calculate the solution $f_h^{n+1}$ as follows
\begin{enumerate}
  \item Compute the charge density using \eqref{eq:charge3}, and solve the Poisson equation as described in Section \ref{sec:FEPoiss}.
  \item Calculate the residual \eqref{eq:residual}, and construct the artificial viscosity \eqref{eq:rv:visc} using the first-order and high-order viscosity.
  \item Compute the time step using the CFL condition \eqref{eq:CFL}.
  \item Apply the five stages fouth-order explicit SSP-Runge-Kutta method from \cite{Kraaijevanger_1991} with the Euler steps as in \eqref{eq:HO}.
\end{enumerate}

\begin{remark}[Extending the method to higher dimensions]
The stabilized finite element approximation of the Vlasov-Poisson system presented above can be easily extended to higher dimensions. Since the method is explicit in time, the most computationally demanding tasks involve storing and inverting matrices at each Euler step \eqref{eq:HO} and computing the residual \eqref{eq:residual}. A matrix-free framework, as presented in \citep{MR4321466}, where the authors employed a DG scheme for the six-dimensional Vlasov-Poisson system, can be readily applied to our setting, given that we also use tensor-product finite elements on quadrilateral elements. 
\end{remark}

\section{Numerical examples}\label{sec:experiment} 

In this section, we verify our proposed finite element methods in the 1D1V Vlasov-Poisson equation. The number of elements is $(N_x-1)/k$ and $(N_v-1)/k$
in the physical and velocity spaces, respectively. The methods are verified by solving several classic benchmark problems of Vlasov-Poisson equations, and we evaluate the accuracy of the methods through a convergence study.

In the following numerical examples, we use periodic boundary conditions, which are implemented by mapping the nodal points to the two opposite boundaries of the rectangular domains. We use ${\rm CFL}=0.4$ for all the experiments and for every polynomial degrees in this work. 
 
\subsection{Landau damping} 

First, we test our methods in the Landau damping phenomenon. We set the phase space to be $\Omega:=[0,4\pi]\times[-6,6]$, and use the following initial data
\begin{equation}
  f(x,v,0)=\frac{1}{\sqrt{2\pi}}{\rm exp}\left(-\frac{v^2}{2}\right)(1+\alpha{\rm cos}(\theta x)),\notag
\end{equation}
where $\theta = 0.5$. We calculate the logarithm of the square root of the electric energy $E_e = {\rm ln}\big( \int_{\Omega_x}(-\partial_x\Phi_h)^2\ {\rm d}x \big)^\frac12$
and plot the evolution of $E_e$ over time for several values of $\alpha$. 
\begin{figure}[htbp] 
  \centering
  \begin{subfigure}[b]{0.5\textwidth}
      \centering
      \includegraphics[width=\linewidth]{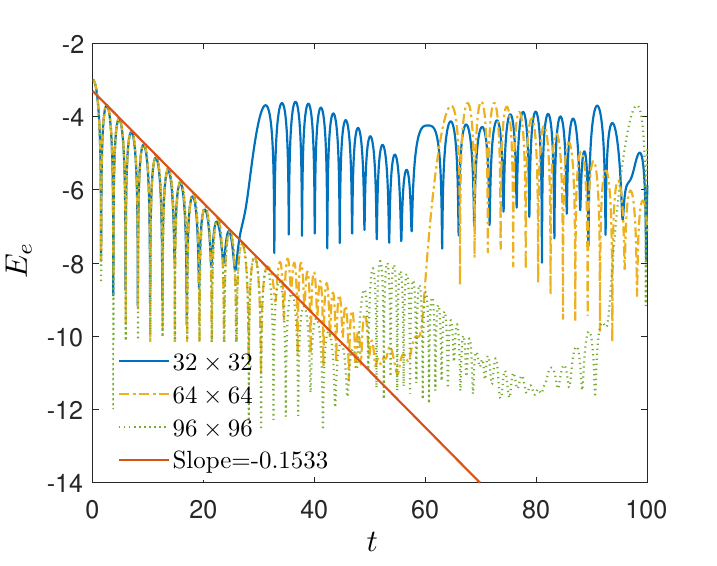}
      \caption{$\mathbb{Q}_1$ solutions}
  \end{subfigure}\hfill
  \begin{subfigure}[b]{0.5\textwidth}
      \centering
      \includegraphics[width=\linewidth]{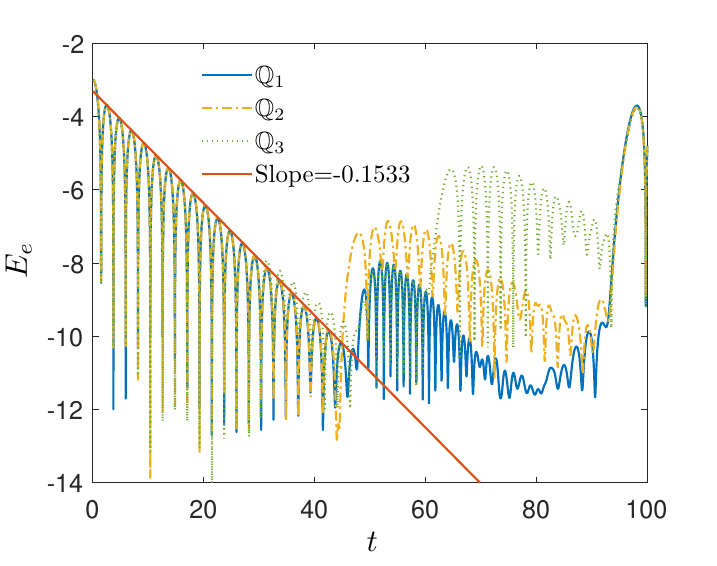}
      \caption{$N_x\times N_v=97\times97$}
  \end{subfigure}\hfill
  \caption{Landau damping. Linear damping, evolution of $E_e$ for $\alpha=0.01$. The red line represents the analytical damping rate. The RV solutions using $\polQ_1$ elements with varying degrees of freedom (a) and different polynomial degrees with the same degree of freedom distribution (b).} 
  \label{fg:Landau1}
\end{figure}

When $\alpha = 0.01$, the damping is relatively weak, and $E_e$ is expected to decrease linearly; the analytical damping rate is 0.1533, as introduced in \cite{MR2016914}. 
Since we discretize the Vlasov-Poisson equation using a grid-based scheme, the recurrence phenomenon will occur in the electric energy. One straightforward 
way to delay the recurrence time is by refining the mesh. Figure \ref{fg:Landau1} (a) includes $\polQ_1$ solutions with different amounts of elements, showing that the recurrence occurs at a larger time stage as the mesh is refined.
We also compare the performance of elements with varying degrees. Since high-order elements inherently contain more degrees of freedom, we ensure fairness by comparing solutions obtained using the same node distribution in Figure \ref{fg:Landau1} (b). From $t=30$ to $t=90$, the lower-order solutions perform better to some extent, as the damping rate remains closer to the analytical rate for a longer duration, and the electric energy is lower compared to the higher-order solution. For finite element methods, high-order polynomials usually achieve higher convergence rates in smooth problems, while low-order polynomials have better performance in discontinuous cases or when shocks or oscillations are present, see for instance \citep[Sec.~4.2]{MR4754161} and \citep[Sec.~6.2]{Kronbichler_2024}. For the Landau damping, the recurrence happens due to the oscillations in the velocity distribution casued by the small initial perturbations, as discussed in \cite{MR4126517}. Therefore, it is not surprising that the lower-order scheme outperforms the higher-order one in this scenario. In addition, we observe that the electric energy eventually recovers to the same level for all element degrees.

\begin{figure}[htbp]
  \centering
  \begin{subfigure}[b]{0.5\textwidth}
      \centering
      \includegraphics[width=\linewidth]{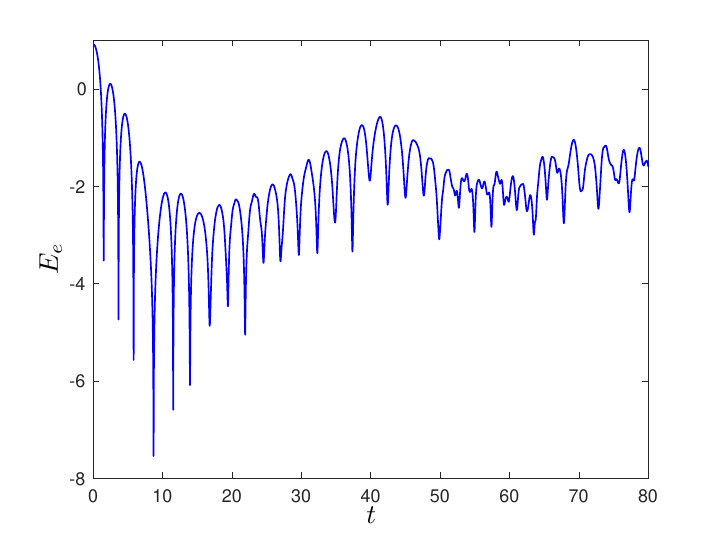}
      \caption{$\polQ_1$ Galerkin FE solution}
  \end{subfigure}\hfill
  \begin{subfigure}[b]{0.5\textwidth}
      \centering
      \includegraphics[width=\linewidth]{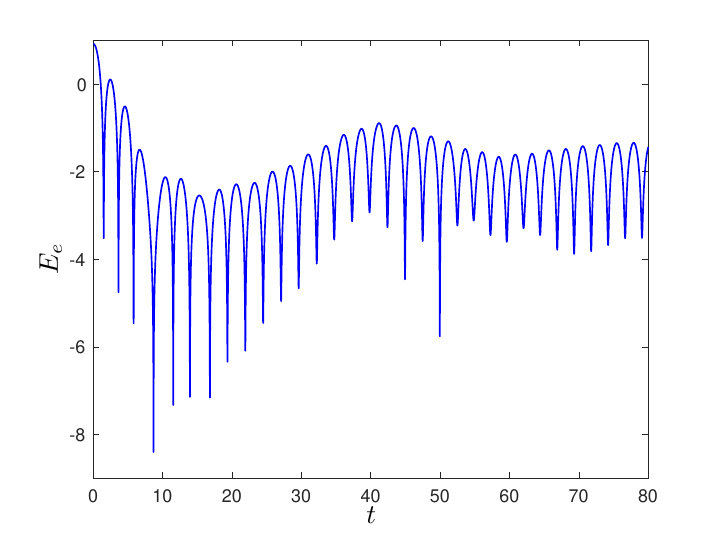} 
      \caption{$\polQ_1$ RV solution}
  \end{subfigure}\hfill
  \begin{subfigure}[b]{0.5\textwidth}
      \centering
      \includegraphics[width=\linewidth]{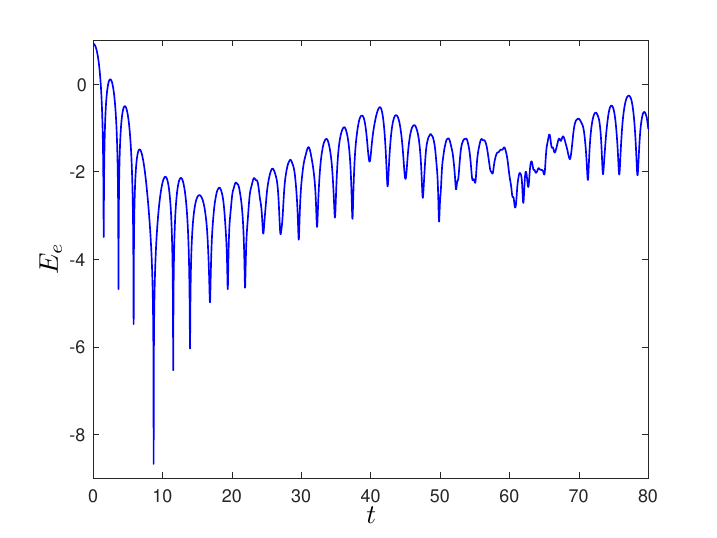}
      \caption{$\polQ_2$ Galerkin FE solution}
  \end{subfigure}\hfill
  \begin{subfigure}[b]{0.5\textwidth}
      \centering
      \includegraphics[width=\linewidth]{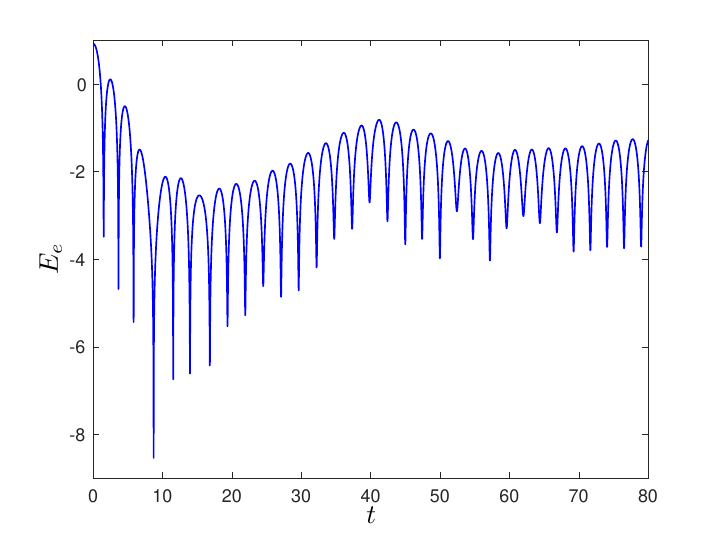} 
      \caption{$\polQ_2$ RV solution}
  \end{subfigure}\hfill
  \begin{subfigure}[b]{0.5\textwidth}
      \centering
      \includegraphics[width=\linewidth]{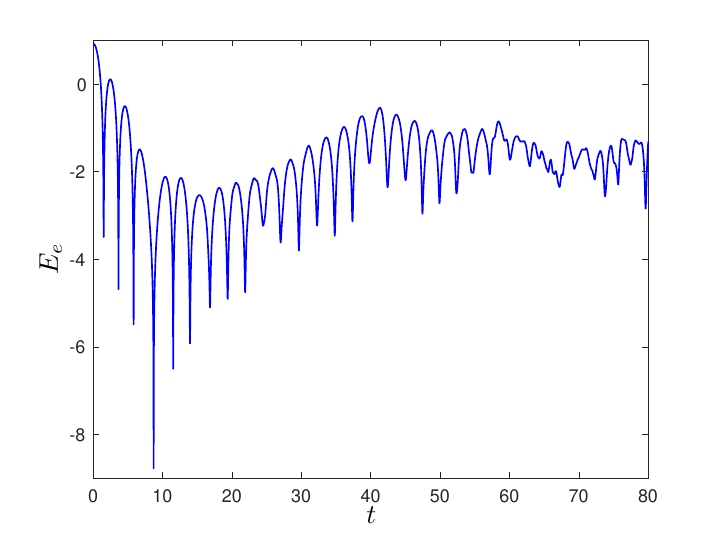}
      \caption{$\polQ_3$ Galerkin FE solution}
  \end{subfigure}\hfill
  \begin{subfigure}[b]{0.5\textwidth}
      \centering
      \includegraphics[width=\linewidth]{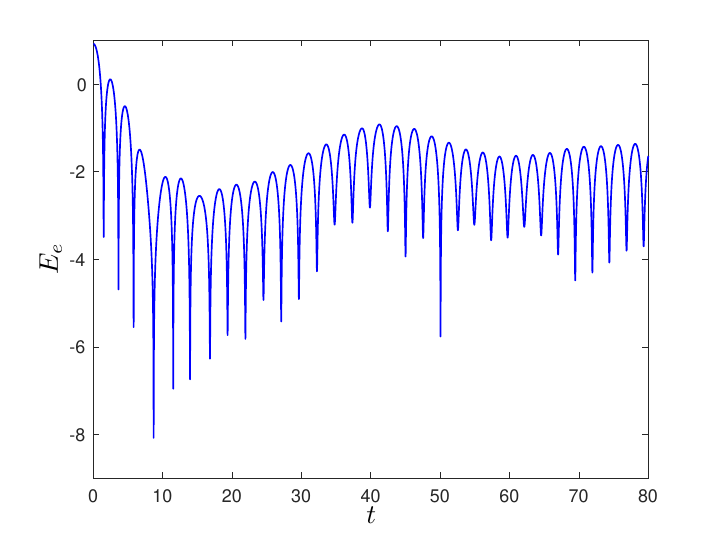} 
      \caption{$\polQ_3$ RV solution}
  \end{subfigure}\hfill
  \caption{Landau damping. Strong damping, the evolution of $E_e$ for $\alpha = 0.5$. The distribution of nodes is $N_x\times N_v=49\times97$. Comparison of the Galerkin solution and stabilized RV solutions for different polynomial spaces.}
  \label{fg:Landau2}
\end{figure}

Now, let us consider the case when $\alpha=0.5$. In this case, the damping is strong, and the solutions are expected to produce shocks and discontinuities. This makes it a suitable benchmark problem to demonstrate the necessity of using artificial viscosity. We present the results obtained by the Galerkin finite element and residual viscosity (RV) methods in Figure \ref{fg:Landau2}, using $\polQ_1$, $\polQ_2$, and $\polQ_3$ elements with the same distribution of nodes, i.e., $N_x\times N_v=49\times97$. The standard Galerkin FE method is unable to resolve the spurious oscillations, regardless of the polynomial degree, as shown in Figure \ref{fg:Landau2} (a), (c), and (e). Conversely, in the RV solutions, the electric energy decreases exponentially at first, then oscillates periodically, which closely matches the reference solutions.

\subsection{Two-stream instability} 
Now we consider the two-stream instability of Vlasov-Poisson equations in the phase space $\Omega := [0,4\pi]\times[-5,5]$. Let the initial data be
\begin{equation}
  f(x,v,0)=\frac{1}{\sqrt{2\pi}}v^2{\rm exp}\left(-\frac{v^2}{2}\right)(1+\alpha{\rm cos}(\theta x)), \notag
\end{equation}
where $\alpha = 0.01$, $\theta = 0.5$. 

We test the convergence rates of our methods through a convergence study. In most cases, the analytical solutions for Vlasov-Poisson equations are not available, but we still be able to obtain the analytical solutions under some specific settings as introduced in \cite{MR3267101}. Assuming periodic boundary conditions and symmetric velocity space, we solve the Vlasov-Poisson equation with the initial data $f(\pmb{x},\pmb{v},0)$ from $t=0$ to $t=T$, and obtain the solution $f(\pmb{x},\pmb{v},T)$. Then if we use the function $f(\pmb{x},-\pmb{v},T)$ as the initial data and solve the Vlasov-Poisson equation again, the analytical solution at $t=T$ is supposed to be $f(\pmb{x},-\pmb{v},0)$. We apply both the standard Galerkin finite element and RV methods, using $\polQ_1$, $\polQ_2$, and $\polQ_3$ elements, and the reference solutions are obtained by the method mentioned above, with $T=5$. We calculate the errors of the solutions and the $L^1$-, $L^2$-, and $L^\infty$-norms of the errors. We run the simulations with different amounts of degrees of freedom, and then calculate the convergence rates of the methods, as shown in Table \ref{tab:GFEM} and \ref{tab:RV}; moreover, the $L^2$-norms of the errors are plotted in Figure \ref{fig:convergence}.  

Based on the results organized here, we conclude that the artificial viscosity does not decrease the accuracy of the Galerkin finite element method when solving smooth problems. We observe that the convergence rates for the $\polQ_1$ and $\polQ_3$ solutions are second-order and fourth-order, respectively, aligning with our expectations. However, the continuous Galerkin method is known to be a sub-optimal scheme, which means the second-order polynomial usually exhibits a second-order convergence instead of third-order, for more details, see \cite{MR3148555}. This phenomenon has been reported in previous works, as referenced in \cite{MR4456197,MR3612753}. In our results, the convergence rate of the $\polQ_2$ solution is higher than second-order and approaches third-order. This finding is not surprising, because the basis function of the $\polQ_2$ element contains more terms compared to that of the $\polP_2$ element in the reference papers, i.e., $\polP_2 = {\rm span}\{1,x,v,x^2,v^2,xv\}$, while $\polQ_2 = {\rm span}\{1,x,v,x^2,v^2,xv,x^2v,xv^2,x^2v^2\}$; therefore, the $\polQ_2$ solutions can achieve recovered convergence rates. 

The $\polQ_3$ solutions in the mesh consisting of $128\times 256$ elements are plotted in Figure \ref{fg:B2_fh}. It is noticed that there are two components that are screwed together gradually in the phase space, and a vortex structure eventually develops. The obtained results coincidence with those in numerous previous works, see for instance \cite{MR4402737,CROUSEILLES20091429, MR4354369}. We also plot the artificial viscosity of the $\polQ_1$ and $\polQ_3$ solutions in Figure \ref{fg:B2_mu}, and we observe that the artificial viscosity tracks the shocks of the distribution function. 
The viscosity \(\varepsilon_x\) and \(\varepsilon_v\) have similar distributions but differ in scale. This is because the residual viscosity dominates, instead of the first-order viscosity, and the residual is normalized with distinct mesh sizes in different dimensions. The $\polQ_3$ viscosity appears blurrier due to the noise in the residuals for high-order polynomials.

\begin{table}[htbp]
    \caption{Two-stream instability: errors of Galerkin FE solutions. Run until $t=5$ and then back to $t=10$.} \label{tab:GFEM}
  \begin{center}
    \begin{tabular}{c|c|c|c|c|c|c|c} \hline
      &$N_x\times N_v$&$L^1$-error&Rate&$L^2$-error&Rate&$L^\infty$-error&Rate\\ \hline
      \multirow{4}{*}{$\polQ_1$} &$31\times31$ & 2.05E-02&-&2.21E-02 &- &3.44E-02&- \\
     &$61\times61$&5.14E-03&1.99&5.52E-03&2.00&9.27E-03&1.89 \\
     &$121\times121$&1.28E-03&2.00&1.38E-03&2.00&2.35E-03&1.98 \\
     &$241\times241$&3.21E-04&2.00&3.45E-04&2.00&5.89E-04&1.99  \\ \hline
     \multirow{4}{*}{$\polQ_2$} &$31\times31$ & 6.95E-03&-&8.11E-03 &- &1.39E-02&- \\
     &$61\times61$&1.03E-03&3.02&1.04E-03&2.96&1.76E-03&2.98 \\
     &$121\times121$&1.28E-04&3.02&1.31E-04&2.99&2.24E-04&2.98 \\
     &$241\times241$&1.60E-05&3.00&1.64E-05&3.00&2.80E-05&3.00  \\ \hline
     \multirow{4}{*}{$\polQ_3$} &$31\times31$ & 2.86E-03&-&2.66E-03 &- &4.01E-03&- \\
     &$61\times61$&2.30E-04&3.63&2.37E-04&3.49&4.03E-04&3.31 \\
     &$121\times121$&1.44E-05&4.00&1.50E-05&3.98&2.92E-05&3.79 \\
     &$241\times241$&9.05E-07&3.99&9.43E-07&3.99&1.90E-06&3.95 \\ \hline
    \end{tabular}
  \end{center}
  \end{table}
  \begin{table}[htbp]
    \caption{Two-stream instability: errors of RV solutions. Run until $t=5$ and then back to $t=10$.}  \label{tab:RV}
  \begin{center}
    \begin{tabular}{c|c|c|c|c|c|c|c} \hline
      & $N_x\times N_v$ &   $L^1$-error  &    Rate  &    $L^2$-error   &    Rate&    $L^\infty$-error   & Rate      \\ \hline
       \multirow{4}{*}{$\polQ_1$} 
     &$31\times31$ &2.21E-02 &- &2.36E-02 &- &3.73E-02 &-  \\
     &$61\times61$ &5.46E-03 &2.01 &5.88E-03 &2.01 &9.94E-03 &1.91 \\
     &$121\times121$ &1.32E-03 &2.05 &1.42E-03 &2.05 &2.43E-03 &2.03\\
     &$241\times241$ &3.24E-04 &2.02 &3.50E-04 &2.02 &5.98E-04 &2.02 \\ \hline
     \multirow{4}{*}{$\polQ_2$} 
     &$31\times31$ &7.93E-03 &- &8.92E-03 &- &1.62E-02 &-  \\
     &$61\times61$ &1.06E-03 &2.91 &1.07E-03 &3.06 &1.89E-03 &3.10 \\
     &$121\times121$ &1.36E-04&2.95 &1.54E-04 &2.80 &3.79E-04 &2.32\\
     &$241\times241$ &1.61E-05 &3.08 &1.69E-05 &3.18 &3.37E-05 &3.49 \\ \hline
        \multirow{4}{*}{$\polQ_3$} 
     &$31\times31$ &2.67E-03 &- &2.52E-03 &- &3.33E-03 &-  \\
     &$61\times61$ &2.23E-04 &3.58 &2.30E-04 &3.46 &3.94E-04 &3.08 \\
     &$121\times121$ &1.42E-05 &3.97 &1.48E-05 &3.95 &2.91E-05 &3.76\\
     &$241\times241$ &9.02E-07 &3.98 &9.38E-07 &3.98 &1.89E-06 &3.94 \\ \hline
    \end{tabular}
  \end{center}
\end{table}

\begin{figure}[htbp]
    \centering
    \begin{subfigure}[b]{0.5\textwidth}
        \centering
        \includegraphics[width=\linewidth]{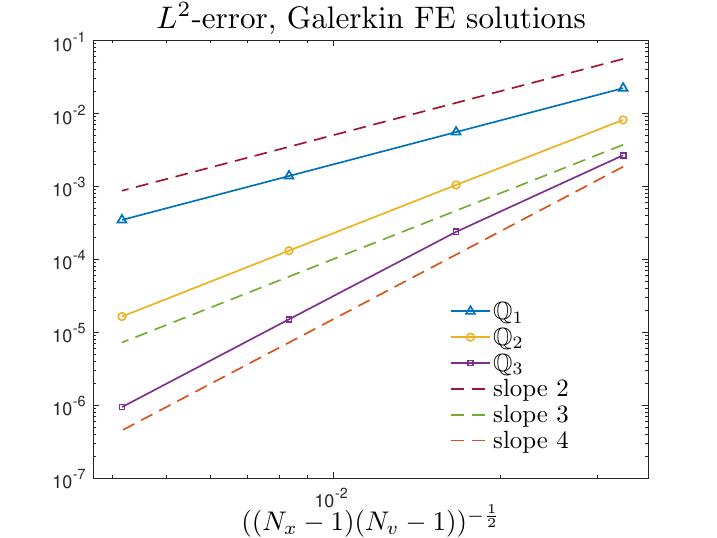}
    \end{subfigure}\hfill
    \begin{subfigure}[b]{0.5\textwidth}
        \centering
        \includegraphics[width=\linewidth]{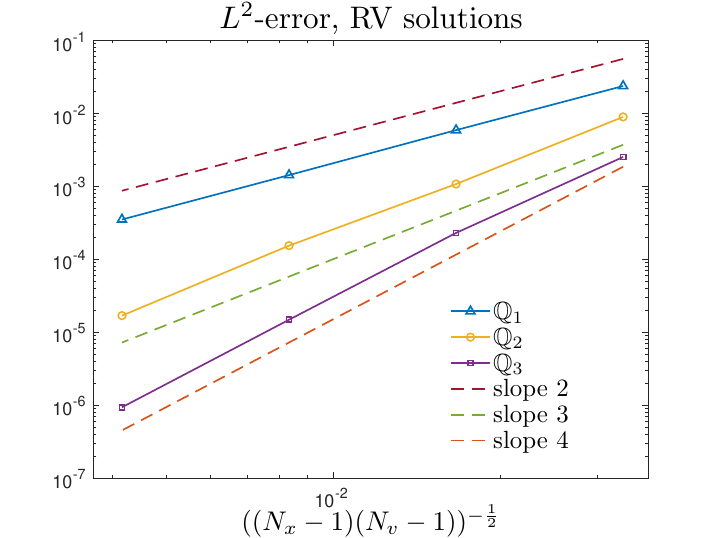} 
    \end{subfigure}\hfill
    \caption{Two-stream instability: convergence orders of the methods. The RV method does not decrease the accuracy of high-order solutions. Run until $t=5$ and then back to $t=10$.}
    \label{fig:convergence}
\end{figure}

\begin{figure}[htbp]
    \centering
    \begin{subfigure}[b]{0.5\textwidth}
      \includegraphics[width=\linewidth]{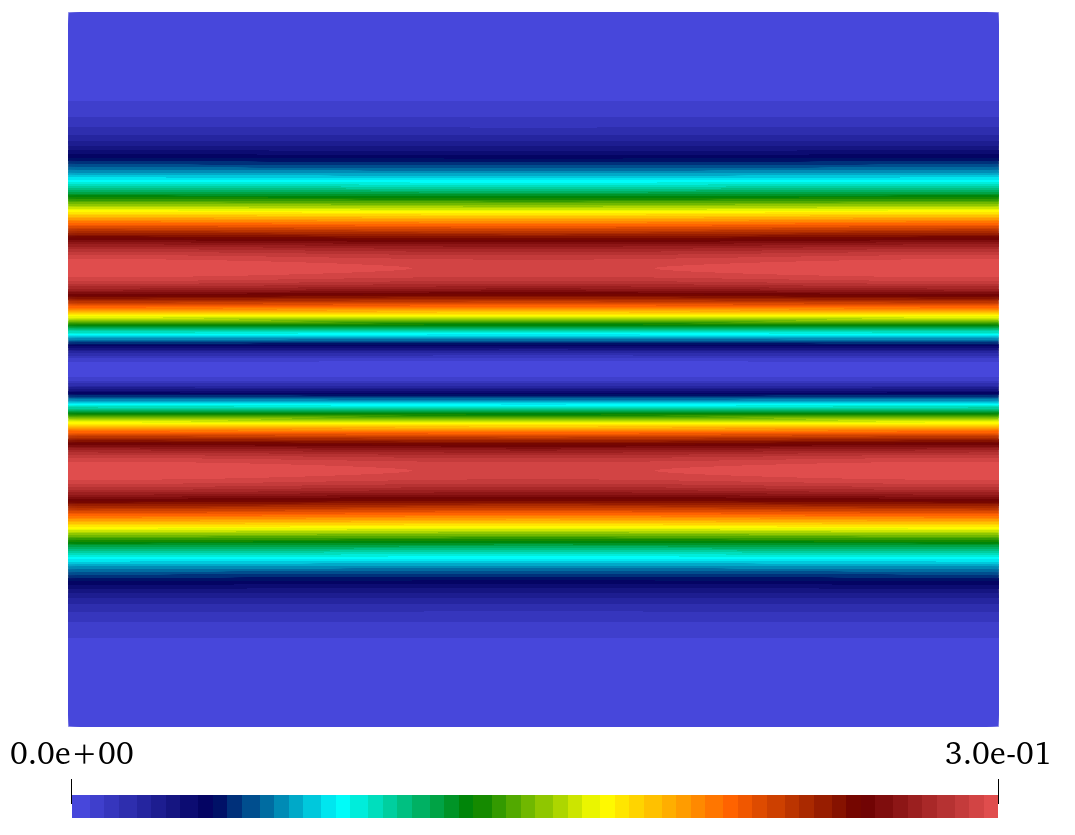}
      \caption{$t=0$}
    \end{subfigure}\hfill
    \begin{subfigure}[b]{0.5\textwidth}
      \includegraphics[width=\linewidth]{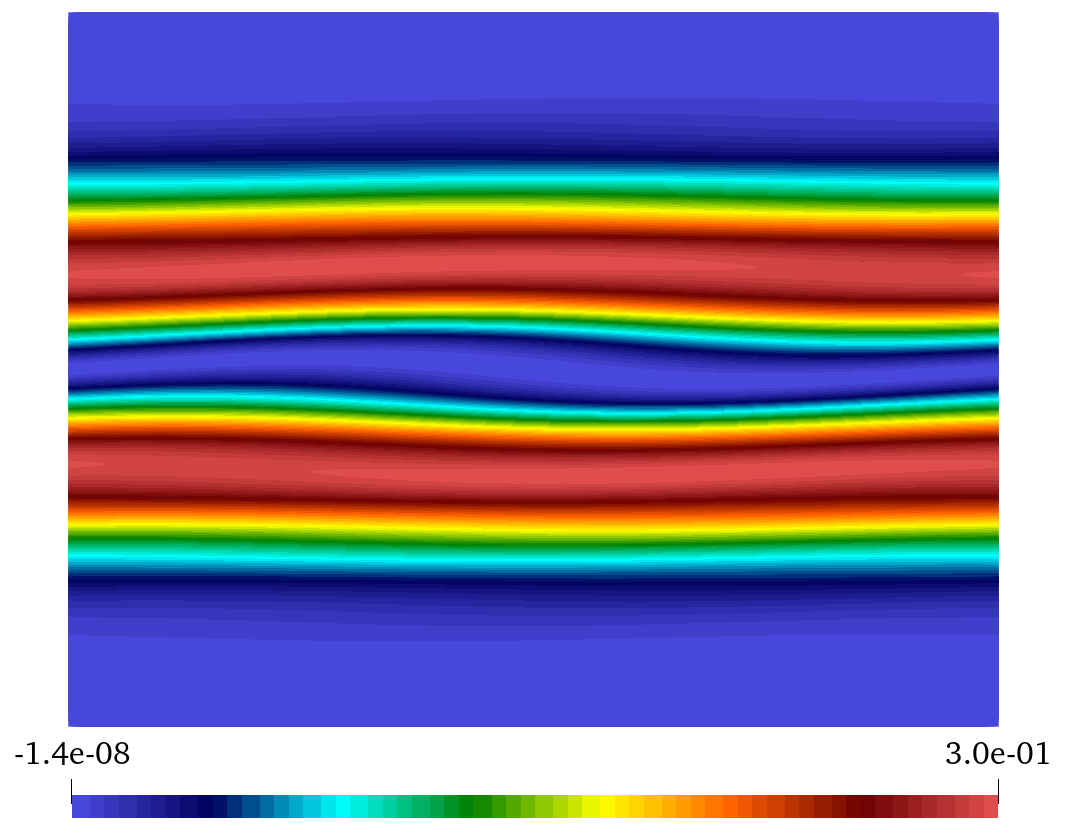}
      \caption{$t=12$}
    \end{subfigure}\hfill
    \centering
    \begin{subfigure}[b]{0.5\textwidth}
      \includegraphics[width=\linewidth]{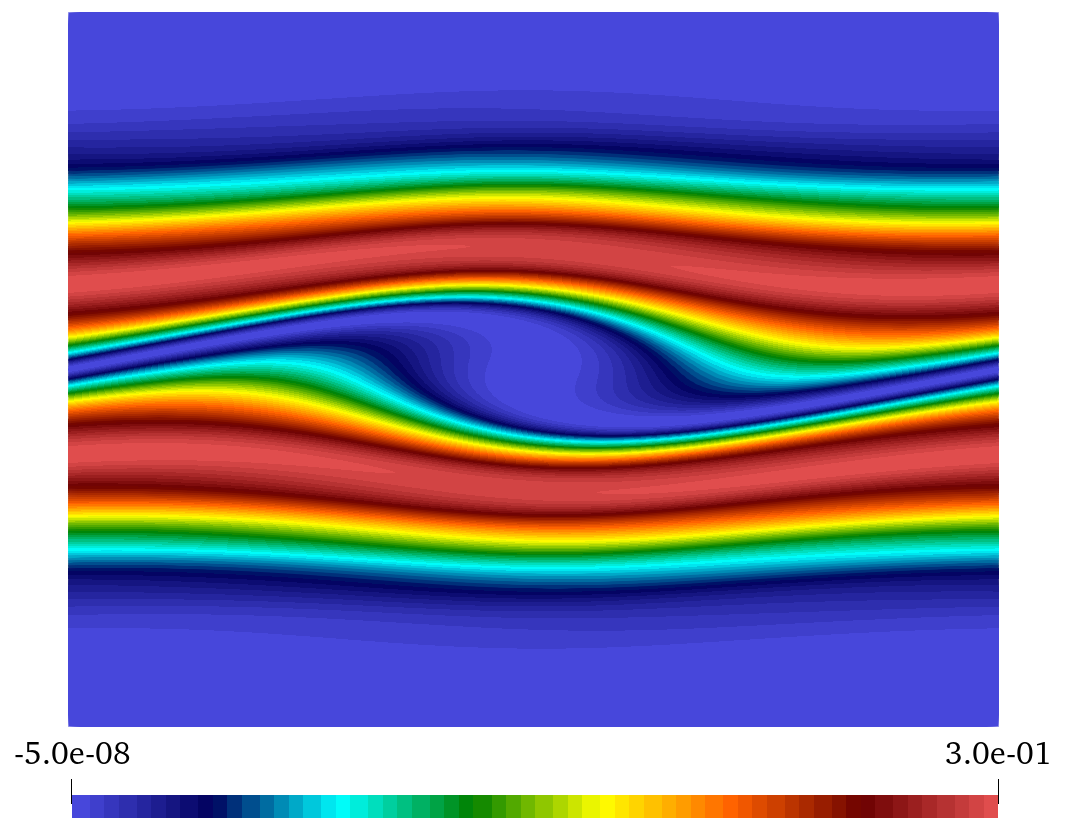}
      \caption{$t=18$}
    \end{subfigure}\hfill
    \begin{subfigure}[b]{0.5\textwidth}
      \includegraphics[width=\linewidth]{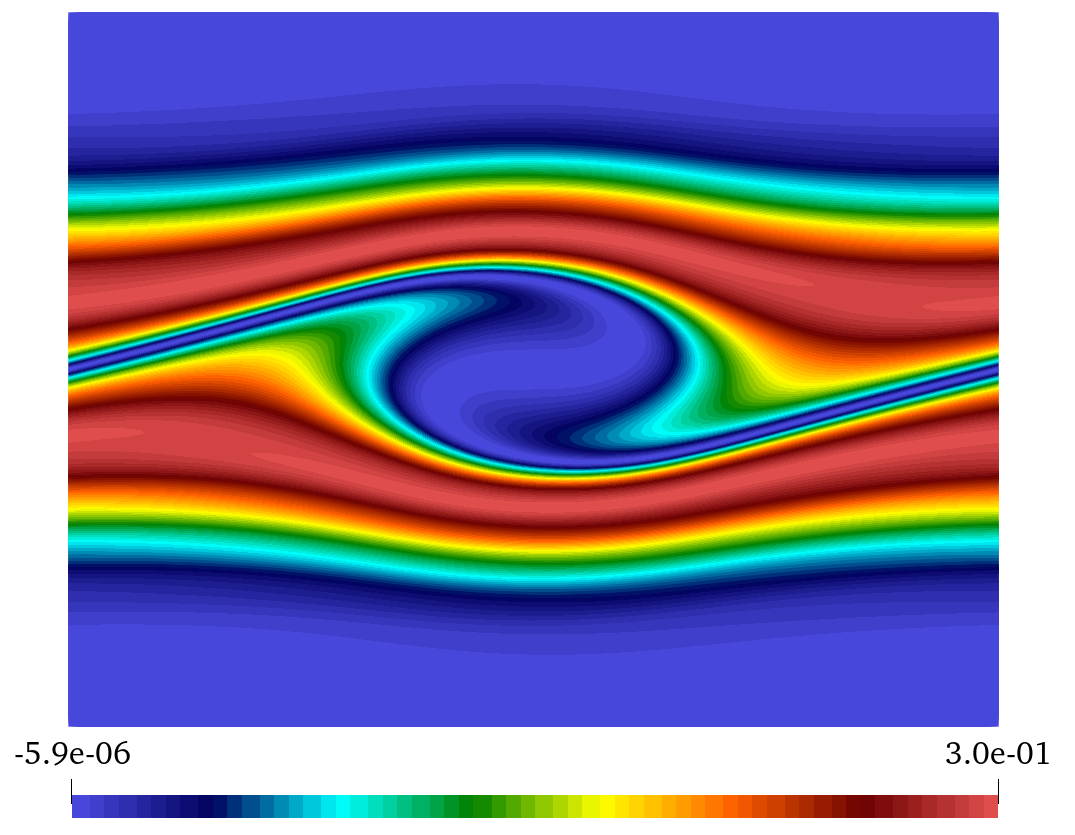}
      \caption{$t=21$}
    \end{subfigure}\hfill
    \centering
    \begin{subfigure}[b]{0.5\textwidth}
      \includegraphics[width=\linewidth]{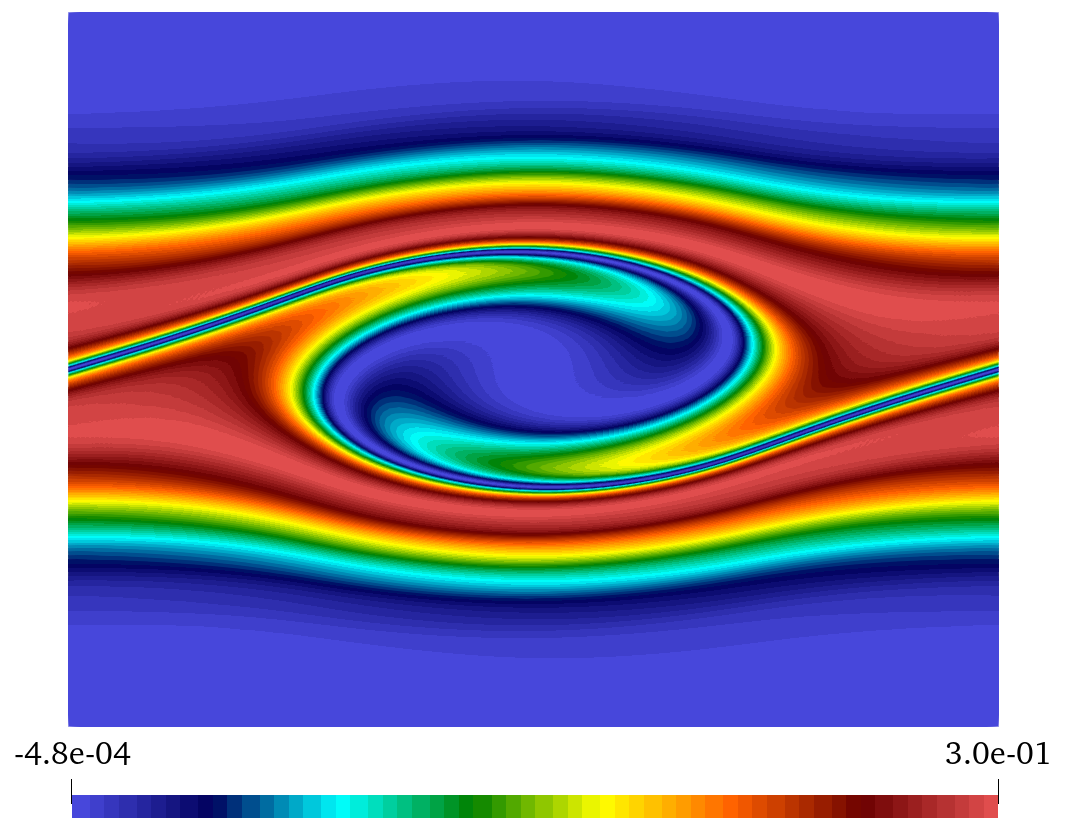}
      \caption{$t=24$}
    \end{subfigure}\hfill
    \begin{subfigure}[b]{0.5\textwidth}
      \includegraphics[width=\linewidth]{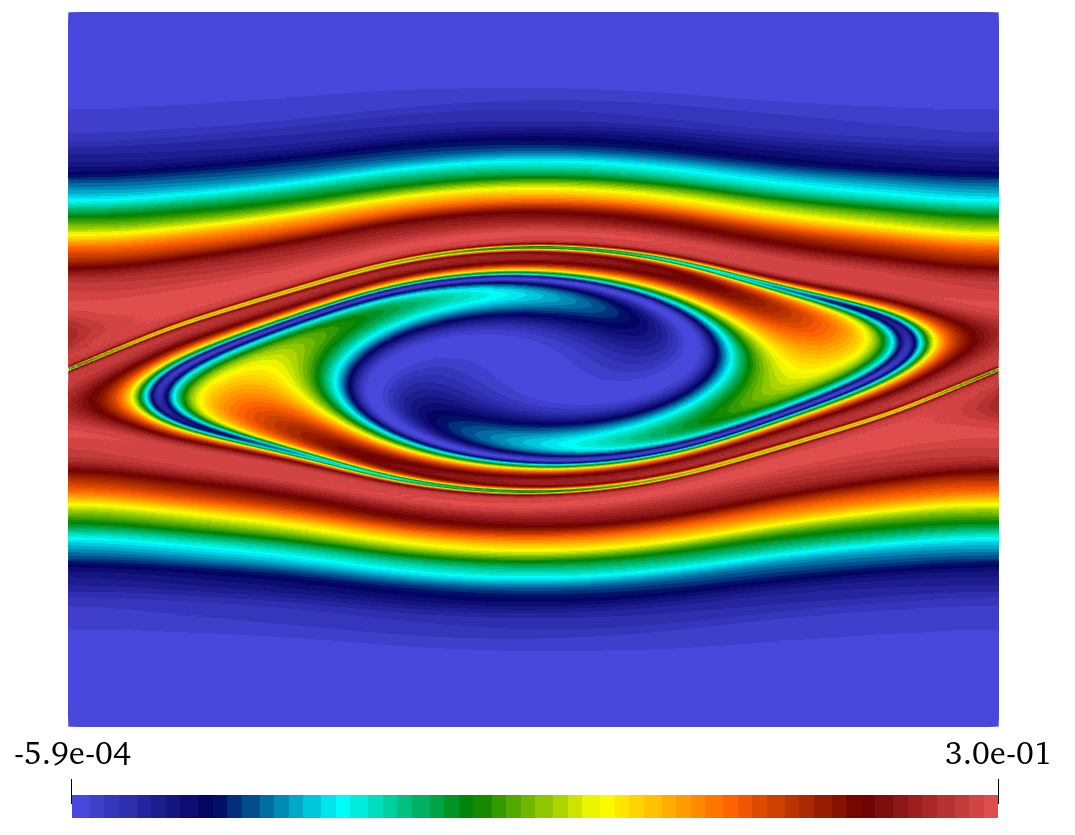}
      \caption{$t=30$}
    \end{subfigure}\hfill
    \caption{Two-stream instability: $\polQ_3$ solutions at $t=0, 12, 18, 21, 24,$ and $30$. The functions are plotted at the nodal points, $N_x\times N_v = 385\times769$.}
    \label{fg:B2_fh}
\end{figure}

\begin{figure}[htbp]
  \centering
  \begin{subfigure}[b]{0.48\textwidth}
    \includegraphics[width=\linewidth]{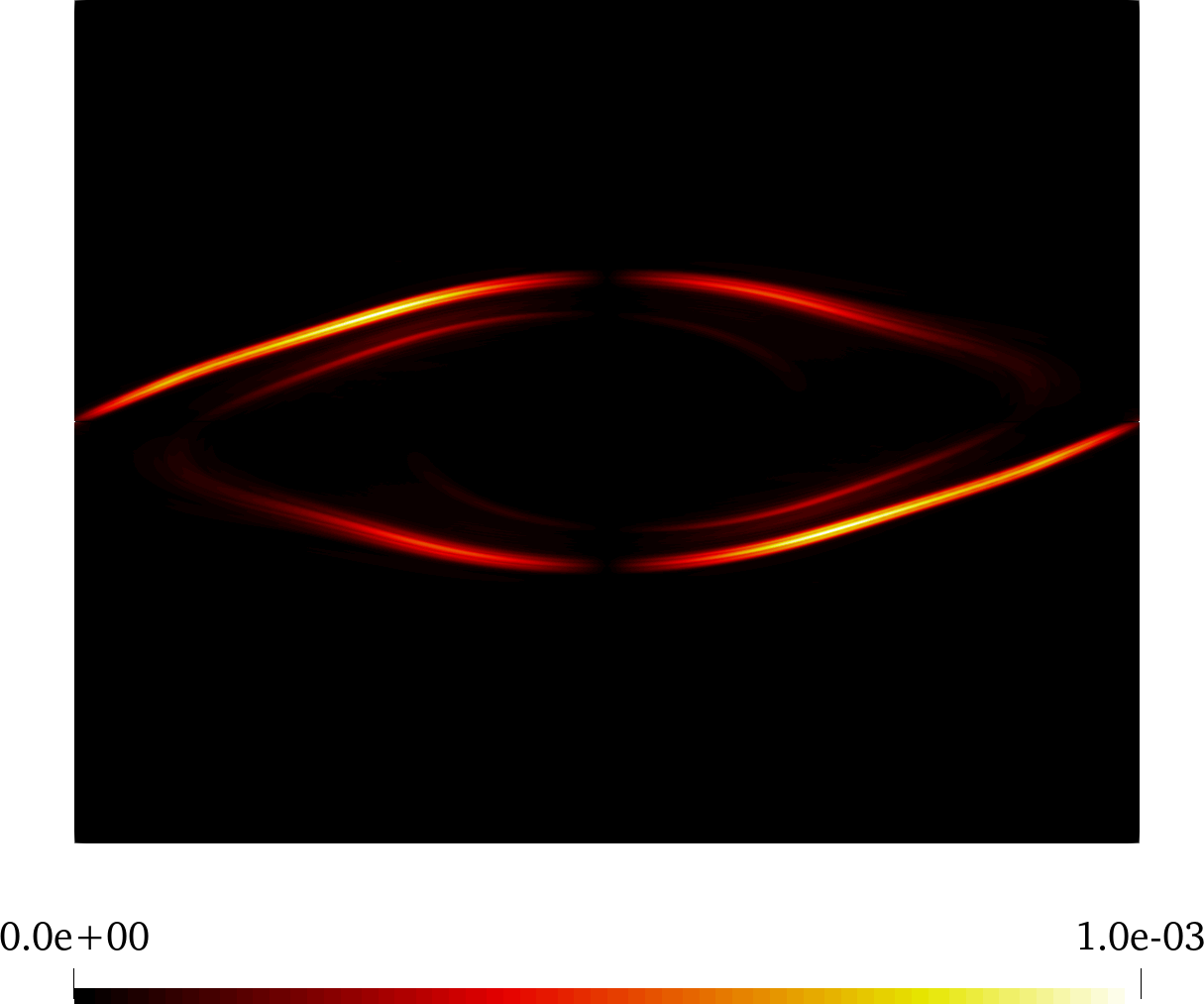}
    \caption{$\polQ_1$ viscosity $\varepsilon_x$}
  \end{subfigure}\hspace{5pt} 
  \begin{subfigure}[b]{0.48\textwidth}
    \includegraphics[width=\linewidth]{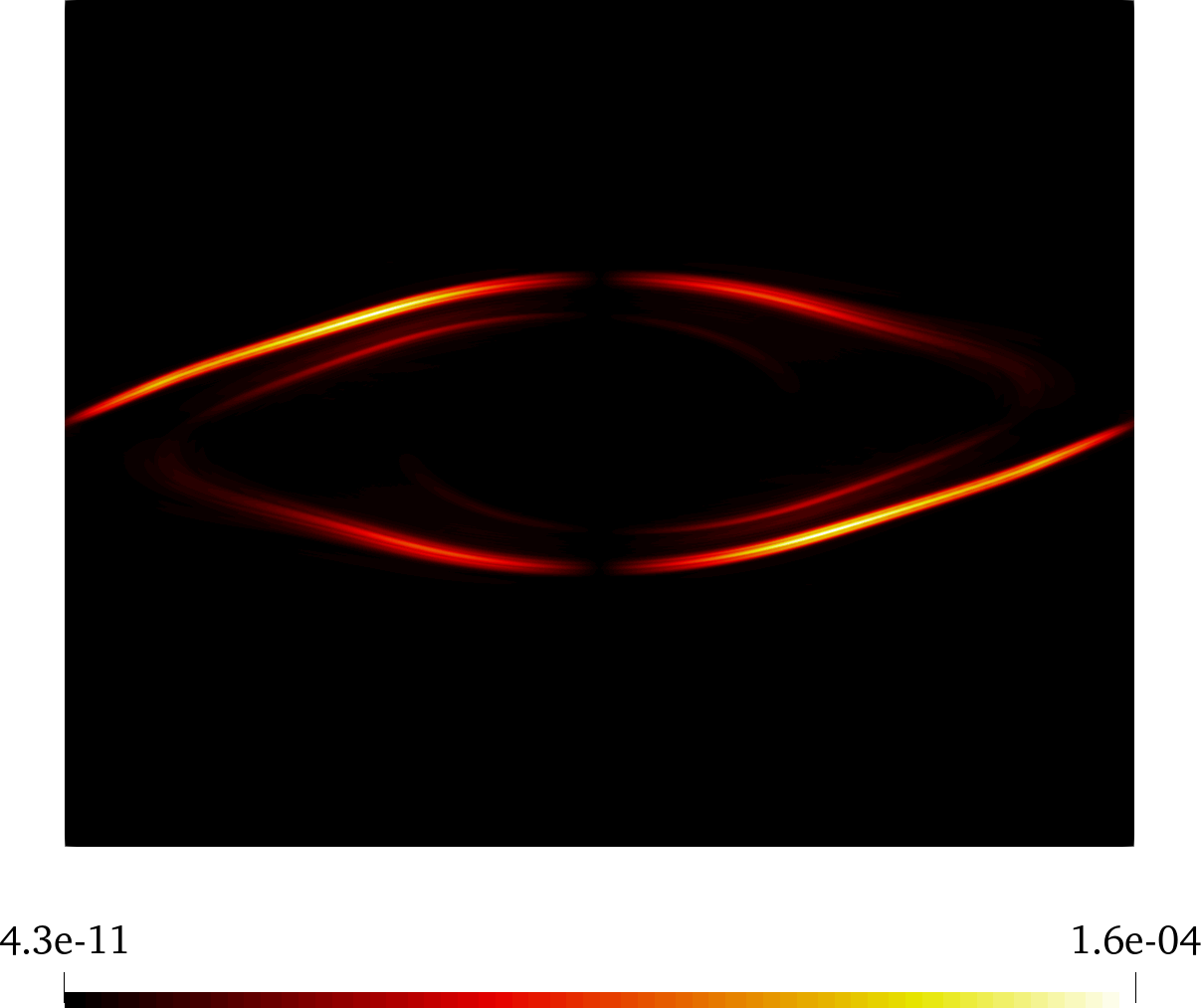}
    \caption{$\polQ_1$ viscosity $\varepsilon_v$}
  \end{subfigure}\hfill
  \centering
  \begin{subfigure}[b]{0.48\textwidth}
    \includegraphics[width=\linewidth]{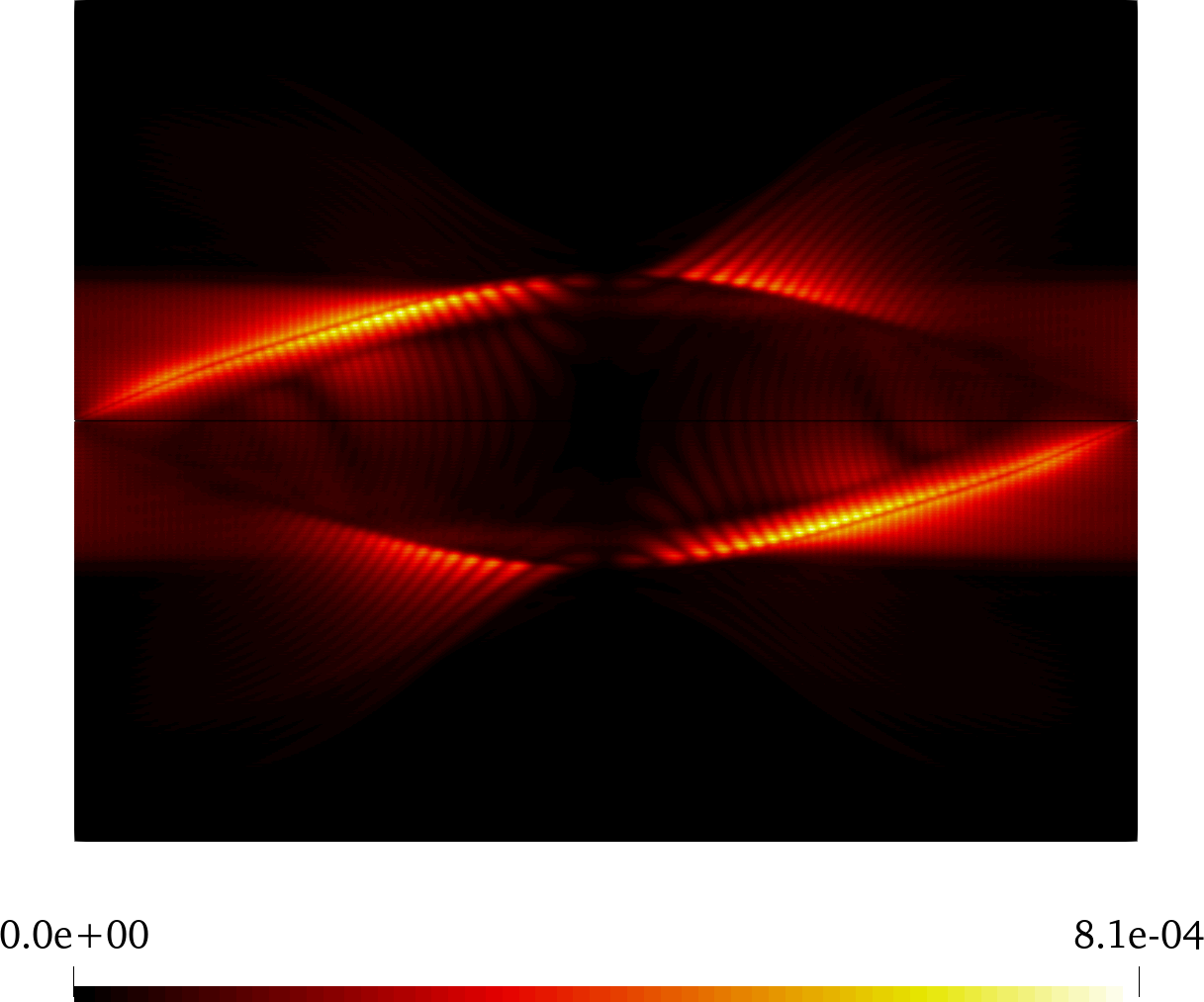}
    \caption{$\polQ_3$ viscosity $\varepsilon_x$}
  \end{subfigure}\hspace{5pt} 
  \begin{subfigure}[b]{0.48\textwidth}
    \includegraphics[width=\linewidth]{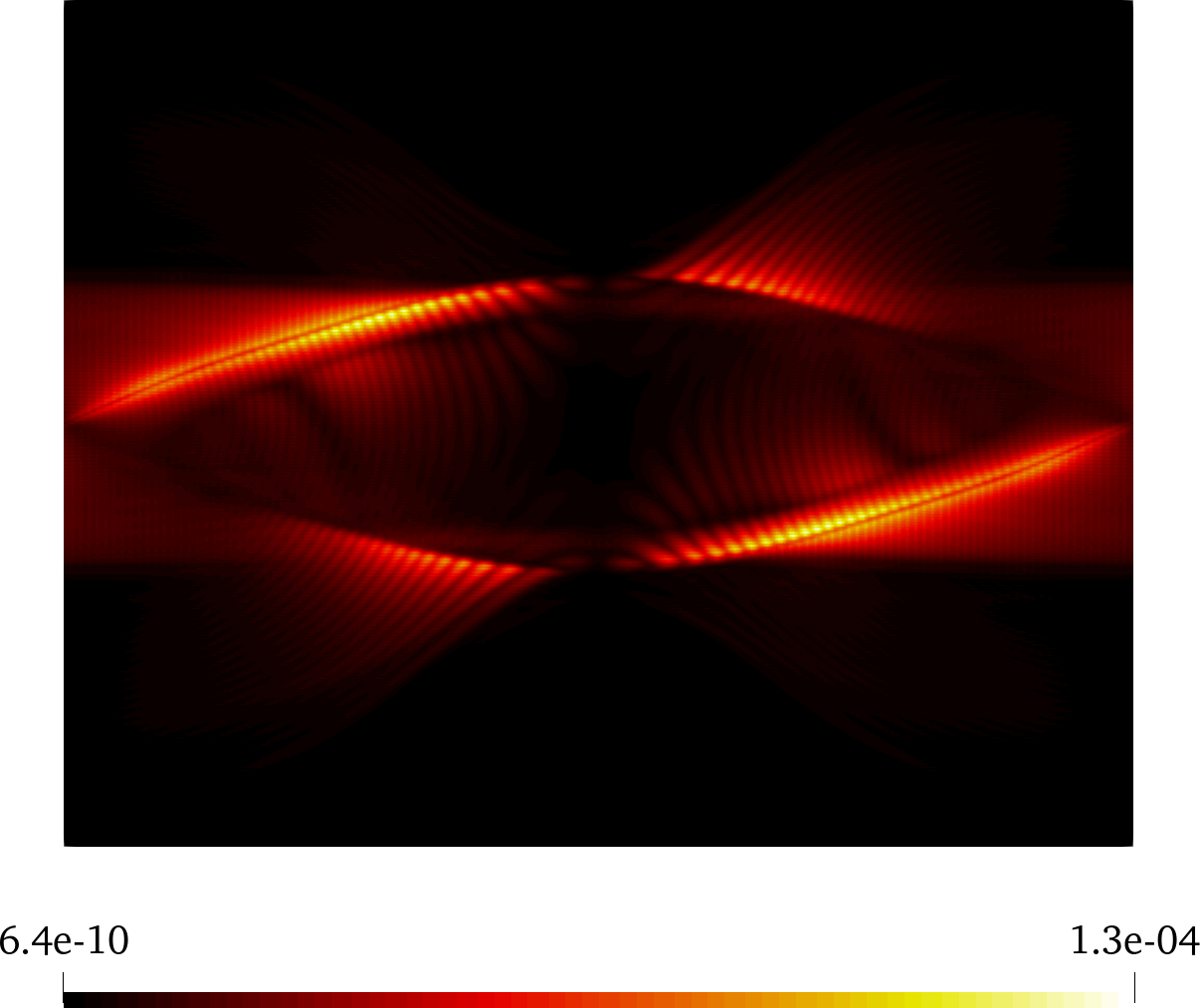}
    \caption{$\polQ_3$ viscosity $\varepsilon_v$}
  \end{subfigure}\hfill
  \caption{Two-stream instability: $\polQ_1$ and $\polQ_3$ artificial viscosity at time $t=30$. The functions are plotted at the nodal points, $N_x\times N_v = 385\times769$.}
\label{fg:B2_mu}
\end{figure}

We also apply our methods to a multi-vortex two-stream instability, where the solutions are expected to develop complex vortex structures. This allows us to evaluate the capability of the methods in capturing these structures and the long-term performance of the methods.
We use the following initial distribution function 
\begin{equation}
  f(x,v,0)=\frac{1}{2v_{th}\sqrt{2\pi}}\left({\rm exp}\left(-\frac{(v-u)^2}{2v^2_{th}}\right)+{\rm exp}\left(-\frac{(v+u)^2}{2v^2_{th}}\right)\right)(1+\alpha{\rm cos}(\theta x)), \notag
\end{equation}
with $u=0.99$, $v_{th}=0.3$, $\alpha = 0.05$, and $\theta = \frac{2}{13}$, in the phase space $\Omega:=[0,13\pi]\times[-5,5]$.
The $\polQ_3$ solutions in the mesh consisting of $128\times 256$ elements are plotted in Figure \ref{fg:B4_fh}.
In our simulations, we observe the formation of multiple vortices, which eventually merge and evolve into a fewer number of vortices in the final time stage. This test case has been studied in many other works,  and our long-term observations are consistent with the results reported in \cite{MR4844786,MR2586230,Umeda2008773}.
Furthermore, our numerical solutions at $t=40$ and $t=70$ closely match those presented in \cite[Fig.3.13.]{MR3738121} and \cite[Fig.6.6.]{MR2843721}, respectively. In particular, we observe the same complex structures as reported in these references, further validating the accuracy of our simulations.

\begin{figure}[htbp]
  \centering
  \begin{subfigure}[b]{0.5\textwidth}
    \includegraphics[width=\linewidth]{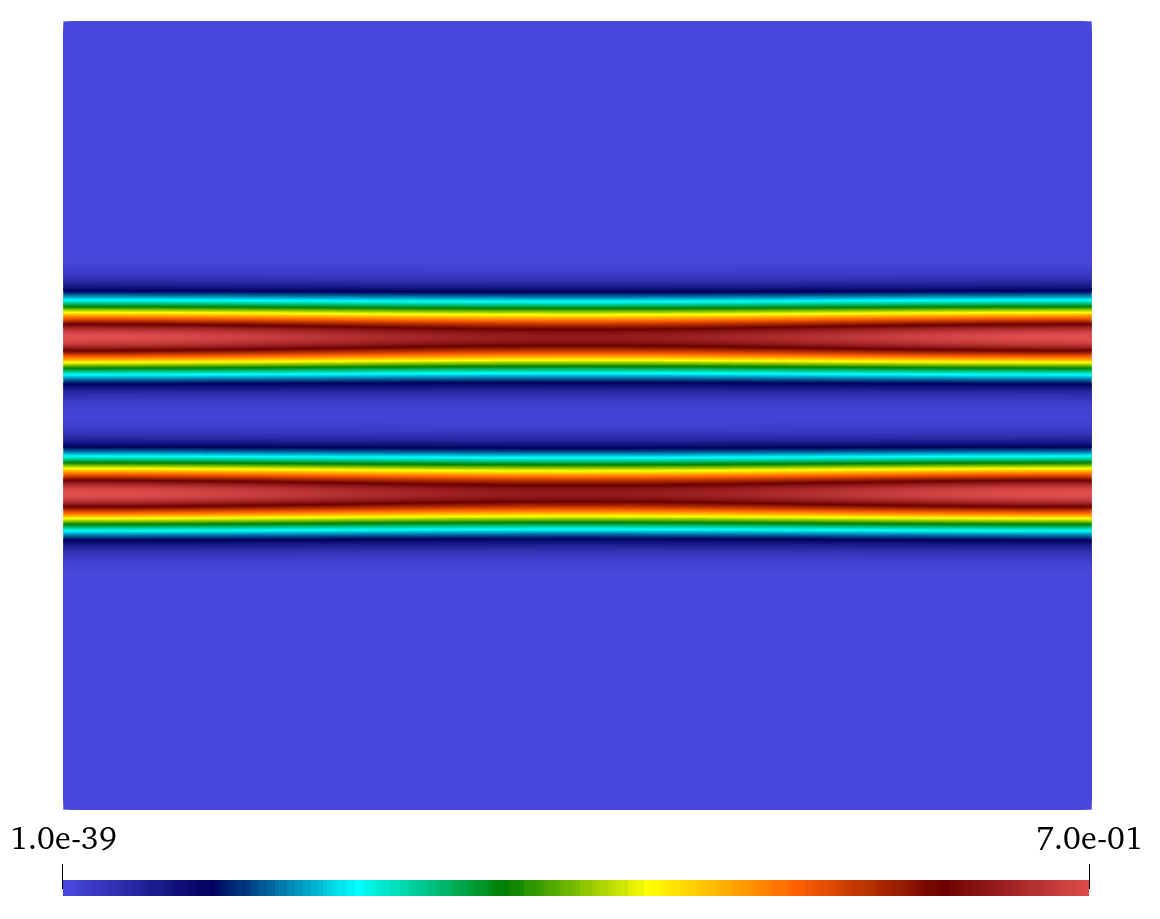}
    \caption{$t=0$}
  \end{subfigure}\hfill
  \begin{subfigure}[b]{0.5\textwidth}
    \includegraphics[width=\linewidth]{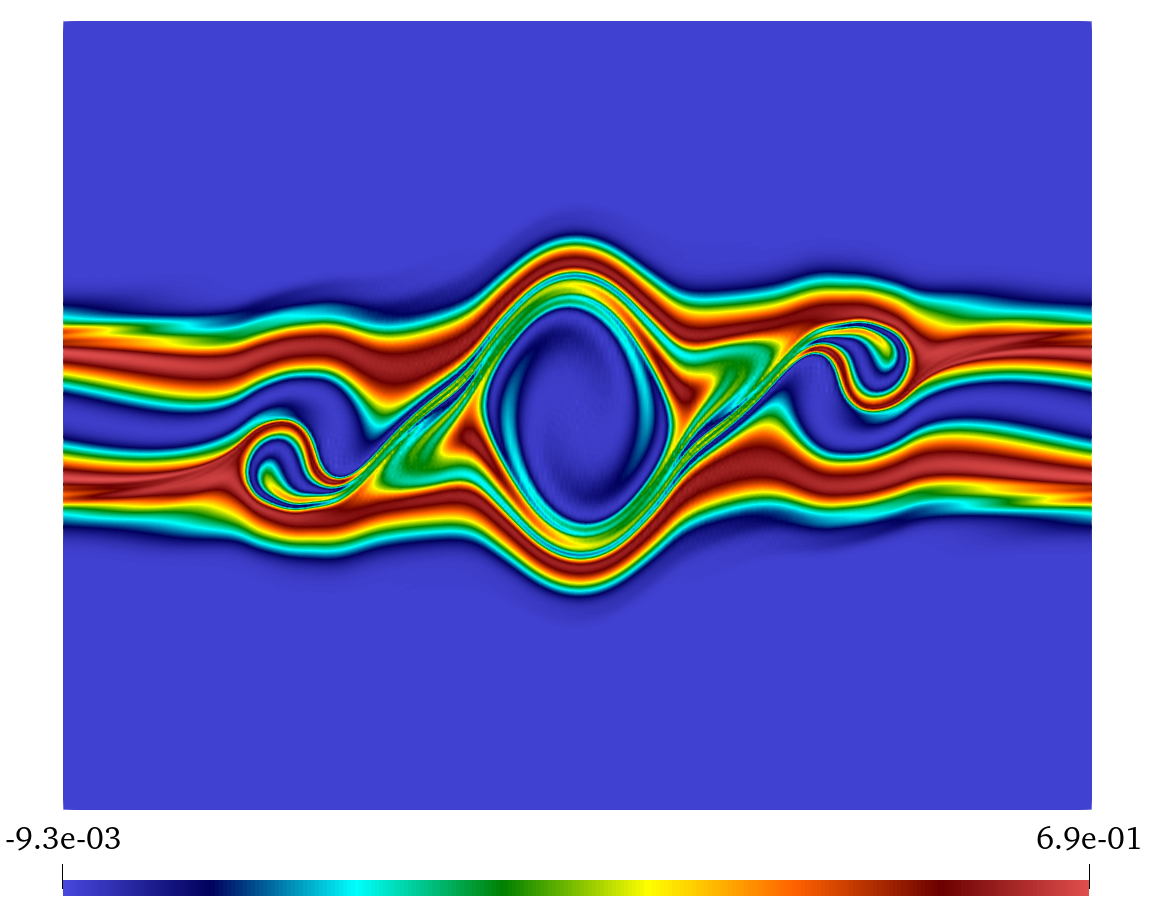}
    \caption{$t=40$}
  \end{subfigure}\hfill
  \centering
  \begin{subfigure}[b]{0.5\textwidth}
    \includegraphics[width=\linewidth]{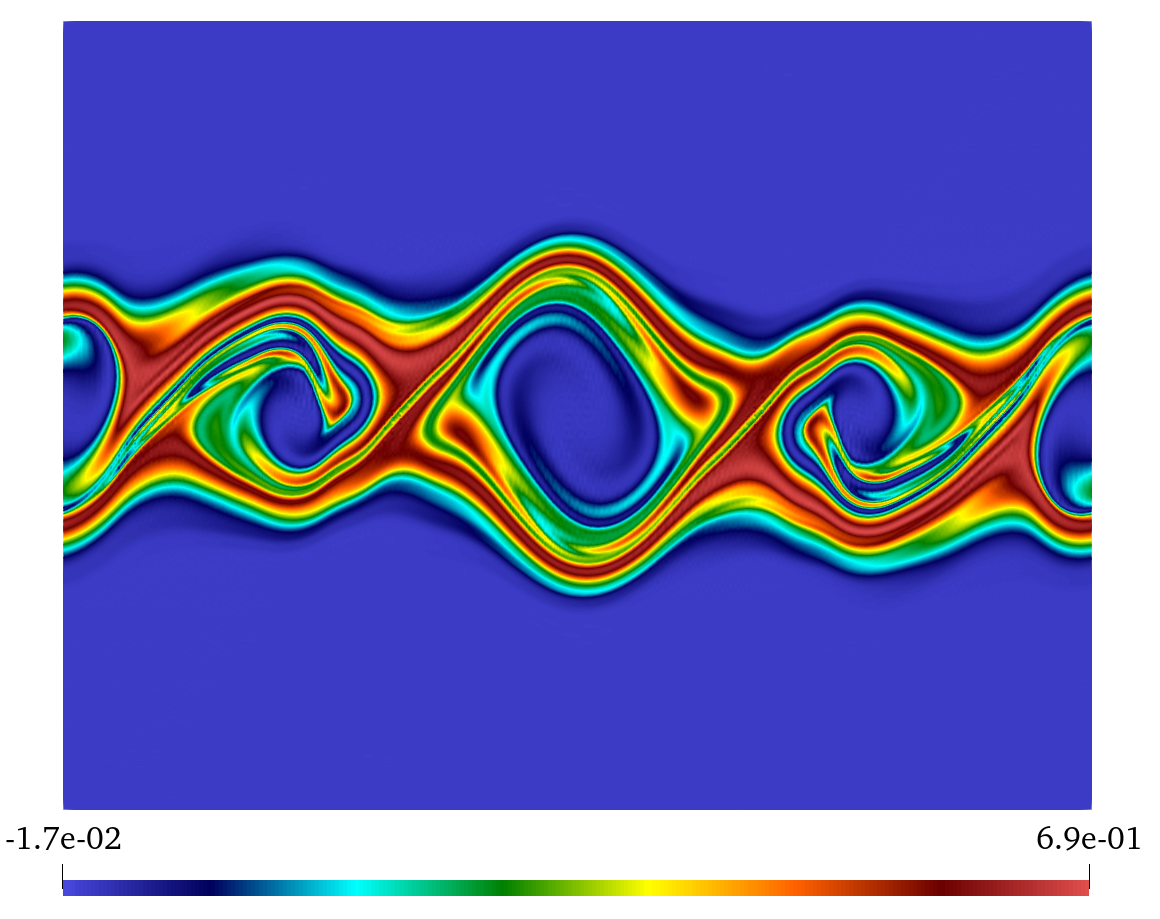}
    \caption{$t=50$}
  \end{subfigure}\hfill
  \begin{subfigure}[b]{0.5\textwidth}
    \includegraphics[width=\linewidth]{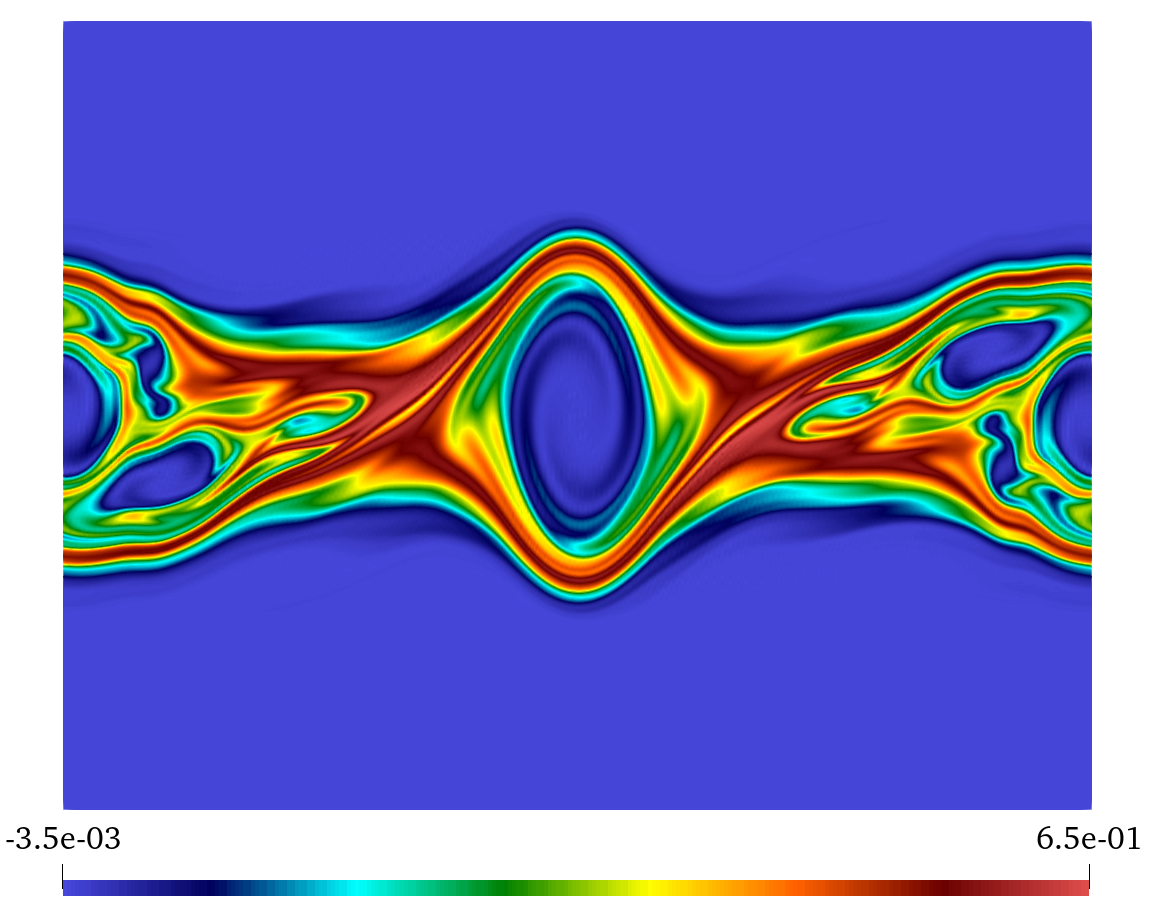}
    \caption{$t=70$}
  \end{subfigure}\hfill
  \centering
  \begin{subfigure}[b]{0.5\textwidth}
    \includegraphics[width=\linewidth]{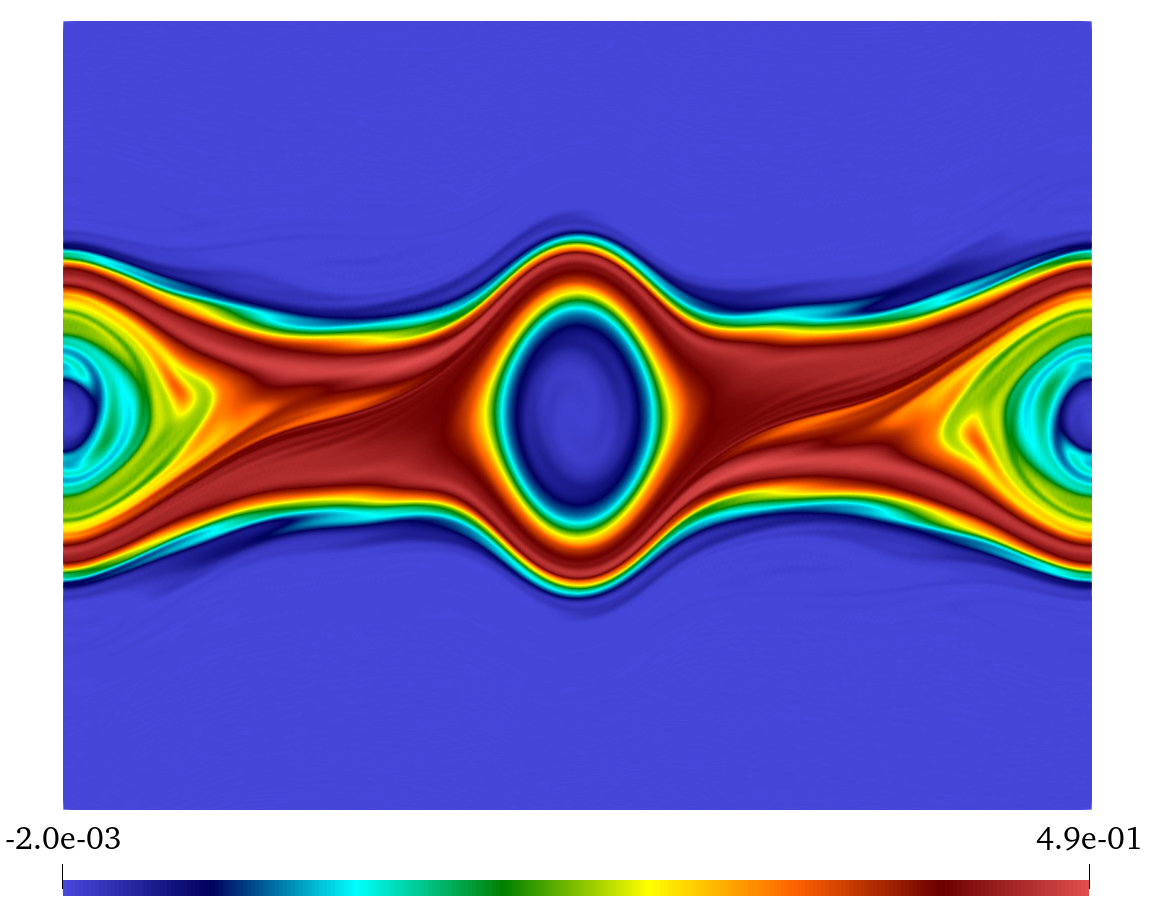}
    \caption{$t=200$}
  \end{subfigure}\hfill
  \begin{subfigure}[b]{0.5\textwidth}
    \includegraphics[width=\linewidth]{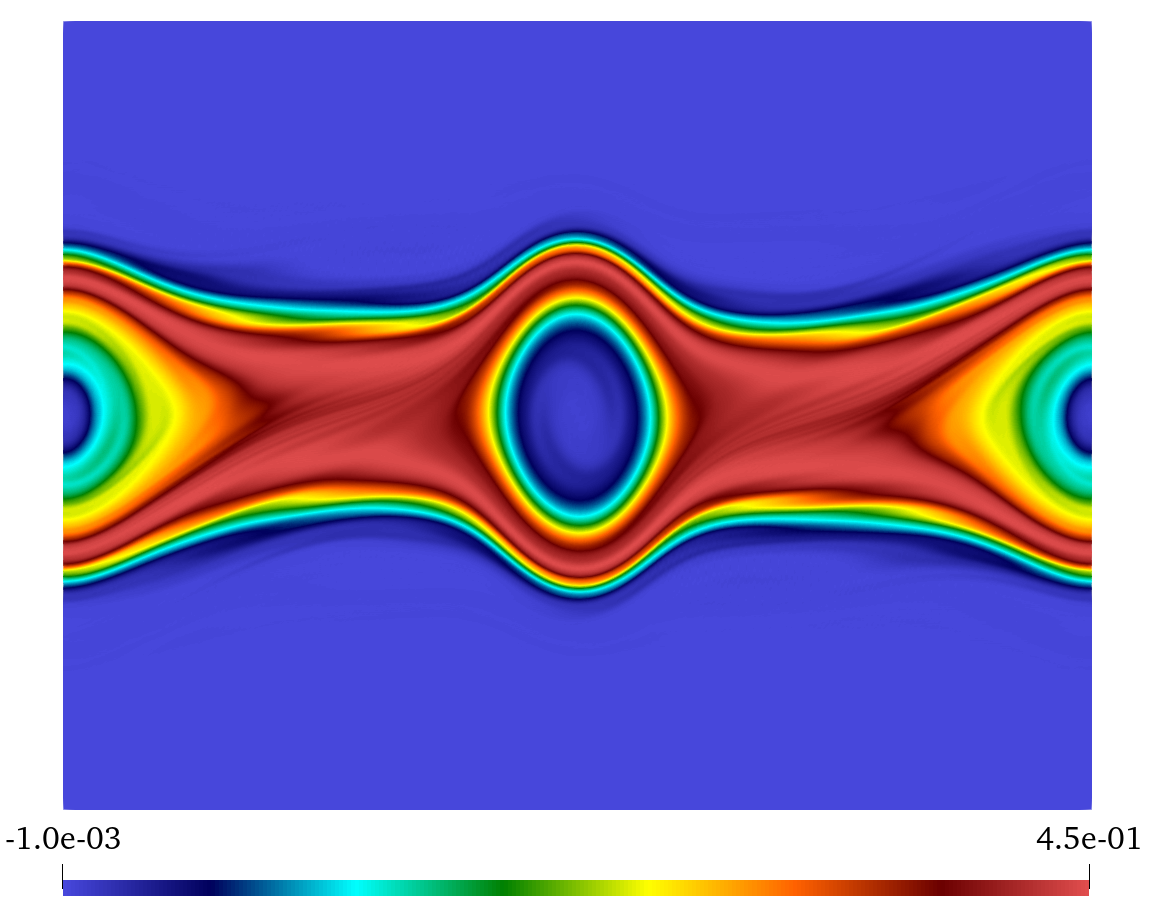}
    \caption{$t=400$}
  \end{subfigure}\hfill
  \caption{Multi-vortex two-stream instability: $\polQ_3$ solutions at $t=0, 40, 50, 70, 200,$ and $400$. The functions are plotted at the nodal points, $N_x\times N_v = 385\times769$.}
  \label{fg:B4_fh}
\end{figure}

\subsection{Bump-on-tail instability} 
Next, we investigate the bump-on-tail instability of Vlasov-Poisson equations. Consider the following initial data
\begin{equation}
  f(x,v,0)=d_0(v)(1+\alpha{\rm cos}(\theta x)),\notag
\end{equation}
in the phase space $\Omega := [0,20\pi]\times[-8,8]$, where
\begin{equation}
    d_0(v)=\frac{1}{\sqrt{2\pi}}\left(0.9{\rm exp}\left(-\frac{v^2}{2}\right)+0.2{\rm exp}(-2(v-4.5)^2)\right),\notag
\end{equation}
and $\alpha=0.04$, $\theta=0.3$. 

The $\polQ_3$ solutions with the mesh consisting of $128 \times 256$ elements are depicted in Figure \ref{fg:B3_fh}. There are two components with different thicknesses in the phase space. It is also observed that three swirls are expanding and embracing each other, and the evolution is periodic along the horizontal
direction.
\begin{figure}[htbp]
    \centering
    \begin{subfigure}[b]{0.5\textwidth}
      \includegraphics[width=\linewidth]{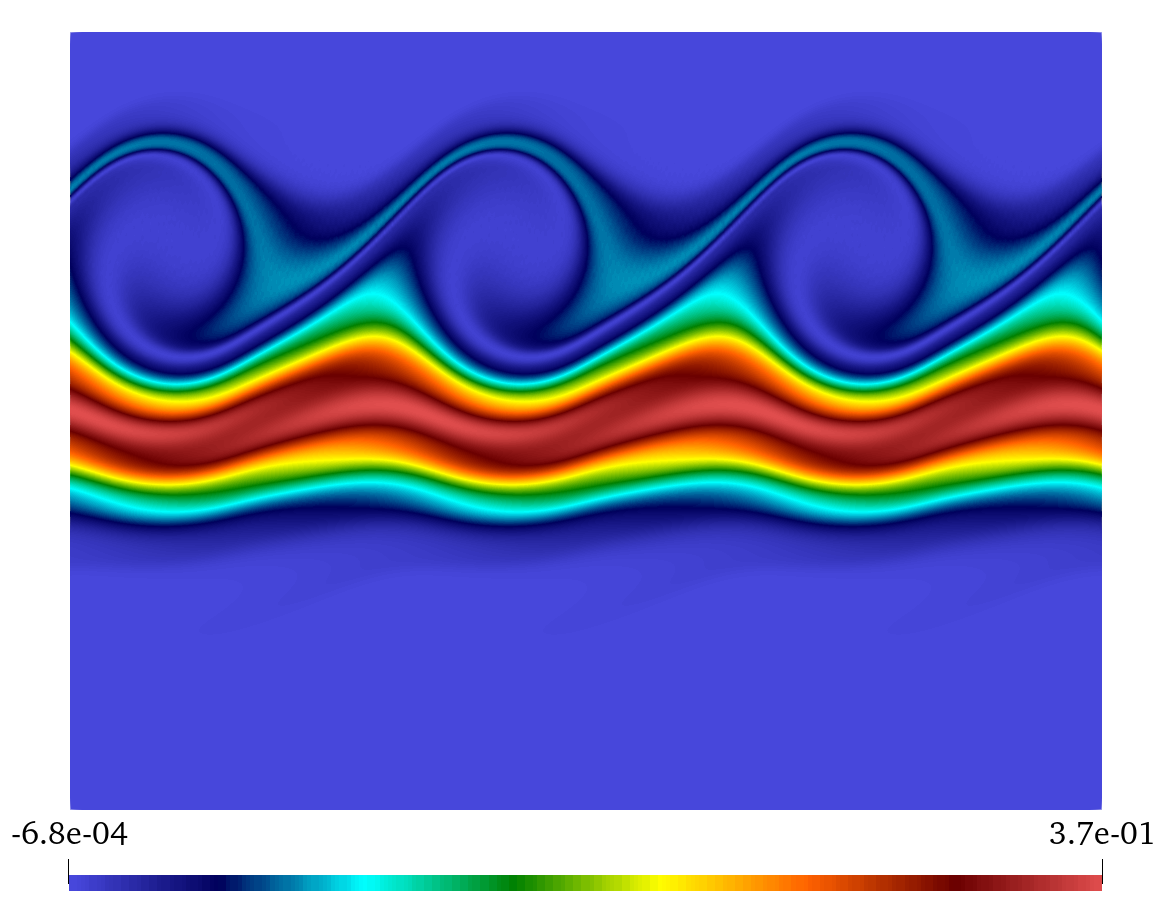}
      \caption{$t=18$}
    \end{subfigure}\hfill
    \begin{subfigure}[b]{0.5\textwidth}
      \includegraphics[width=\linewidth]{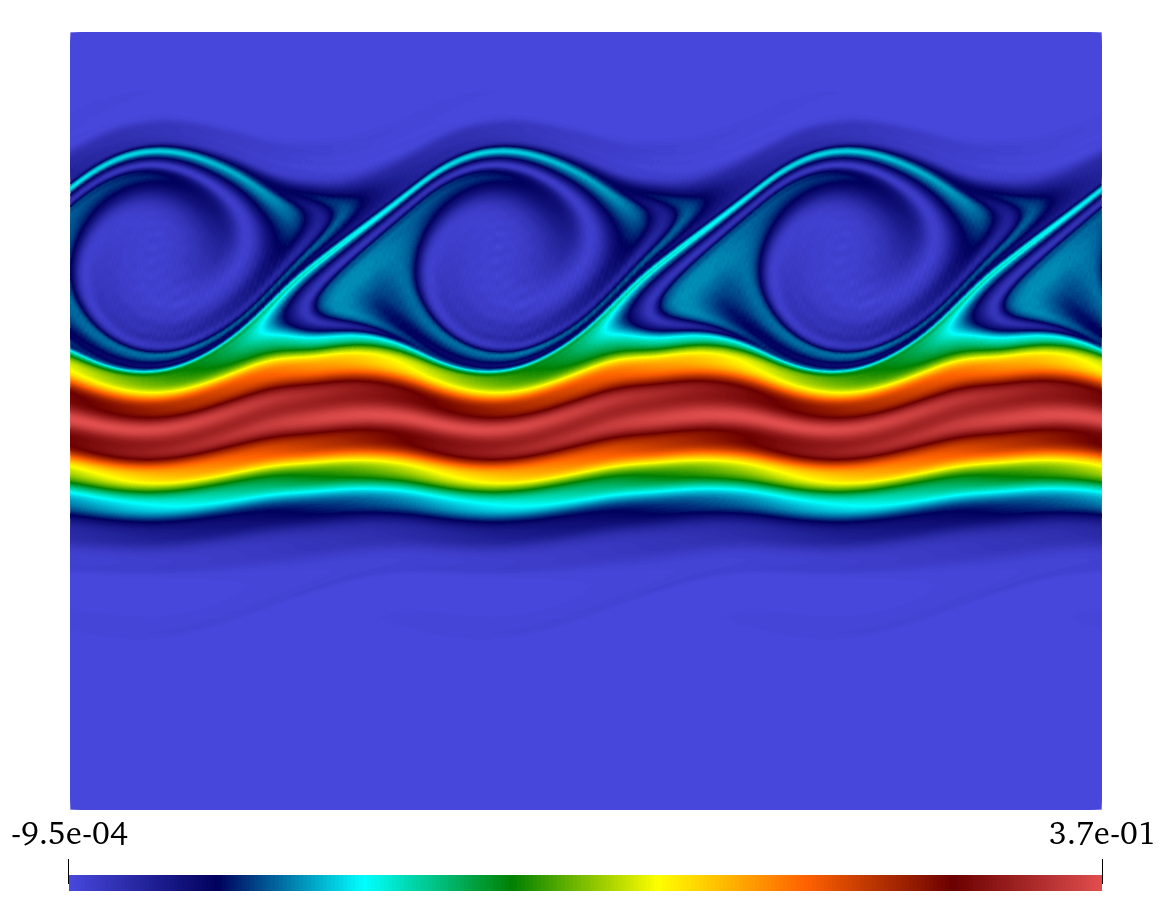}
      \caption{$t=30$}
    \end{subfigure}\hfill
    \centering
    \begin{subfigure}[b]{0.5\textwidth}
      \includegraphics[width=\linewidth]{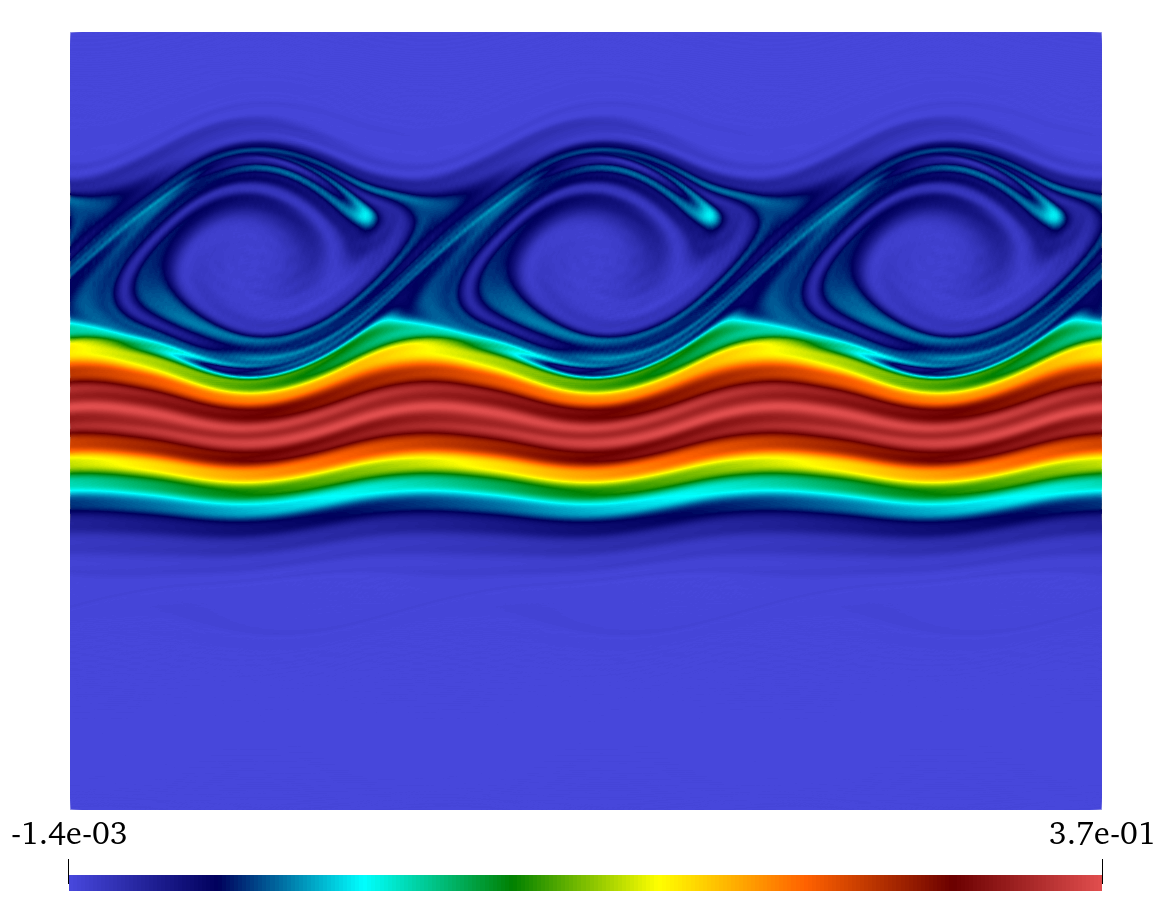}
      \caption{$t=50$}
    \end{subfigure}\hfill
    \begin{subfigure}[b]{0.5\textwidth}
      \includegraphics[width=\linewidth]{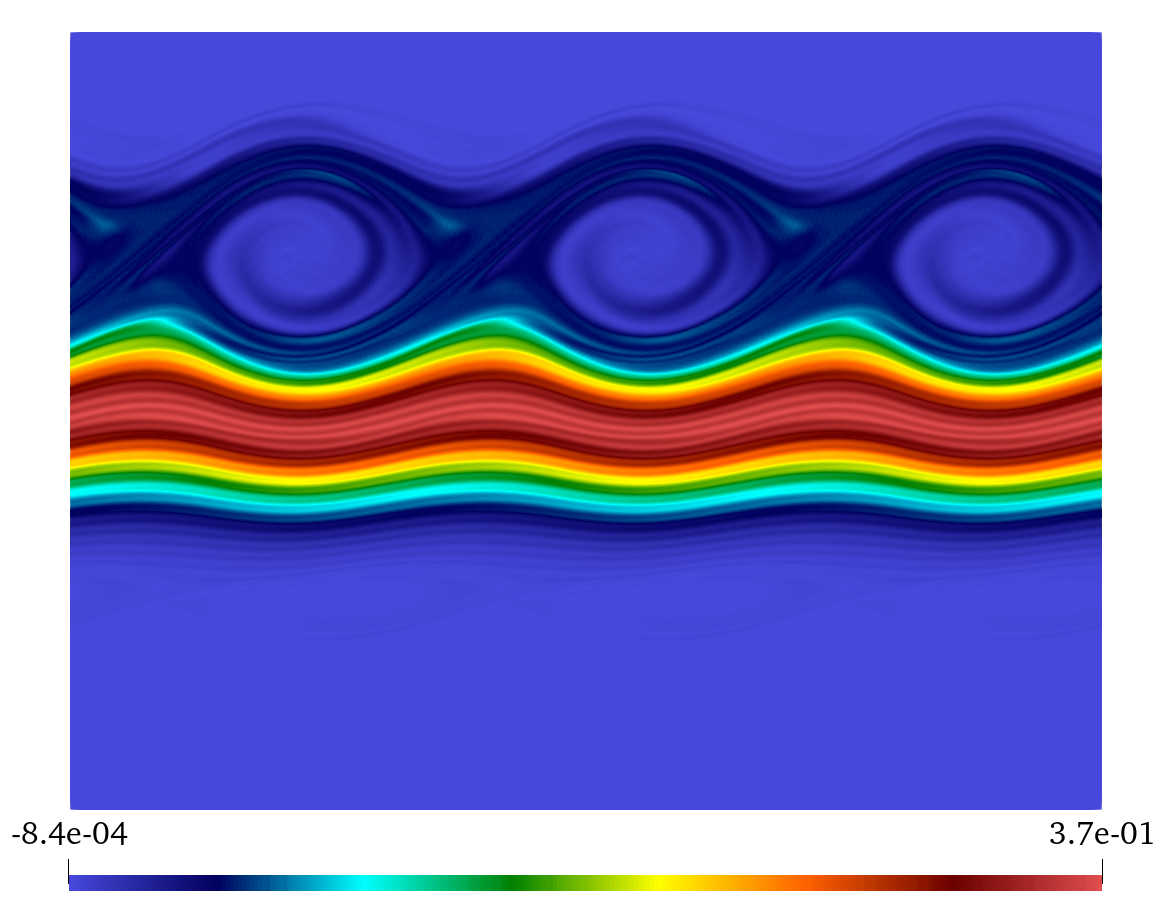}
      \caption{$t=100$}
    \end{subfigure}\hfill
    \centering
    \begin{subfigure}[b]{0.5\textwidth}
      \includegraphics[width=\linewidth]{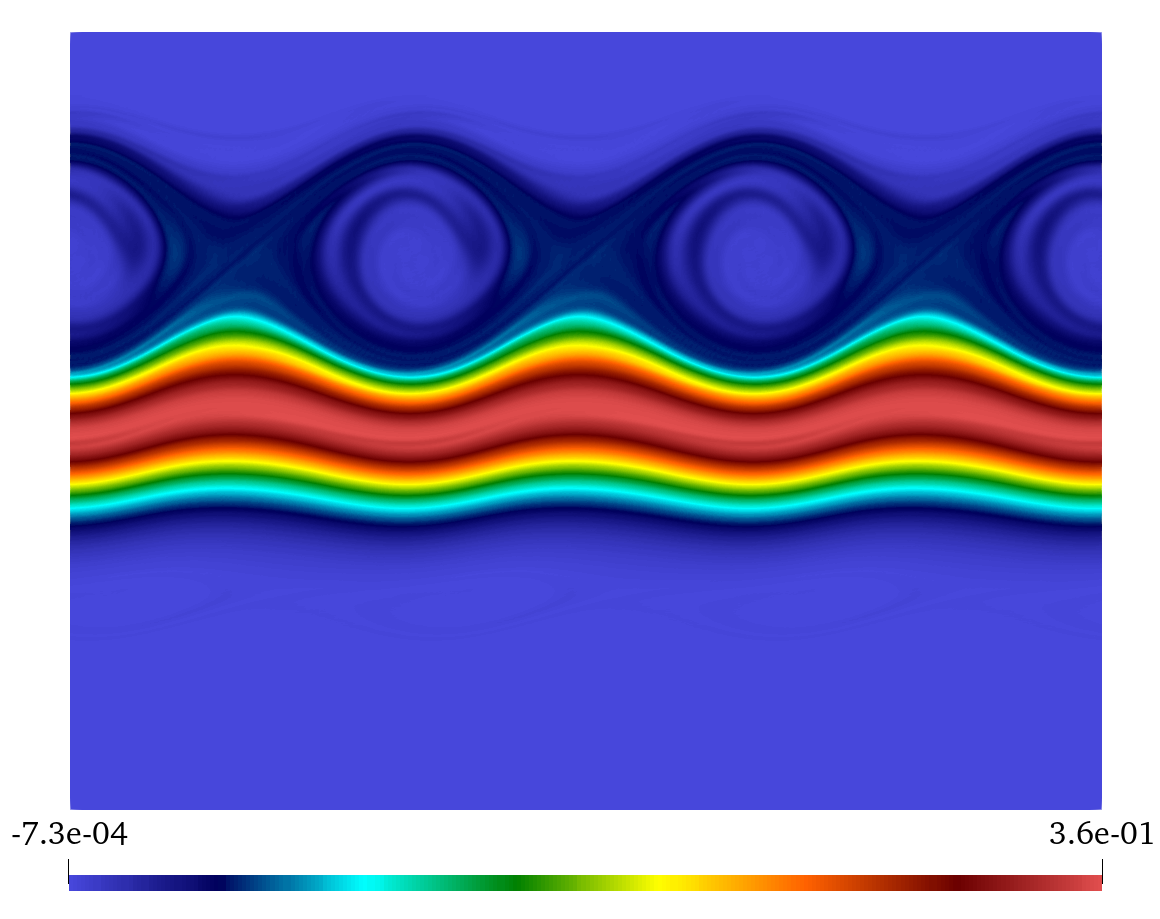}
      \caption{$t=200$}
    \end{subfigure}\hfill
    \begin{subfigure}[b]{0.5\textwidth}
      \includegraphics[width=\linewidth]{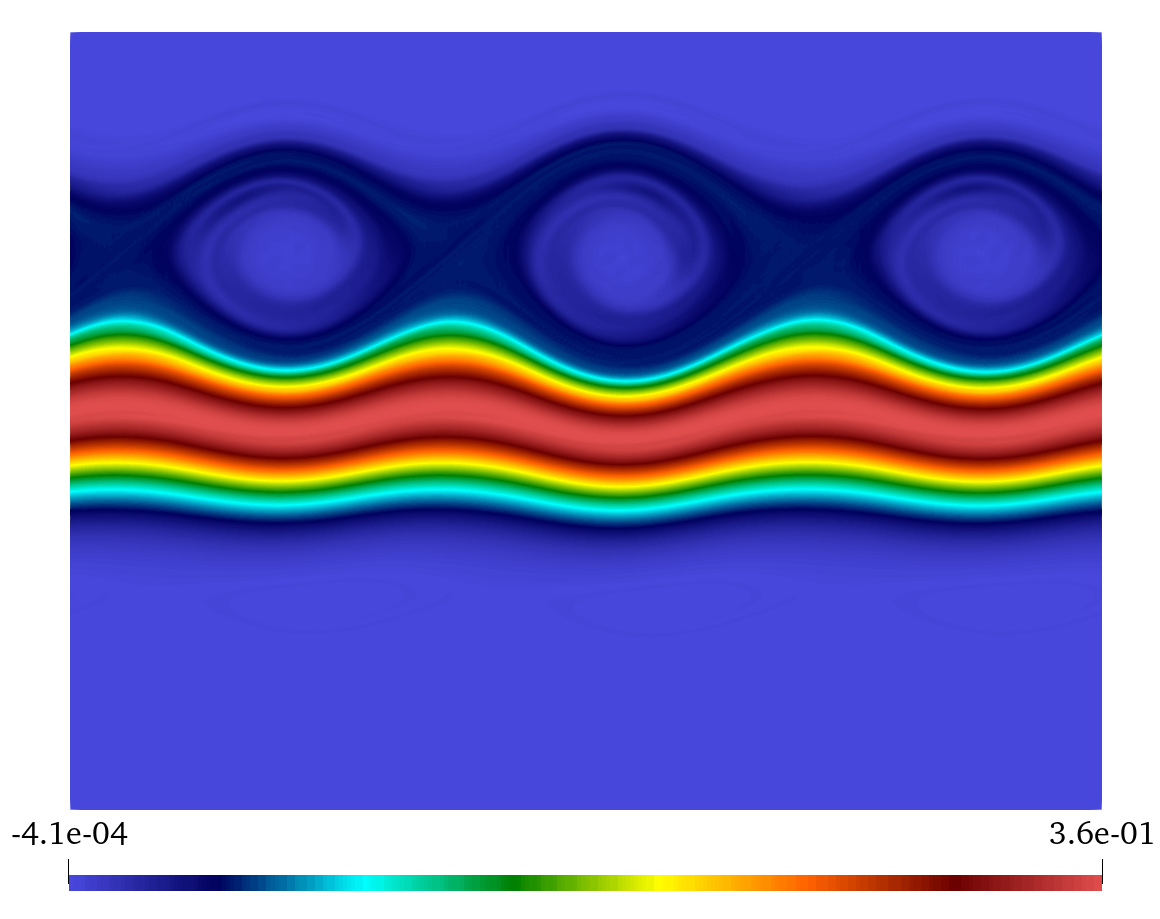}
      \caption{$t=400$}
    \end{subfigure}\hfill
    \caption{Bump-on-tail instability: $\polQ_3$ solutions at $t=18, 30, 50, 100, 200,$ and $400$. The functions are plotted at the nodal points, $N_x\times N_v = 385\times769$.}
    \label{fg:B3_fh}
\end{figure}

In addition to observing the evolution of the distribution function, it is crucial to verify that our approximations maintain the conservative properties of the system, as outlined in Proposition~\ref{prop:cons}. In Proposition~\ref{prop:conservation} we showed that mass conservation holds true for forward Euler scheme. We now verify that the total mass is conservative; furthermore, for the other properties, we expect that the deviation can be reduced through some approaches. For the $\polQ_3$ solutions, we record the values of the total mass, momentum, energy, and $L^2$-norm of $f_h$ in 
$\Omega$. For each quantity, we define its deviation as the difference from the initial value divided by the initial value. The total mass is conserved regarding some round off errors; the magnitudes of the mass deviations are extremely small, as shown in Table \ref{tab:mass}. Plot the time evolution of the deviations of the other properties in Figure \ref{fg:conserv}. the deviations of the momentum, energy, and the $L^2$-norm in our results are comparable in magnitude to those in \cite{MR2843721}, and we can minimize these deviations through mesh refinement.
\begin{figure}[htbp]
  \centering
  \begin{subfigure}[b]{0.5\textwidth}
    \includegraphics[width=\linewidth]{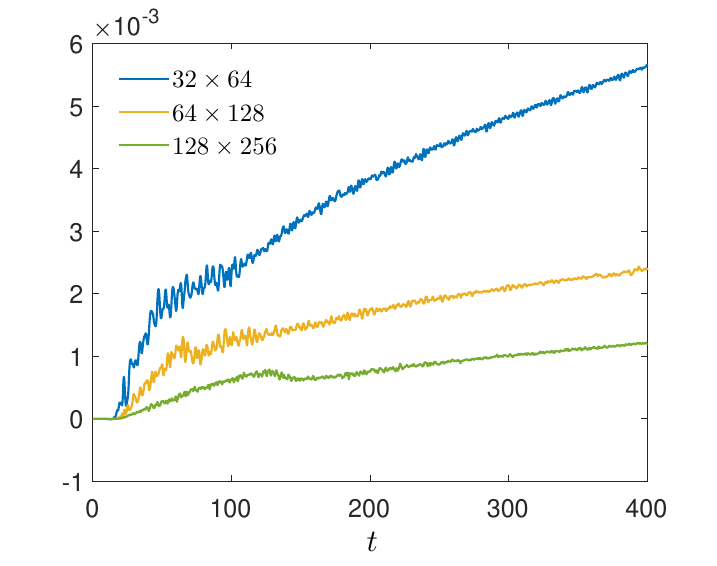}
    \caption{momentum}
  \end{subfigure}\hfill
  \begin{subfigure}[b]{0.5\textwidth}
    \includegraphics[width=\linewidth]{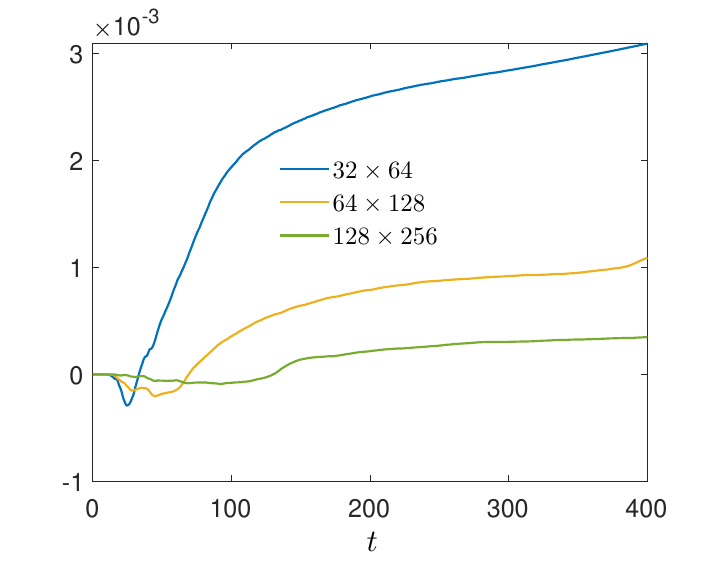} 
    \caption{energy}
  \end{subfigure}\hfill
  \begin{subfigure}[b]{0.5\textwidth}
    \includegraphics[width=\linewidth]{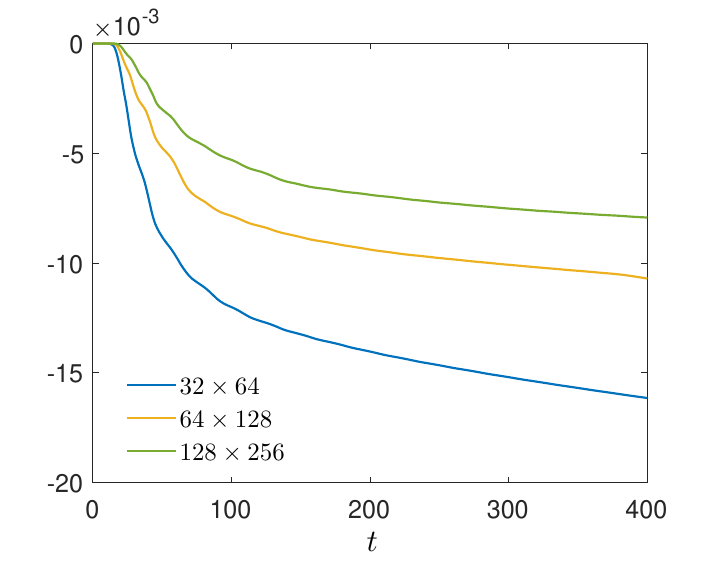} 
    \caption{$\Vert f_h\Vert_{L^2(\Omega)}$}
  \end{subfigure}\hfill
  \caption{Bump-on-tail instability: time evolution of the deviations when different DOFs are used. The element is $\polQ_3$ and we choose three grids of increasing density.}
  \label{fg:conserv}
\end{figure}

Finally, we solve this problem using two different viscosity coefficients, as discussed in Remark~\ref{remark:isotropic}. A mesh of $50 \times 300$ elements is employed, using the $\polQ_3$ polynomial space. The results at the final time are shown in Figure~\ref{fg:B3_min_fh}. It is evident that when the anisotropy ratio in the mesh is too large, the isotropic viscosity fails to provide sufficient damping for the spurious oscillations that develop in the solution. This clearly demonstrates the advantages of the anisotropic viscosity method proposed in this paper.
 
\begin{figure}[htbp]
  \centering
  \begin{subfigure}[b]{0.4\textwidth}
    \includegraphics[width=\linewidth]{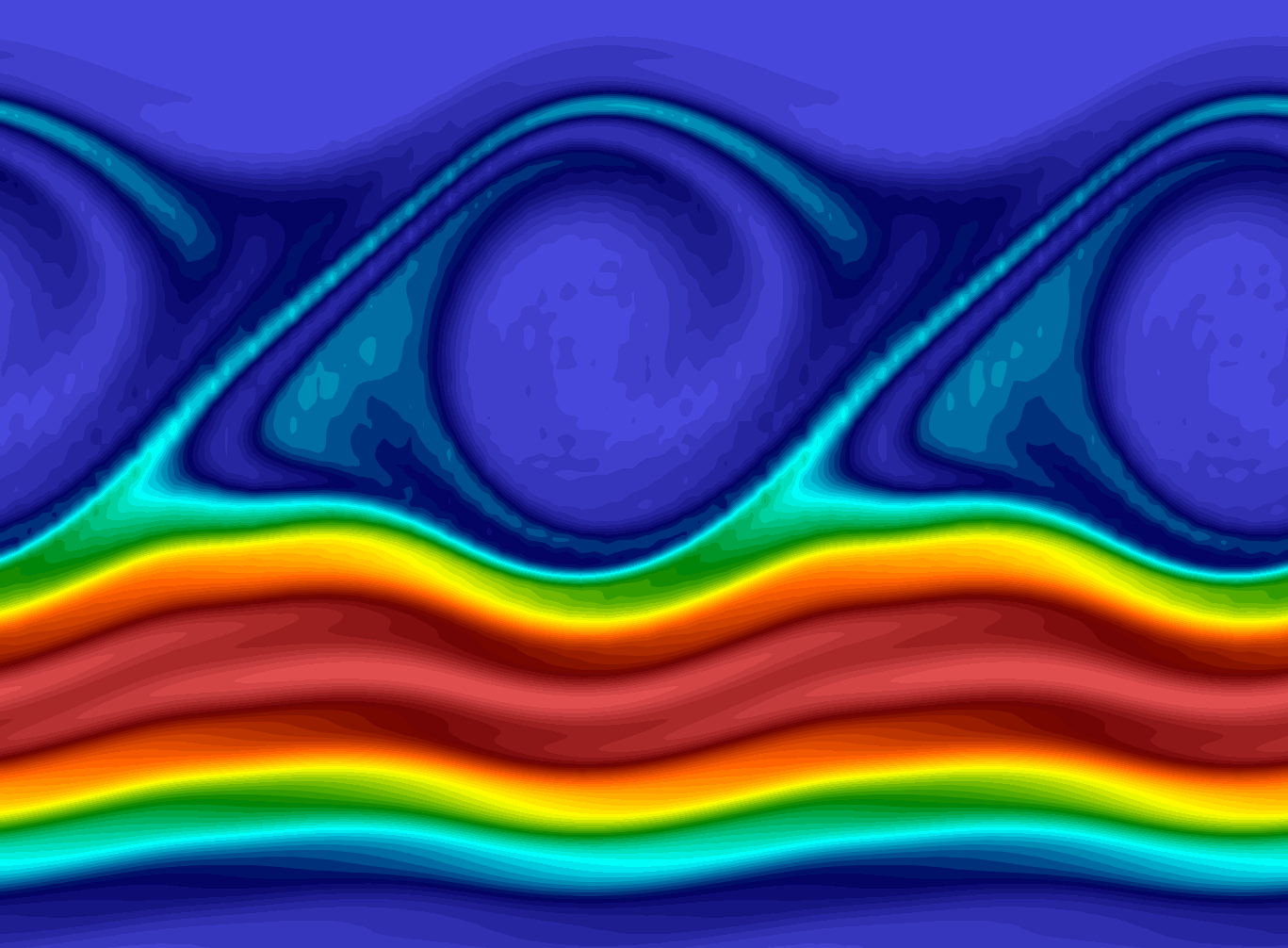}
    \caption{Anisotropic viscosity}
  \end{subfigure}
  \hspace{0.2in}  
  \begin{subfigure}[b]{0.4\textwidth}
    \includegraphics[width=\linewidth]{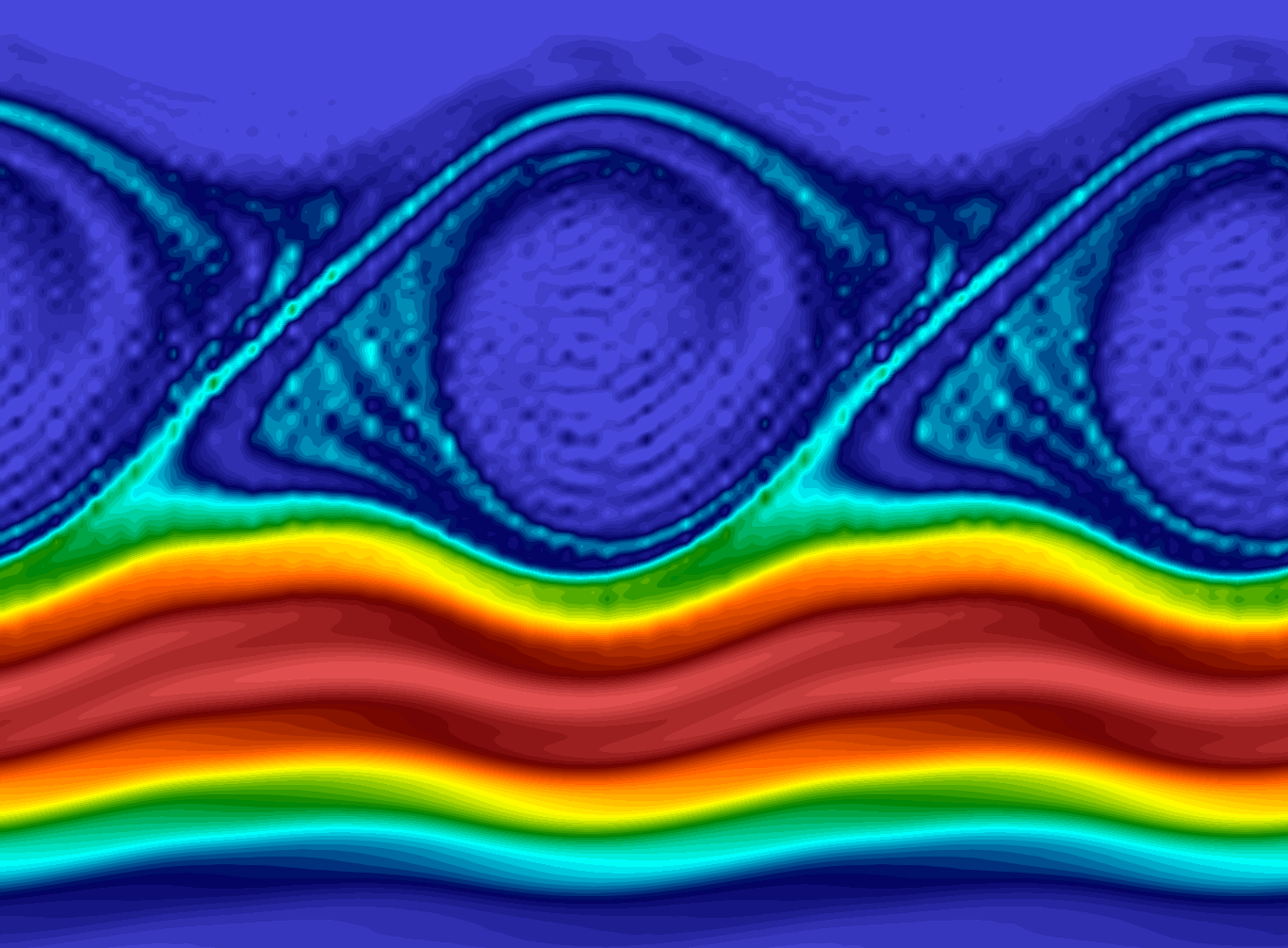}
    \caption{Isotropic viscosity}
  \end{subfigure}
  \caption{Bump-on-tail instability: $\polQ_3$ solutions at $t=30$ are depicted using two artificial viscosity coefficients, the left panel is anisotropic, and the right panel is isotropic, $N_x\times N_v = 151\times901$.}
  \label{fg:B3_min_fh}
\end{figure}

\begin{table}[htbp]
    \caption{Bump-on-tail instability: maximum deviations of total mass until $t=400$. The element is $\polQ_3$ and we choose three grids of increasing density.} \label{tab:mass}
  \begin{center}
    \begin{tabular}{c|c} \hline
     $\#{\rm elements}$ & deviation  \\ \hline
     $32\times64$ & 4.44E-15 \\
     $64\times128$ & 1.15E-14\\
     $128\times256$ &2.91E-14\\ \hline
    \end{tabular}
  \end{center}
\end{table}

\section{Discussion}\label{sec:dis}
In this article, we present a framework for finite element methods to solve the Vlasov–Poisson equations. The discretization process consists of the following steps. First, we obtain a semi-discretization of the Vlasov equation using the continuous finite element method, which achieves high convergence rates with high-order polynomial basis functions and preserves the total mass of the Vlasov-Poisson system. To handle numerical oscillations, we introduce an anisotropic residual-based artificial viscosity scheme that remains mass-conservative. Second, the system is fully discretized using a fourth-order explicit strong-stability-preserving Runge-Kutta scheme, with the electric field updated at each Runge-Kutta stage, and the Poisson equation solved using continuous finite element methods.

The convergence study shows that the finite element method achieves fourth-order accuracy for smooth problems when using third-order Lagrange polynomials as the basis functions. The residual viscosity (RV) method approximates the Vlasov-Poisson equation without introducing numerical oscillations. Additionally, the residual-based artificial viscosity technique preserves the high accuracy of the standard finite element method. The methods are validated using classic benchmark problems of the Vlasov-Poisson equations in two dimensions. For Landau damping, the numerical solutions produced by our methods accurately predict the damping rate, and the artificial viscosity effectively handles the strong shocks in the solutions. Furthermore, the methods successfully simulate the plasma distribution in the two-stream instability and bump-on-tail instability phenomena. We also confirm that our method strictly conserves mass. However, other properties, such as energy, are not conserved by our methods, although the deviations in these quantities decrease as the mesh is refined.

There are several promising avenues for future research to expand on this work. First, the implementation of our methods in higher-dimensional settings requires further investigation. The main challenges involve storage and computation, as the data and computational cost grow exponentially in higher dimensions. A possible approach is to combine our methods with matrix-free algorithms, such as those introduced in \cite{MR4573597}, to reduce storage requirements and accelerate computations. Second, our methods do not preserve positivity, which is an important property for plasma distributions. Finally, the methods introduced in this article could be extended to the Vlasov-Maxwell equations, which are a more complex variant of the Vlasov-Poisson system. This model couples the Vlasov equation with the Maxwell equations, introducing additional complexity. One potential approach is to combine the continuous finite element function space from our methods with a structure-preserving space for the Maxwell equations.

\section*{CRediT authorship contribution statement}
\pmb{Junjie Wen}: Writing-original draft, Writing–review $\&$ editing, Visualization, Validation, Software, Methodology, Formal
analysis, Investigation. \pmb{Murtazo Nazarov}: Conceptualization, Formal analysis, Funding acquisition, Investigation, Methodology, Software, Supervision,  Writing–review $\&$ editing. 

\section*{Declaration of competing interest}
The authors declare the following financial interests/personal relationships which may be considered as potential competing
interests: Murtazo Nazarov reports financial support was provided by Swedish Research Council (VR). If there are other authors, they
declare that they have no known competing financial interests or personal relationships that could have appeared to influence the
work reported in this paper.

\section*{Acknowledgements}
The first and second authors are financially supported by the Swedish Research Council (VR) under grant numbers 2021-04620 and 2021-05095, which is gratefully acknowledged. Open access funding is provided by Uppsala University.

\appendix
\section{Guiding-center model}\label{Ap:GCM}
The guiding-center model is a variant of Vlasov-Poisson equations. This model describes the density evolution of a highly magnetized plasma in the transverse plane of a
tokamak. In this system, the charge density is governed by the following advection equation
  \begin{equation}
    \partial_t\rho(\pmb{x},t)+\pmb{E}^{\perp}(\pmb{x},t)\cdot\nabla_{\pmb{x}}\rho(\pmb{x},t)= 0,\notag
  \end{equation}
where the advection field is $\pmb{E}^{\perp} = (E_2(x_1,x_2),-E_1(x_1,x_2))^\mathsf{T}$, when we consider the system in two-dimensional setting. As in the Vlasov-Poisson equation, the electric field is determined by the charge density via the Poisson equation:
\begin{equation}
  -\nabla^2_{\pmb{x}}\Phi=\rho,\qquad\pmb{E} = -\nabla_{\pmb{x}}\Phi.\notag
\end{equation}

This system is similar to the Vlasov-Poisson equation, but unlike Vlasov-Poisson, all components of the advection field depend on all spatial coordinates of the phase space. 
Hence, the splitting technique commonly used for Vlasov-Poisson equations is difficult to apply in this case, as introduced in \cite{MR4844786, MR2586230, MR1672731}. Our methods do not rely on any splitting process, making this an suitable test case for their applicability.
We test our methods use the following initial data
\begin{equation}
  \rho(x_1,x_2,0)= \left\{
    \begin{aligned}
      &(1+\alpha\cos(\ell\theta))\exp(-4(r-6.5)^2),\quad&&{\rm if}\ r^-\leq r\leq r^+, \\
      &0, &&{\rm otherwise,}
    \end{aligned}
  \right.\notag
\end{equation}
where $\alpha=0.1$, $\ell = 6$, $\theta = {\rm atan}2(x_2,x_1)$, $r^+=8$, $r^-=5$, and $r = \sqrt{x_1^2+x_2^2}$. The domain is $\Omega_{\pmb{x}}:=[-15,15]\times[-15,15]$, and we impose periodic boundary conditions on $\rho$ and $\Phi$.
The $\polQ_3$ solutions in the mesh consisting of $128\times 128$ elements are plotted in Figure \ref{fig:GCM}. We observe the presence of vortices that gradually merge, similar to the findings reported in \cite{MR4844786, MR3517454}.

\begin{figure}[htbp]
  \centering
  \begin{subfigure}[b]{0.33\textwidth}
    \includegraphics[width=\linewidth]{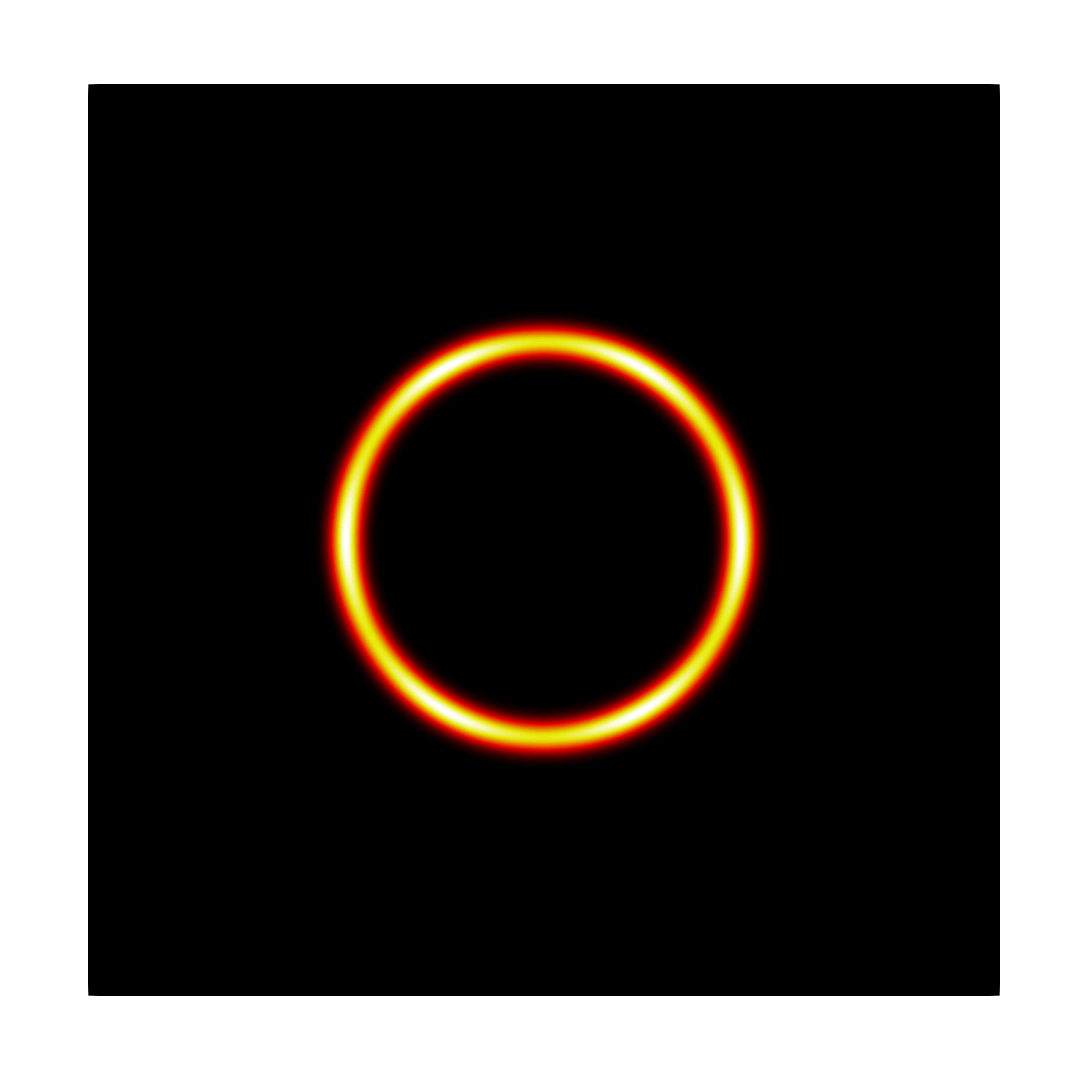}
    \caption{$t=0$}
  \end{subfigure}\hfill
  \begin{subfigure}[b]{0.33\textwidth}
    \includegraphics[width=\linewidth]{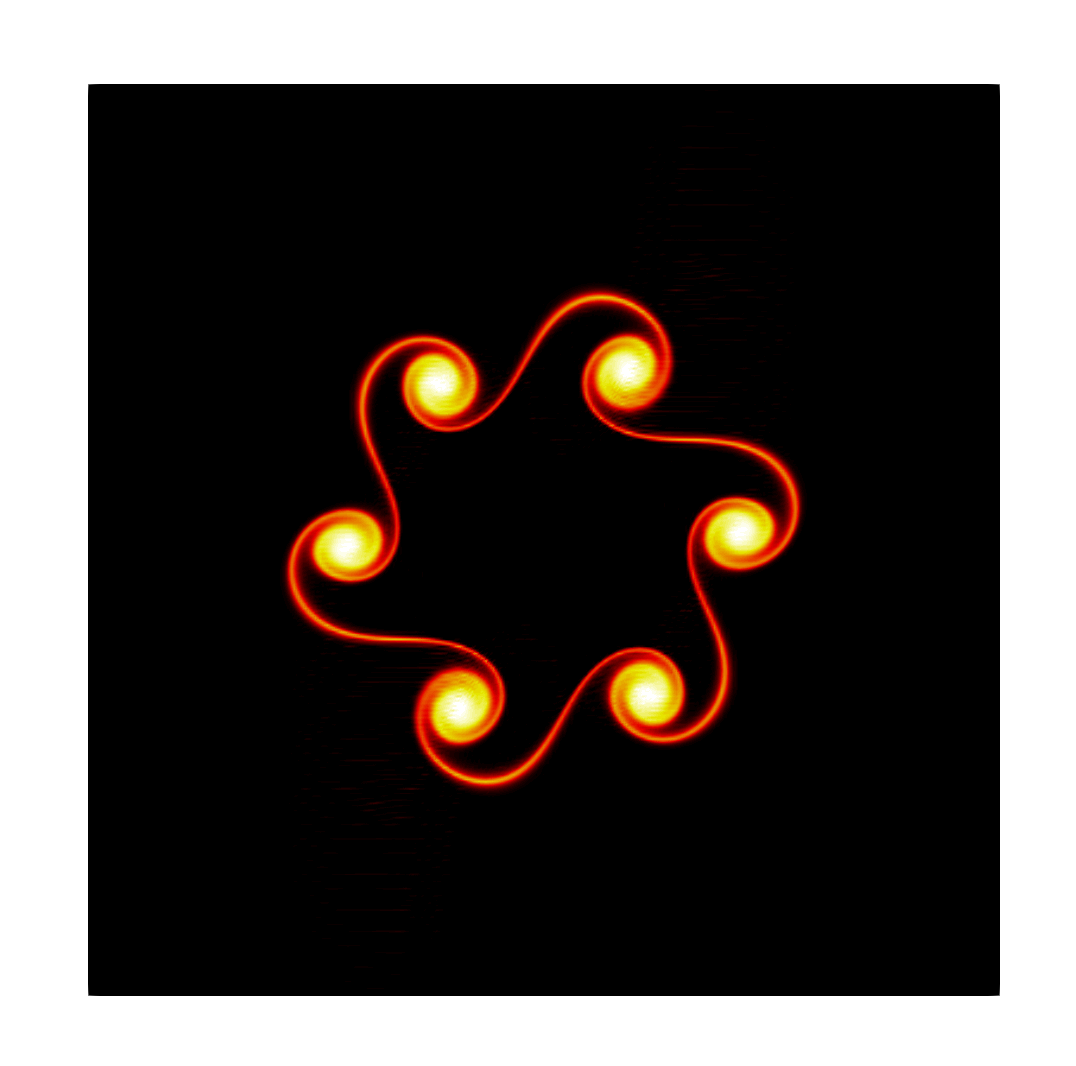}
    \caption{$t=25$}
  \end{subfigure}\hfill
  \begin{subfigure}[b]{0.33\textwidth}
    \includegraphics[width=\linewidth]{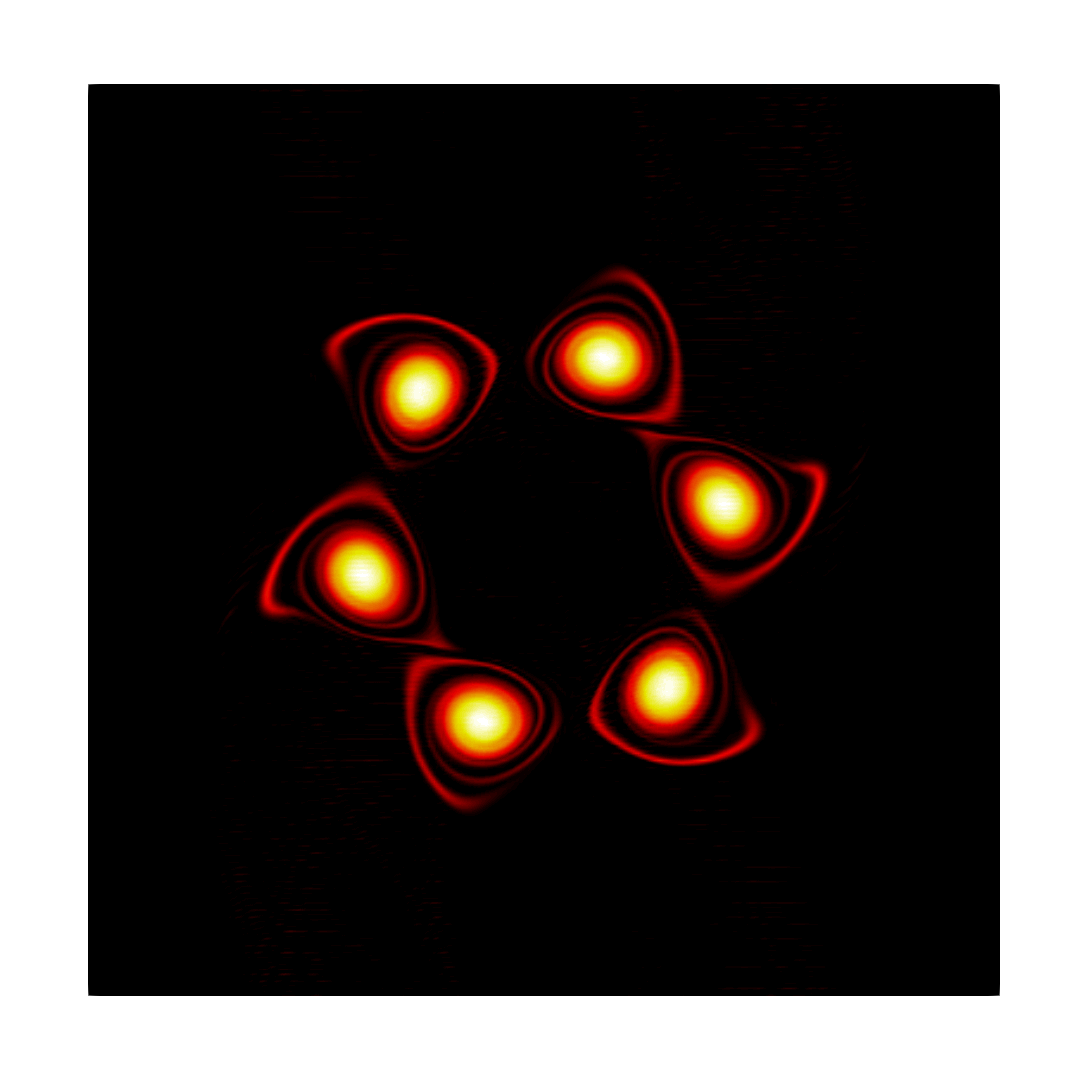}
    \caption{$t=50$}
  \end{subfigure}\hfill
  \centering
  \begin{subfigure}[b]{0.33\textwidth}
    \includegraphics[width=\linewidth]{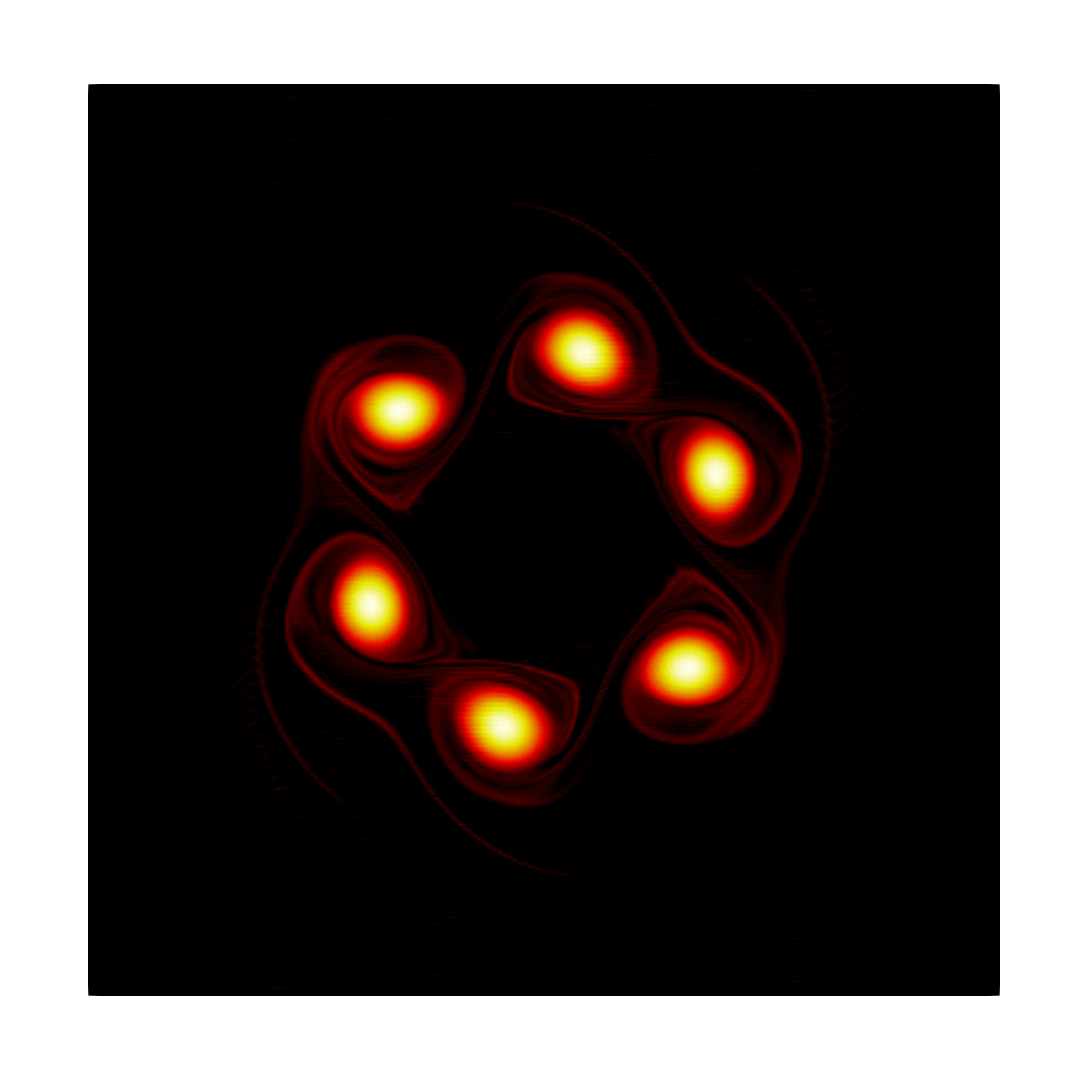}
    \caption{$t=100$}
  \end{subfigure}\hfill
  \begin{subfigure}[b]{0.33\textwidth}
    \includegraphics[width=\linewidth]{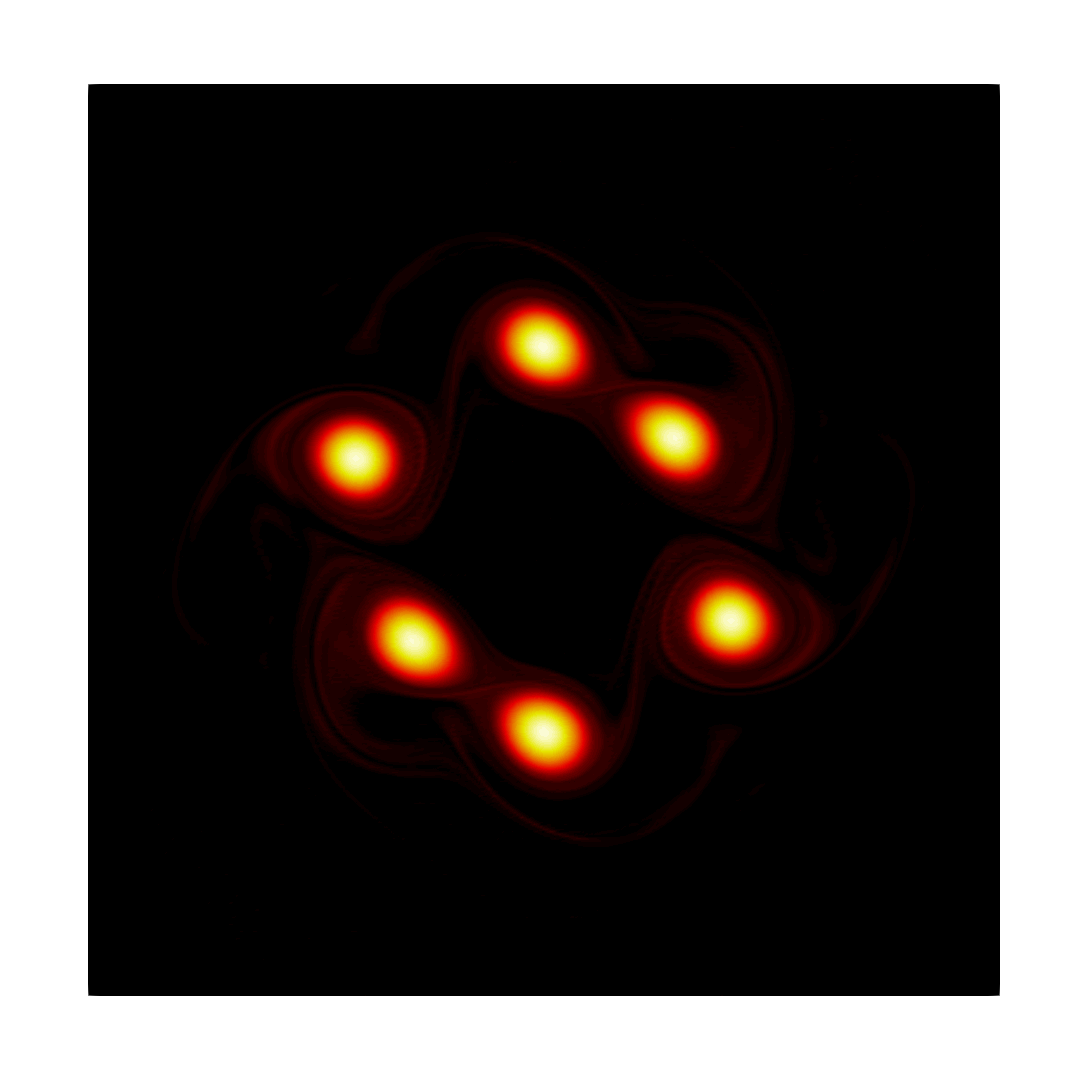}
    \caption{$t=200$}
  \end{subfigure}\hfill
  \begin{subfigure}[b]{0.33\textwidth}
    \includegraphics[width=\linewidth]{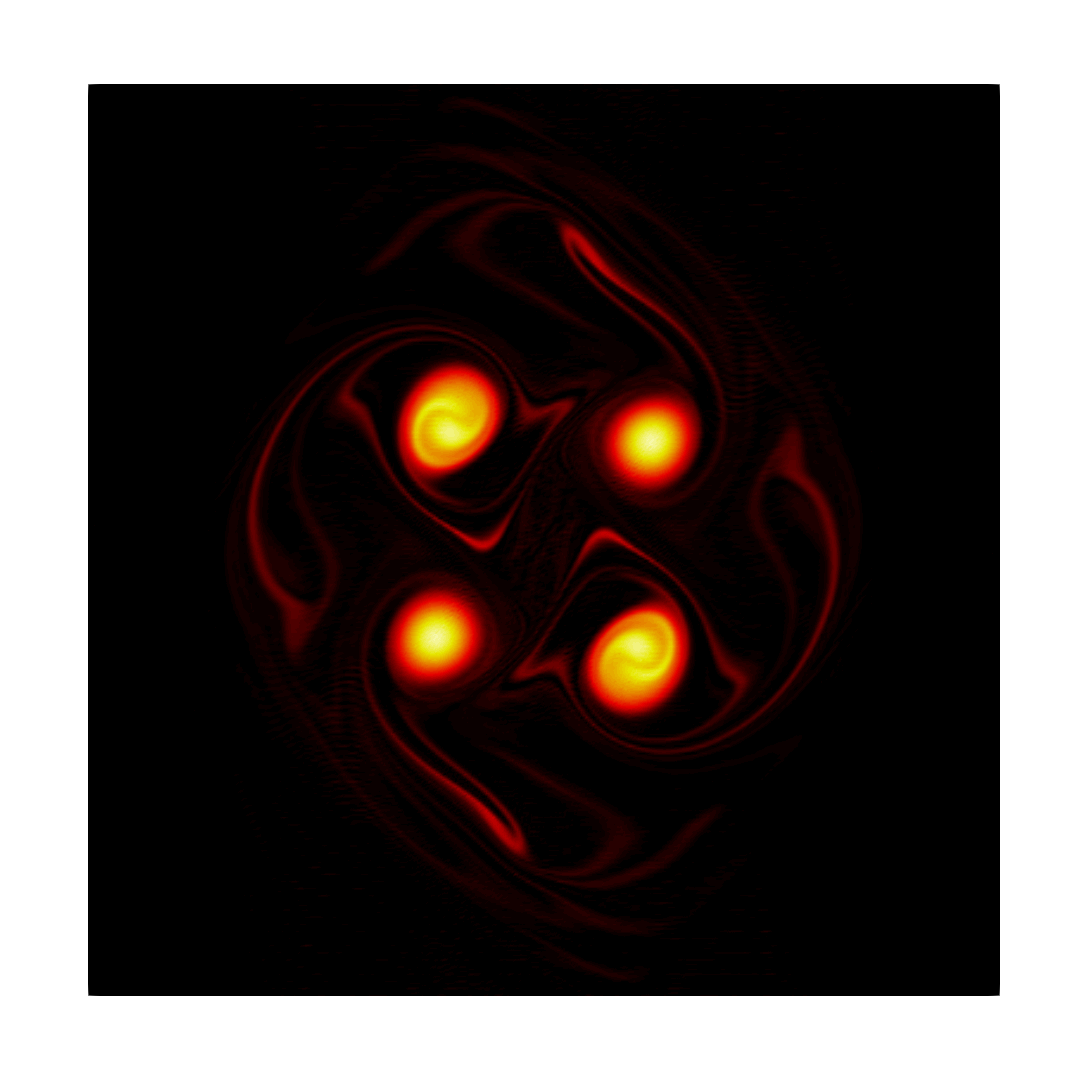}
    \caption{$t=500$}
  \end{subfigure}\hfill
  \caption{Guiding-center model: $\polQ_3$ solutions at $t=0, 25, 50, 100, 200,$ and $500$. The functions are plotted at the nodal points, $N_{x_1}\times N_{x_2} = 385\times385$.}
  \label{fig:GCM}
\end{figure}

\bibliographystyle{bib}
\bibliography{ref}

\begin{thebibliography}{10}
\expandafter\ifx\csname url\endcsname\relax
  \def\url#1{\texttt{#1}}\fi
\expandafter\ifx\csname urlprefix\endcsname\relax\def\urlprefix{URL }\fi
\expandafter\ifx\csname href\endcsname\relax
  \def\href#1#2{#2} \def\path#1{#1}\fi

\bibitem{TAYLOR1967155}
J.~Taylor,
  \href{https://www.sciencedirect.com/science/article/pii/0032063367900736}{Disturbance
  of a rarefied plasma by a supersonic body on the basis of the poisson-vlasov
  equations-1: The heuristic method}, Planetary and Space Science 15~(1) (1967)
  155--187.
\newblock \href {https://doi.org/https://doi.org/10.1016/0032-0633(67)90073-6}
  {\path{doi:https://doi.org/10.1016/0032-0633(67)90073-6}}.
\newline\urlprefix\url{https://www.sciencedirect.com/science/article/pii/0032063367900736}

\bibitem{MR1119268}
H.~D. Victory, Jr., E.~J. Allen, \href{https://doi.org/10.1137/0728065}{The
  convergence theory of particle-in-cell methods for multidimensional
  {V}lasov-{P}oisson systems}, SIAM J. Numer. Anal. 28~(5) (1991) 1207--1241.
\newblock \href {https://doi.org/10.1137/0728065} {\path{doi:10.1137/0728065}}.
\newline\urlprefix\url{https://doi.org/10.1137/0728065}

\bibitem{DEGOND20105630}
P.~Degond, F.~Deluzet, L.~Navoret, A.-B. Sun, M.-H. Vignal,
  \href{https://www.sciencedirect.com/science/article/pii/S0021999110001774}{Asymptotic-preserving
  particle-in-cell method for the vlasov–poisson system near
  quasineutrality}, Journal of Computational Physics 229~(16) (2010)
  5630--5652.
\newblock \href {https://doi.org/https://doi.org/10.1016/j.jcp.2010.04.001}
  {\path{doi:https://doi.org/10.1016/j.jcp.2010.04.001}}.
\newline\urlprefix\url{https://www.sciencedirect.com/science/article/pii/S0021999110001774}

\bibitem{MR2818633}
G.~Chen, L.~Chac\'on, D.~C. Barnes,
  \href{https://doi.org/10.1016/j.jcp.2011.05.031}{An energy- and
  charge-conserving, implicit, electrostatic particle-in-cell algorithm}, J.
  Comput. Phys. 230~(18) (2011) 7018--7036.
\newblock \href {https://doi.org/10.1016/j.jcp.2011.05.031}
  {\path{doi:10.1016/j.jcp.2011.05.031}}.
\newline\urlprefix\url{https://doi.org/10.1016/j.jcp.2011.05.031}

\bibitem{MR3485969}
F.~Filbet, L.~M. Rodrigues,
  \href{https://doi.org/10.1137/15M104952X}{Asymptotically stable
  particle-in-cell methods for the {V}lasov-{P}oisson system with a strong
  external magnetic field}, SIAM J. Numer. Anal. 54~(2) (2016) 1120--1146.
\newblock \href {https://doi.org/10.1137/15M104952X}
  {\path{doi:10.1137/15M104952X}}.
\newline\urlprefix\url{https://doi.org/10.1137/15M104952X}

\bibitem{MR3421063}
E.~Camporeale, G.~L. Delzanno, B.~K. Bergen, J.~D. Moulton,
  \href{https://doi.org/10.1016/j.cpc.2015.09.002}{On the velocity space
  discretization for the {V}lasov-{P}oisson system: comparison between implicit
  {H}ermite spectral and particle-in-cell methods}, Comput. Phys. Commun. 198
  (2016) 47--58.
\newblock \href {https://doi.org/10.1016/j.cpc.2015.09.002}
  {\path{doi:10.1016/j.cpc.2015.09.002}}.
\newline\urlprefix\url{https://doi.org/10.1016/j.cpc.2015.09.002}

\bibitem{MR4456467}
A.~Gu, Y.~He, Y.~Sun,
  \href{https://doi.org/10.1016/j.jcp.2022.111472}{Hamiltonian particle-in-cell
  methods for {V}lasov-{P}oisson equations}, J. Comput. Phys. 467 (2022) Paper
  No. 111472, 20.
\newblock \href {https://doi.org/10.1016/j.jcp.2022.111472}
  {\path{doi:10.1016/j.jcp.2022.111472}}.
\newline\urlprefix\url{https://doi.org/10.1016/j.jcp.2022.111472}

\bibitem{MR1870837}
F.~Filbet, \href{https://doi.org/10.1137/S003614290037321X}{Convergence of a
  finite volume scheme for the {V}lasov-{P}oisson system}, SIAM J. Numer. Anal.
  39~(4) (2001) 1146--1169.
\newblock \href {https://doi.org/10.1137/S003614290037321X}
  {\path{doi:10.1137/S003614290037321X}}.
\newline\urlprefix\url{https://doi.org/10.1137/S003614290037321X}

\bibitem{FILBET2001166}
F.~Filbet, E.~Sonnendrücker, P.~Bertrand,
  \href{https://www.sciencedirect.com/science/article/pii/S0021999101968184}{Conservative
  numerical schemes for the vlasov equation}, Journal of Computational Physics
  172~(1) (2001) 166--187.
\newblock \href {https://doi.org/https://doi.org/10.1006/jcph.2001.6818}
  {\path{doi:https://doi.org/10.1006/jcph.2001.6818}}.
\newline\urlprefix\url{https://www.sciencedirect.com/science/article/pii/S0021999101968184}

\bibitem{5518448}
J.~W. Banks, J.~A.~F. Hittinger, A new class of nonlinear finite-volume methods
  for vlasov simulation, IEEE Transactions on Plasma Science 38~(9) (2010)
  2198--2207.
\newblock \href {https://doi.org/10.1109/TPS.2010.2056937}
  {\path{doi:10.1109/TPS.2010.2056937}}.

\bibitem{MR3854174}
G.~V. Vogman, U.~Shumlak, P.~Colella,
  \href{https://doi.org/10.1016/j.jcp.2018.07.029}{Conservative fourth-order
  finite-volume {V}lasov-{P}oisson solver for axisymmetric plasmas in
  cylindrical {$(r,v_r,v_\theta)$} phase space coordinates}, J. Comput. Phys.
  373 (2018) 877--899.
\newblock \href {https://doi.org/10.1016/j.jcp.2018.07.029}
  {\path{doi:10.1016/j.jcp.2018.07.029}}.
\newline\urlprefix\url{https://doi.org/10.1016/j.jcp.2018.07.029}

\bibitem{MR1672731}
E.~Sonnendr\"ucker, J.~Roche, P.~Bertrand, A.~Ghizzo,
  \href{https://doi.org/10.1006/jcph.1998.6148}{The semi-{L}agrangian method
  for the numerical resolution of the {V}lasov equation}, J. Comput. Phys.
  149~(2) (1999) 201--220.
\newblock \href {https://doi.org/10.1006/jcph.1998.6148}
  {\path{doi:10.1006/jcph.1998.6148}}.
\newline\urlprefix\url{https://doi.org/10.1006/jcph.1998.6148}

\bibitem{MR2843721}
J.-M. Qiu, C.-W. Shu,
  \href{https://doi.org/10.1016/j.jcp.2011.07.018}{Positivity preserving
  semi-{L}agrangian discontinuous {G}alerkin formulation: theoretical analysis
  and application to the {V}lasov-{P}oisson system}, J. Comput. Phys. 230~(23)
  (2011) 8386--8409.
\newblock \href {https://doi.org/10.1016/j.jcp.2011.07.018}
  {\path{doi:10.1016/j.jcp.2011.07.018}}.
\newline\urlprefix\url{https://doi.org/10.1016/j.jcp.2011.07.018}

\bibitem{MR2806222}
J.~A. Rossmanith, D.~C. Seal,
  \href{https://doi.org/10.1016/j.jcp.2011.04.018}{A positivity-preserving
  high-order semi-{L}agrangian discontinuous {G}alerkin scheme for the
  {V}lasov-{P}oisson equations}, J. Comput. Phys. 230~(16) (2011) 6203--6232.
\newblock \href {https://doi.org/10.1016/j.jcp.2011.04.018}
  {\path{doi:10.1016/j.jcp.2011.04.018}}.
\newline\urlprefix\url{https://doi.org/10.1016/j.jcp.2011.04.018}

\bibitem{KLIMAS1987202}
A.~J. Klimas,
  \href{https://www.sciencedirect.com/science/article/pii/0021999187900520}{A
  method for overcoming the velocity space filamentation problem in
  collisionless plasma model solutions}, Journal of Computational Physics
  68~(1) (1987) 202--226.
\newblock \href {https://doi.org/https://doi.org/10.1016/0021-9991(87)90052-0}
  {\path{doi:https://doi.org/10.1016/0021-9991(87)90052-0}}.
\newline\urlprefix\url{https://www.sciencedirect.com/science/article/pii/0021999187900520}

\bibitem{MR1259903}
A.~J. Klimas, W.~M. Farrell, \href{https://doi.org/10.1006/jcph.1994.1011}{A
  splitting algorithm for {V}lasov simulation with filamentation filtration},
  J. Comput. Phys. 110~(1) (1994) 150--163.
\newblock \href {https://doi.org/10.1006/jcph.1994.1011}
  {\path{doi:10.1006/jcph.1994.1011}}.
\newline\urlprefix\url{https://doi.org/10.1006/jcph.1994.1011}

\bibitem{MR1977366}
F.~Filbet, E.~Sonnendr\"ucker,
  \href{https://doi.org/10.1016/S0010-4655(02)00694-X}{Comparison of {E}ulerian
  {V}lasov solvers}, Comput. Phys. Comm. 150~(3) (2003) 247--266.
\newblock \href {https://doi.org/10.1016/S0010-4655(02)00694-X}
  {\path{doi:10.1016/S0010-4655(02)00694-X}}.
\newline\urlprefix\url{https://doi.org/10.1016/S0010-4655(02)00694-X}

\bibitem{ZAKI1988184}
S.~Zaki, L.~Gardner, T.~Boyd,
  \href{https://www.sciencedirect.com/science/article/pii/0021999188900101}{A
  finite element code for the simulation of one-dimensional vlasov plasmas. i.
  theory}, Journal of Computational Physics 79~(1) (1988) 184--199.
\newblock \href {https://doi.org/https://doi.org/10.1016/0021-9991(88)90010-1}
  {\path{doi:https://doi.org/10.1016/0021-9991(88)90010-1}}.
\newline\urlprefix\url{https://www.sciencedirect.com/science/article/pii/0021999188900101}

\bibitem{ZAKI1988200}
S.~Zaki, T.~Boyd, L.~Gardner,
  \href{https://www.sciencedirect.com/science/article/pii/0021999188900113}{A
  finite element code for the simulation of one-dimensional vlasov plasmas. ii.
  applications}, Journal of Computational Physics 79~(1) (1988) 200--208.
\newblock \href {https://doi.org/https://doi.org/10.1016/0021-9991(88)90011-3}
  {\path{doi:https://doi.org/10.1016/0021-9991(88)90011-3}}.
\newline\urlprefix\url{https://www.sciencedirect.com/science/article/pii/0021999188900113}

\bibitem{MR3011445}
M.~Nazarov, \href{https://doi.org/10.1016/j.camwa.2012.11.003}{Convergence of a
  residual based artificial viscosity finite element method}, Comput. Math.
  Appl. 65~(4) (2013) 616--626.
\newblock \href {https://doi.org/10.1016/j.camwa.2012.11.003}
  {\path{doi:10.1016/j.camwa.2012.11.003}}.
\newline\urlprefix\url{https://doi.org/10.1016/j.camwa.2012.11.003}

\bibitem{MR4456197}
T.~A. Dao, M.~Nazarov, \href{https://doi.org/10.1007/s10915-022-01918-4}{A
  high-order residual-based viscosity finite element method for the ideal {MHD}
  equations}, J. Sci. Comput. 92~(3) (2022) Paper No. 77, 24.
\newblock \href {https://doi.org/10.1007/s10915-022-01918-4}
  {\path{doi:10.1007/s10915-022-01918-4}}.
\newline\urlprefix\url{https://doi.org/10.1007/s10915-022-01918-4}

\bibitem{MR4496880}
L.~Lundgren, M.~Nazarov, \href{https://doi.org/10.1016/j.cam.2022.114846}{A
  high-order artificial compressibility method based on {T}aylor series
  time-stepping for variable density flow}, J. Comput. Appl. Math. 421 (2023)
  Paper No. 114846, 19.
\newblock \href {https://doi.org/10.1016/j.cam.2022.114846}
  {\path{doi:10.1016/j.cam.2022.114846}}.
\newline\urlprefix\url{https://doi.org/10.1016/j.cam.2022.114846}

\bibitem{barrata2023dolfinx}
I.~A. Barrata, J.~P. Dean, J.~S. Dokken, M.~Habera, J.~HALE, C.~Richardson,
  M.~E. Rognes, M.~W. Scroggs, N.~Sime, G.~N. Wells, Dolfinx: The next
  generation fenics problem solving environment (2023).

\bibitem{MR2427085}
J.-L. Guermond, R.~Pasquetti,
  \href{https://doi.org/10.1016/j.crma.2008.05.013}{Entropy-based nonlinear
  viscosity for {F}ourier approximations of conservation laws}, C. R. Math.
  Acad. Sci. Paris 346~(13-14) (2008) 801--806.
\newblock \href {https://doi.org/10.1016/j.crma.2008.05.013}
  {\path{doi:10.1016/j.crma.2008.05.013}}.
\newline\urlprefix\url{https://doi.org/10.1016/j.crma.2008.05.013}

\bibitem{GUERMOND20114248}
J.-L. Guermond, R.~Pasquetti, B.~Popov,
  \href{https://www.sciencedirect.com/science/article/pii/S0021999110006583}{Entropy
  viscosity method for nonlinear conservation laws}, Journal of Computational
  Physics 230~(11) (2011) 4248--4267, special issue High Order Methods for CFD
  Problems.
\newblock \href {https://doi.org/https://doi.org/10.1016/j.jcp.2010.11.043}
  {\path{doi:https://doi.org/10.1016/j.jcp.2010.11.043}}.
\newline\urlprefix\url{https://www.sciencedirect.com/science/article/pii/S0021999110006583}

\bibitem{GuermondNaPo11}
J.-L. Guermond, M.~Nazarov, B.~Popov, Implementation of the entropy viscosity
  method, Tech. Rep. 4015, KTH, Numerical Analysis, NA, qC 20110720 (2011).

\bibitem{Nazarov_Hoffman_2013}
M.~Nazarov, J.~Hoffman,
  \href{http://dx.doi.org/10.1002/fld.3663}{Residual-based artificial viscosity
  for simulation of turbulent compressible flow using adaptive finite element
  methods}, Internat. J. Numer. Methods Fluids 71~(3) (2013) 339--357.
\newblock \href {https://doi.org/10.1002/fld.3663}
  {\path{doi:10.1002/fld.3663}}.
\newline\urlprefix\url{http://dx.doi.org/10.1002/fld.3663}

\bibitem{MR3612753}
M.~Nazarov, A.~Larcher,
  \href{https://doi.org/10.1016/j.cma.2016.12.010}{Numerical investigation of a
  viscous regularization of the {E}uler equations by entropy viscosity},
  Comput. Methods Appl. Mech. Engrg. 317 (2017) 128--152.
\newblock \href {https://doi.org/10.1016/j.cma.2016.12.010}
  {\path{doi:10.1016/j.cma.2016.12.010}}.
\newline\urlprefix\url{https://doi.org/10.1016/j.cma.2016.12.010}

\bibitem{MR4754161}
T.~A. Dao, M.~Nazarov, \href{https://doi.org/10.1016/j.jcp.2024.113146}{A nodal
  based high order nonlinear stabilization for finite element approximation of
  magnetohydrodynamics}, J. Comput. Phys. 512 (2024) Paper No. 113146, 24.
\newblock \href {https://doi.org/10.1016/j.jcp.2024.113146}
  {\path{doi:10.1016/j.jcp.2024.113146}}.
\newline\urlprefix\url{https://doi.org/10.1016/j.jcp.2024.113146}

\bibitem{stiernstrom2021residual}
V.~Stiernstr\"om, L.~Lundgren, M.~Nazarov, K.~Mattsson,
  \href{https://doi.org/10.1016/j.jcp.2020.110100}{A residual-based artificial
  viscosity finite difference method for scalar conservation laws}, J. Comput.
  Phys. 430 (2021) Paper No. 110100, 29.
\newblock \href {https://doi.org/10.1016/j.jcp.2020.110100}
  {\path{doi:10.1016/j.jcp.2020.110100}}.
\newline\urlprefix\url{https://doi.org/10.1016/j.jcp.2020.110100}

\bibitem{Lundgren_2024}
L.~Lundgren, M.~Nazarov, \href{https://doi.org/10.1016/j.jcp.2023.112608}{A
  high-order residual-based viscosity finite element method for incompressible
  variable density flow}, J. Comput. Phys. 497 (2024) Paper No. 112608, 23.
\newblock \href {https://doi.org/10.1016/j.jcp.2023.112608}
  {\path{doi:10.1016/j.jcp.2023.112608}}.
\newline\urlprefix\url{https://doi.org/10.1016/j.jcp.2023.112608}

\bibitem{Kraaijevanger_1991}
J.~F. B.~M. Kraaijevanger,
  \href{https://doi.org/10.1007/BF01933264}{Contractivity of {R}unge-{K}utta
  methods}, BIT 31~(3) (1991) 482--528.
\newblock \href {https://doi.org/10.1007/BF01933264}
  {\path{doi:10.1007/BF01933264}}.
\newline\urlprefix\url{https://doi.org/10.1007/BF01933264}

\bibitem{MR4321466}
P.~Munch, K.~Kormann, M.~Kronbichler,
  \href{https://doi.org/10.1145/3469720}{hyper.deal: an efficient, matrix-free
  finite-element library for high-dimensional partial differential equations},
  ACM Trans. Math. Software 47~(4) (2021) Art. 33, 34.
\newblock \href {https://doi.org/10.1145/3469720} {\path{doi:10.1145/3469720}}.
\newline\urlprefix\url{https://doi.org/10.1145/3469720}

\bibitem{MR2016914}
N.~Besse, E.~Sonnendr\"ucker,
  \href{https://doi.org/10.1016/S0021-9991(03)00318-8}{Semi-{L}agrangian
  schemes for the {V}lasov equation on an unstructured mesh of phase space}, J.
  Comput. Phys. 191~(2) (2003) 341--376.
\newblock \href {https://doi.org/10.1016/S0021-9991(03)00318-8}
  {\path{doi:10.1016/S0021-9991(03)00318-8}}.
\newline\urlprefix\url{https://doi.org/10.1016/S0021-9991(03)00318-8}

\bibitem{Kronbichler_2024}
M.~Kronbichler, M.~Maier, I.~Tomas,
  \href{https://arxiv.org/abs/2402.04514}{Graph-based methods for hyperbolic
  systems of conservation laws using discontinuous space discretizations, part
  i: building blocks} (2024).
\newblock \href {http://arxiv.org/abs/2402.04514} {\path{arXiv:2402.04514}}.
\newline\urlprefix\url{https://arxiv.org/abs/2402.04514}

\bibitem{MR4126517}
M.~Mehrenberger, L.~Navoret, N.~Pham,
  \href{https://doi.org/10.4208/cicp.oa-2019-0022}{Recurrence phenomenon for
  {V}lasov-{P}oisson simulations on regular finite element mesh}, Commun.
  Comput. Phys. 28~(3) (2020) 877--901.
\newblock \href {https://doi.org/10.4208/cicp.oa-2019-0022}
  {\path{doi:10.4208/cicp.oa-2019-0022}}.
\newline\urlprefix\url{https://doi.org/10.4208/cicp.oa-2019-0022}

\bibitem{MR3267101}
Y.~Cheng, A.~J. Christlieb, X.~Zhong,
  \href{https://doi.org/10.1016/j.jcp.2014.08.041}{Energy-conserving
  discontinuous {G}alerkin methods for the {V}lasov-{M}axwell system}, J.
  Comput. Phys. 279 (2014) 145--173.
\newblock \href {https://doi.org/10.1016/j.jcp.2014.08.041}
  {\path{doi:10.1016/j.jcp.2014.08.041}}.
\newline\urlprefix\url{https://doi.org/10.1016/j.jcp.2014.08.041}

\bibitem{MR3148555}
M.~Ainsworth, \href{https://doi.org/10.1016/j.jcp.2013.11.003}{Dispersive
  behaviour of high order finite element schemes for the one-way wave
  equation}, J. Comput. Phys. 259 (2014) 1--10.
\newblock \href {https://doi.org/10.1016/j.jcp.2013.11.003}
  {\path{doi:10.1016/j.jcp.2013.11.003}}.
\newline\urlprefix\url{https://doi.org/10.1016/j.jcp.2013.11.003}

\bibitem{MR4402737}
M.~Campos~Pinto, K.~Kormann, E.~Sonnendr\"ucker,
  \href{https://doi.org/10.1007/s10915-022-01781-3}{Variational framework for
  structure-preserving electromagnetic particle-in-cell methods}, J. Sci.
  Comput. 91~(2) (2022) Paper No. 46, 39.
\newblock \href {https://doi.org/10.1007/s10915-022-01781-3}
  {\path{doi:10.1007/s10915-022-01781-3}}.
\newline\urlprefix\url{https://doi.org/10.1007/s10915-022-01781-3}

\bibitem{CROUSEILLES20091429}
N.~Crouseilles, G.~Latu, E.~Sonnendrücker,
  \href{https://www.sciencedirect.com/science/article/pii/S0021999108005652}{A
  parallel vlasov solver based on local cubic spline interpolation on patches},
  Journal of Computational Physics 228~(5) (2009) 1429--1446.
\newblock \href {https://doi.org/https://doi.org/10.1016/j.jcp.2008.10.041}
  {\path{doi:https://doi.org/10.1016/j.jcp.2008.10.041}}.
\newline\urlprefix\url{https://www.sciencedirect.com/science/article/pii/S0021999108005652}

\bibitem{MR4354369}
M.~Bessemoulin-Chatard, F.~Filbet,
  \href{https://doi.org/10.1016/j.jcp.2021.110881}{On the stability of
  conservative discontinuous {G}alerkin/{H}ermite spectral methods for the
  {V}lasov-{P}oisson system}, J. Comput. Phys. 451 (2022) Paper No. 110881, 28.
\newblock \href {https://doi.org/10.1016/j.jcp.2021.110881}
  {\path{doi:10.1016/j.jcp.2021.110881}}.
\newline\urlprefix\url{https://doi.org/10.1016/j.jcp.2021.110881}

\bibitem{MR4844786}
Y.~Kiechle, E.~Chudzik, C.~Helzel,
  \href{https://doi.org/10.1016/j.jcp.2024.113693}{A positivity-preserving
  {A}ctive {F}lux method for the {V}lasov-{P}oisson system}, J. Comput. Phys.
  524 (2025) Paper No. 113693.
\newblock \href {https://doi.org/10.1016/j.jcp.2024.113693}
  {\path{doi:10.1016/j.jcp.2024.113693}}.
\newline\urlprefix\url{https://doi.org/10.1016/j.jcp.2024.113693}

\bibitem{MR2586230}
N.~Crouseilles, M.~Mehrenberger, E.~Sonnendr\"ucker,
  \href{https://doi.org/10.1016/j.jcp.2009.11.007}{Conservative
  semi-{L}agrangian schemes for {V}lasov equations}, J. Comput. Phys. 229~(6)
  (2010) 1927--1953.
\newblock \href {https://doi.org/10.1016/j.jcp.2009.11.007}
  {\path{doi:10.1016/j.jcp.2009.11.007}}.
\newline\urlprefix\url{https://doi.org/10.1016/j.jcp.2009.11.007}

\bibitem{Umeda2008773}
T.~Umeda,
  \href{https://www.scopus.com/inward/record.uri?eid=2-s2.0-50849102460&doi=10.1186%2fBF03352826&partnerID=40&md5=7394da461c5a49dce291ed72f4a7bda4}{A
  conservative and non-oscillatory scheme for vlasov code simulations}, Earth,
  Planets and Space 60~(7) (2008) 773 – 779, cited by: 59; All Open Access,
  Hybrid Gold Open Access.
\newblock \href {https://doi.org/10.1186/BF03352826}
  {\path{doi:10.1186/BF03352826}}.
\newline\urlprefix\url{https://www.scopus.com/inward/record.uri?eid=2-s2.0-50849102460&doi=10.1186%2fBF03352826&partnerID=40&md5=7394da461c5a49dce291ed72f4a7bda4}

\bibitem{MR3738121}
X.~Cai, W.~Guo, J.-M. Qiu, \href{https://doi.org/10.1016/j.jcp.2017.10.048}{A
  high order semi-{L}agrangian discontinuous {G}alerkin method for
  {V}lasov-{P}oisson simulations without operator splitting}, J. Comput. Phys.
  354 (2018) 529--551.
\newblock \href {https://doi.org/10.1016/j.jcp.2017.10.048}
  {\path{doi:10.1016/j.jcp.2017.10.048}}.
\newline\urlprefix\url{https://doi.org/10.1016/j.jcp.2017.10.048}

\bibitem{MR4573597}
P.~Munch, T.~Heister, L.~Prieto~Saavedra, M.~Kronbichler,
  \href{https://doi.org/10.1145/3580314}{Efficient distributed matrix-free
  multigrid methods on locally refined meshes for {FEM} computations}, ACM
  Trans. Parallel Comput. 10~(1) (2023) Art. 3, 38.
\newblock \href {https://doi.org/10.1145/3580314} {\path{doi:10.1145/3580314}}.
\newline\urlprefix\url{https://doi.org/10.1145/3580314}

\bibitem{MR3517454}
M.~Mehrenberger, L.~S. Mendoza, C.~Prouveur, E.~Sonnendr\"ucker,
  \href{https://doi.org/10.1051/proc/201653010}{Solving the guiding-center
  model on a regular hexagonal mesh}, in: C{EMRACS} 2014---numerical modeling
  of plasmas, Vol.~53 of ESAIM Proc. Surveys, EDP Sci., Les Ulis, 2016, pp.
  149--176.
\newblock \href {https://doi.org/10.1051/proc/201653010}
  {\path{doi:10.1051/proc/201653010}}.
\newline\urlprefix\url{https://doi.org/10.1051/proc/201653010}

\end{thebibliography}

\end{document}